\documentclass{amsart}
\usepackage[utf8]{inputenc}
\usepackage{amsmath}
\usepackage{amsfonts}
\usepackage{amssymb}
\usepackage{float}
\usepackage{tikz}
\usetikzlibrary{intersections}
\usepackage[colorlinks=true]{hyperref}
\hypersetup{urlcolor=blue, citecolor=red}
\usepackage{hyperref}

\numberwithin{equation}{section}

\newtheorem{theorem}{Theorem}[section]
\newtheorem{corollary}[theorem]{Corollary}

\newtheorem{lemma}[theorem]{Lemma}
\newtheorem{proposition}[theorem]{Proposition}

\theoremstyle{definition}
\newtheorem{definition}[theorem]{Definition}
\newtheorem{remark}[theorem]{Remark}

\newcommand{\R}{\mathbb{R}}
\newcommand{\Rd}{\mathbb{R}^{d}}
\newcommand{\Rn}{\mathbb{R}^{d}}
\newcommand{\N}{\mathbb{N}}
\newcommand{\C}{\mathbb{C}}
\newcommand{\T}{\mathbb{T}}
\newcommand{\Z}{\mathbb{Z}}

\newcommand {\A}{\mathcal{A}}

\newcommand {\bmo}{\mathrm{bmo}}
\newcommand {\BMO}{{\mathrm{BMO}}}

\newcommand{\F}{\mathcal{F}}
\newcommand {\HT}{\mathcal{H}}
\newcommand {\Hp}{\mathcal{H}^{p}_{FIO}(\R^d)}

\newcommand {\Hps}{\mathcal{H}^{s,p}_{FIO}(\R^d)}

\newcommand {\Hpsm}{\mathcal{H}^{s-1,p}_{FIO}(\R^d)}

\newcommand {\ind}{{\mathbf{1}}}
\newcommand {\La}{\mathcal{L}}
\newcommand {\loc}{\mathrm{loc}}
\newcommand {\ud}{\mathrm{d}}

\newcommand {\veps}{\varepsilon}

\newcommand {\ka}{\kappa}
\newcommand {\la}{\lambda}
\newcommand {\ph}{\varphi}
\newcommand {\Sp}{S^{*}\Rd}
\newcommand {\Sw}{\mathcal{S}}
\newcommand {\supp}{\mathrm{supp}}
\newcommand {\Tp}{T^{*}\Rd}
\newcommand {\w}{\omega}
\newcommand {\wh}{\widehat}

\newcommand {\rb}{\rangle}
\newcommand {\lb}{\langle}
\newcommand {\cs}{\mathbf{c}}
\newcommand {\sn}{\mathbf{s}}

\newcommand {\vanish}[1]{\relax}

\DeclareFontFamily{U}{mathx}{}
\DeclareFontShape{U}{mathx}{m}{n}{<-> mathx10}{}
\DeclareSymbolFont{mathx}{U}{mathx}{m}{n}
\DeclareMathAccent{\widehat}{0}{mathx}{"70}
\DeclareMathAccent{\widecheck}{0}{mathx}{"71}

\title{Local smoothing for rough wave equations}

\author{Jan~Rozendaal}
\address[J.~Rozendaal]{Institute of Mathematics, Polish Academy of Sciences\\
\'{S}niadeckich 8\\
00-656 Warsaw\\
Poland}
\email{jrozendaal@impan.pl}

\author[Robert~Schippa]{Robert Schippa*}
\address[R.~Schippa]{University of Bonn, Mathematical Institute, Endenicher Allee 60, 53115 Bonn, Germany}
\email{rschippa@buni-bonn.de}


\makeatletter
\@namedef{subjclassname@2020}{%
  \textup{2020} Mathematics Subject Classification}
\makeatother


\thanks{*Corresponding author.}

\begin{document}

\begin{abstract}
We prove local smoothing estimates for wave equations with $C^{r}$ coefficients, relying on bilinear restriction estimates associated with rough wave propagation. In two and three dimensions, for $C^{1,1}$ coefficients we recover the sharp $L^p$-$L^q$ local smoothing bounds previously established for smooth coefficients. These are implied by stronger $L^{p}$-$L^{q}$ decoupling inequalities with endpoint derivative loss. As a corollary, we obtain non-trivial $L^{p}$-$L^{p}$ local smoothing estimates as well.
\end{abstract}
\maketitle

\section{Introduction}

\subsection{Wave equations with smooth coefficients}

Sharp fixed-time $L^p$ estimates for solutions to the Euclidean wave equation
\[
\begin{cases}
\partial_t^2 u= \Delta u,& (t,x) \in \R \times \R^d, \\
u(0)= u_0,& \partial_{t}u(0) = u_{1},
\end{cases}
\]
were obtained by Peral \cite{Peral1980} and Miyachi \cite{Miyachi1980}. For $1<p<\infty$ and $t\in\R$, they read
\begin{equation}\label{eq:fixedtimeintro}
\| u(t) \|_{L^p(\R^d)} \lesssim_t \| u_0 \|_{W^{2s(p),p}(\Rd)}+\|u_{1}\|_{W^{2s(p)-1,p}(\R^d)}.
\end{equation}
Here and throughout, we write $W^{\alpha,p}(\Rd):=(1-\Delta)^{-\alpha/2}L^{p}(\Rd)$ and 
\begin{equation}\label{eq:sp}
s(p):=\frac{d-1}{2}\Big|\frac{1}{2}-\frac{1}{p}\Big|,
\end{equation}
for $1\leq p\leq \infty$ and $\alpha\in\R$.

The phenomenon of local smoothing concerns improvements over the fixed-time bounds, obtained by averaging in time. Phrased in terms of the Euclidean half-wave propagators, Sogge's local smoothing conjecture from \cite{Sogge1991} states that
\begin{equation}
\label{eq:LocalSmoothingDiagonal}
\Big(\int_{0}^{1}\|e^{it \sqrt{-\Delta}} f\|_{L^{p}(\R^d)}^{p}\ud t\Big)^{1/p} \lesssim \| f \|_{W^{\alpha,p}(\R^d)}
\end{equation}
for $2<p<\infty$ and $\alpha > \max( 2s(p) - \frac{1}{p}, 0 \big)$. 
The critical index is $p_{c}:= \frac{2d}{d-1}$, in the sense that the conjecture is equivalent to the statement that \eqref{eq:LocalSmoothingDiagonal} holds for $p=p_{c}$ with $\alpha > 0$. A positive resolution of this conjecture would imply the Bochner--Riesz and restriction conjectures, making it a notoriously difficult problem which has only been solved for $d=2$ in \cite{GuthWangZhang2020}.

In \cite{SchlagSogge1997}, Schlag and Sogge considered more general local smoothing estimates:
\begin{equation}
\label{eq:LocalSmoothingOffDiagonal}
\Big(\int_{0}^{1}\|e^{it \sqrt{-\Delta}} f\|_{L^{q}(\R^d)}^{q}\ud t\Big)^{1/q}\lesssim \| f \|_{W^{\alpha,p}(\R^d)}
\end{equation}
for $2\leq p\leq q<\infty$. In this context, 
we note that \eqref{eq:LocalSmoothingOffDiagonal} follows for $1\leq p\leq q\leq \infty$ with $q'\geq p$ by interpolating the fixed-time estimates with the classical dispersive estimate, the latter corresponding to the case where $p=1$ and $q=\infty$. 
Clearly, $q\geq p$ is necessary for \eqref{eq:LocalSmoothingOffDiagonal} to hold, due to translation invariance. 

For the remaining values of $p$ and $q$, i.e.~those for which $2\leq p\leq q\leq \infty$ and $q'<p$, it is known that necessarily
\begin{equation}
\label{eq:NecessaryConditions}
\alpha \geq \max\big(\tfrac{d+1}{2} \big( \tfrac{1}{p} - \tfrac{1}{q} \big), \tfrac{d-1}{2} - \tfrac{d+1}{q} + \tfrac{1}{p}\big).
\end{equation}
The first condition on the right-hand side of \eqref{eq:NecessaryConditions} arises from the anisotropic Knapp example, that is, a single wave packet. This dominates in the upper triangle in Figure \ref{fig:LocalSmoothing}. The second condition in \eqref{eq:NecessaryConditions} is due to the focusing example - a radially symmetric bump function focusing at the origin. This dominates in the lower triangle in Figure \ref{fig:LocalSmoothing}. 

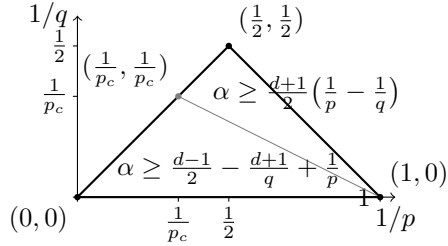
\begin{figure}[ht!]
\label{fig:LocalSmoothing}
\begin{tikzpicture}[scale=4]
  \draw[->] (0,0) -- (1.05,0) node[below] {$1/p$};
  \draw[->] (0,0) -- (0,0.6) node[left]  {$1/q$};

  \foreach \x in {0,1/3,1/2,1} {
    \draw (\x,0) -- (\x,-0.015);
  }
  \foreach \y in {0,1/3,1/2,1} {
    \draw (0,\y) -- (-0.015,\y);
  }

  \node[below] at (1/3,0) {$\frac{1}{p_{c}}$};
  \node[below] at (1/2,0) {$\frac12$};
  \node[left] at (1,0)   {$1$};

  \node[left]  at (0,1/3) {$\frac{1}{p_{c}}$};
  \node[left]  at (0,1/2) {$\frac12$};

  \coordinate (A) at (1/2,1/2);
  \coordinate (B) at (1/3,1/3);
  \coordinate (C) at (0,0);
  \coordinate (D) at (1,0);
  \coordinate (E) at (5/12,1/4);
  \coordinate (F) at (0.1,0);

  \draw[thick] (A) -- (B) -- (C) -- (D);

  \draw[thick] (A) -- (D);
   \draw[gray] (B) -- (D);

  \foreach \P in {A,C,D} {
    \fill (\P) circle (0.3pt);
  }
  \fill[gray]  (B) circle (0.3pt);
 
  \node[above right] at (A) {$(\tfrac12,\tfrac12)$};
  \node[above left] at (B) {$(\tfrac{1}{p_{c}},\tfrac{1}{p_{c}})$};
  \node[below left]  at (C) {$(0,0)$};
  \node[above right] at (D) {$(1,0)$};
  \node[above right] at (E) {$\alpha \geq \frac{d+1}{2}\big( \frac{1}{p} - \frac{1}{q} \big) $};
  \node[above right] at (F) {$\alpha \geq \frac{d-1}{2} - \frac{d+1}{q} + \frac{1}{p}$};
\end{tikzpicture}
\caption{$L^p$-$L^q$ local smoothing estimates. The conjectured endpoint for the local smoothing conjecture and the \emph{Schlag--Sogge line}, interpolating with the dispersive estimate, are colored in gray.}
\end{figure}

The index $p_{c}$ is again critical for \eqref{eq:LocalSmoothingOffDiagonal}, in the sense that proving the conjecture for $p=q=p_{c}$ would yield sharp bounds for all $p$ and $q$. Moreover, the estimate at $p=q=p_c$ fails for $\alpha= 0$, by an example of Wolff related to the Kakeya example.

Schlag and Sogge considered estimates on the boundary between the upper and lower triangles in Figure \ref{fig:LocalSmoothing}, i.e.~on the line between $\big( \frac{1}{p_{c}}, \frac{1}{p_{c}} \big)$ and $(1,0)$ in $[0,1]^{2}$, given by all $(\frac{1}{p},\frac{1}{q})$ such that $\frac{d+1}{q} = (d-1) \big( 1 - \frac{1}{p} \big)$ and $\frac{1}{q} \leq \frac{d-1}{2d}$. Here the conjecture arising from \eqref{eq:NecessaryConditions} is that \eqref{eq:LocalSmoothingOffDiagonal} holds for $\alpha > \frac{d}{q}-\frac{d-1}{2}$. Notably, the 
point $(\frac{1}{2},\frac{d-1}{2(d+1)})$ lies on this line and here \eqref{eq:LocalSmoothingOffDiagonal} corresponds to a Strichartz estimate. 

In \cite{TaoVargas2000} (see also \cite{TaoVargasVega1998}) Tao and Vargas showed that sharp versions of \eqref{eq:LocalSmoothingOffDiagonal} can be obtained for $2\leq p<q\leq \infty$ using bilinear techniques, thereby improving on the bounds in \cite{SchlagSogge1997}.  
They relied on the sharp bilinear restriction estimate for the cone due to Wolff \cite{Wolff2001} (see  \cite{Tao01} for the endpoint). In terms of half-wave propagators and with $q_{b}:= \frac{2(d+3)}{d+1}$ and $\alpha>d-\frac{2(d+1)   }{q_{b}}$, the latter result says that
\begin{equation*}
\Big(\int_{0}^{1}\| e^{it \sqrt{- \Delta}} f e^{it \sqrt{-\Delta}}f_2 \|_{L^{q_{b}/2}(\R^d)}^{q_{b}/2}\ud t\Big)^{2/q_{b}} \lesssim
N^{\alpha}
\| f\|_{L^{2}(\Rd)} \| g\|_{L^{2}(\Rd)}
\end{equation*}
for all $N\geq1$ and $f,g\in L^{2}(\Rd)$ such that $\supp(\wh{f}\,)$ and $\supp(\wh{g})$ are contained in $\{\xi\in\Rd\mid |\xi|\eqsim N\}$ and have angular separation approximately $1$. 

The estimates discussed so far concern the Euclidean wave equation. On the other hand, starting with the seminal work by Seeger, Sogge and Stein in \cite{SeSoSt91} on fixed-time bounds, it has been shown that many results for the Euclidean wave equation extend to the smooth variable-coefficient setting. In the context of this article, we mention the contribution \cite{Lee2006} of S.~Lee, which extends the Tao--Vargas argument.\cite[Corollary 1.5]{Lee2006} yields sharp bounds as in \cite{TaoVargas2000}, concerning \eqref{eq:LocalSmoothingOffDiagonal} but with $\Delta$ replaced by a variable-coefficient Laplacian with smooth coefficients. The proof relies on a bilinear estimate for oscillatory integral operators under transversality assumptions on the propagation of wave packets.

\subsection{Wave equations with rough coefficients}

In this article, we obtain $L^p$-$L^q$ local smoothing estimates for wave equations with $C^{r}$ coefficients, $r\geq 2$. 

We consider the initial value problem
\begin{equation}\label{eq:RoughWaveEquation}
\begin{gathered}
\partial_t^2 u(t,x) = \sum_{i,j=1}^{d}\partial_{x_{i}}(g_{ij}(x) \partial_{x_{j}} u(t,x)), \\
u(0,x) = u_0(x),\qquad\partial_{t}u(0,x) = u_1(x),
\end{gathered}
\end{equation}
for $t\in \R$ and $x=(x_{1},\ldots,x_{d})\in \R^d$. The coefficients $(g_{ij})_{i,j=1}^{d}$ are real-valued, bounded and uniformly elliptic, in the sense that there exists a $c>0$ such that
\[
\sum_{i,j=1}^{d}g_{ij}(x)\xi_{i}\xi_{j}\geq c|\xi|^{2}
\]
for all $x,\xi\in\Rd$. The $(g_{ij})_{i,j=1}^{d}$ are assumed to be time-independent for simplicity.  
Crucially, we do not suppose that the coefficients are smooth; we merely require that $(g_{ij})_{i,j=1}^{d}\subseteq C^{r}(\Rd)$ for some $r\geq 2$, with bounded $C^{r}(\Rd)$ norm. 

Set
\begin{equation}\label{eq:Q}
\mathfrak{Q}:=\Big(\frac{(d-1)(d+3)}{2(d^{2}+2d-1)},\frac{(d-1)(d+1)}{2(d^{2}+2d-1)}\Big)\in[0,1]^{2}.
\end{equation}
Let $\mathcal{T}_{l}$ be the closed triangle in $[0,1]^{2}$ with vertices $(0,0)$, $(1,0)$ and $\mathfrak{Q}$, and let $\mathcal{T}_{u}$ be the closed triangle in $[0,1]^{2}$ with vertices $(\frac{1}{2},\frac{1}{2})$, $(1,0)$ and $\mathfrak{Q}$. Write
\begin{equation}\label{eq:alphapq}
\alpha(p,q):=\begin{cases}
\frac{d-1}{2}-\frac{d+1}{q}+\frac{1}{p}&\text{if }(\frac{1}{p},\frac{1}{q})\in \mathcal{T}_{l},\\
\frac{d+1}{2}(\frac{1}{p}-\frac{1}{q})&\text{if }(\frac{1}{p},\frac{1}{q})\in \mathcal{T}_{u},
\end{cases}
\end{equation}
and note that $\alpha(p,q)=\max\big(\tfrac{d+1}{2} \big( \tfrac{1}{p} - \tfrac{1}{q} \big), \tfrac{d-1}{2} - \tfrac{d+1}{q} + \tfrac{1}{p}\big)$ as in \eqref{eq:NecessaryConditions}.

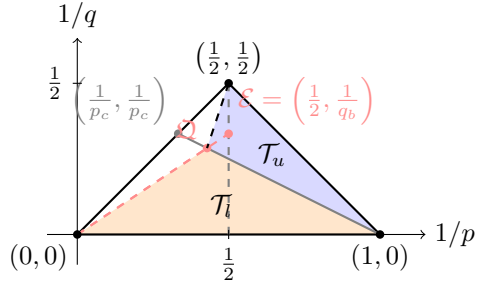
\begin{figure}[ht!]
\label{fig:roughsmoothing}
\begin{tikzpicture}[scale=4]
	
	\def\d{3} 
	
	\coordinate (P) at ({(\d-1)/(2*\d)},{(\d-1)/(2*\d)});
	\coordinate (Q) at ({1/2},{(\d+1)/(2*(\d+3))});
	\coordinate (R) at ({(\d-1)*(\d+3)/(2*(\d*\d+2*\d-1))},
	{(\d-1)*(\d+1)/(2*(\d*\d+2*\d-1))});
	
	\coordinate (A) at (0,0);
	\coordinate (B) at (0.5,0.5);
	\coordinate (C) at (1,0);
	
	\fill[orange!20] (A) -- (R) -- (C) -- cycle;
	\fill[blue!15] (R) -- (B) -- (C) -- cycle;
	
	\draw[->] (-0.1,0) -- (1.15,0) node[right] {$1/p$};
	\draw[->] (0,-0.1) -- (0,0.65) node[above] {$1/q$};
	
	\draw[thick] (A) -- (B) -- (C) -- cycle;
	
	\draw[gray, thick] (P) -- (C);
	
	\draw[red!45, dashed, thick] (A) -- (Q);
	\draw[red!45, dashed, thick] (A) -- (R);
	\draw[gray, dashed, thick] (B) -- (0.5,0);
	\node at (barycentric cs:A=1,R=1,C=1) {$\mathcal{T}_l$};
	\node at (barycentric cs:R=1,B=1,C=1) {$\mathcal{T}_u$};
	
	\fill (A) circle (0.4pt) node[below left] {$(0,0)$};
	\fill (B) circle (0.4pt) node[above] {$\left(\frac12,\frac12\right)$};
	\fill (C) circle (0.4pt) node[below] {$(1,0)$};
	
	\fill[gray] (P) circle (0.4pt)
	node[above left] {$\left(\frac{1}{p_{c}},\frac{1}{p_{c}}\right)$};
	
	\fill[red!45] (Q) circle (0.4pt)
	node[above right] {$\mathcal{E} = \left(\frac12,\frac{1}{q_{b}}\right)$};
	
	\fill[red!45] (R) circle (0.4pt)
	node[above left] {$\mathfrak{Q}$};
	\draw[dashed, thick] (B) -- (R);
	
	\draw (0.5,0.015) -- (0.5,-0.015) node[below] {$\frac12$};
	\draw (0.015,0.5) -- (-0.015,0.5) node[left] {$\frac12$};
	
\end{tikzpicture}
\caption{ For the rough evolution we obtain estimates in the shaded regions for $d=2,3$, $2 < p \leq \infty$ by bilinear arguments. Estimates for $1 \leq p \leq 2$ follow from interpolating with known bounds; see below.}
\end{figure}

Our main result, involving the exponent $s(p)$ from \eqref{eq:sp}, is as follows. 

\begin{theorem}\label{thm:mainLpintro}
Let $2<p<q<\infty$ be such that $(\frac{1}{p},\frac{1}{q})\in\mathcal{T}_{l}\cup\mathcal{T}_{u}$, and let $\alpha \geq \alpha(p,q)$. For $d\geq 4$, assume additionally that 
$r>\max(2s(p)+1,\alpha(p,q))$. Suppose that one of the following conditions holds:
\begin{enumerate}
\item\label{it:mainLpintro1} $\alpha>\alpha(p,q)$; 
\item\label{it:mainLpintro2} $(\frac{1}{p},\frac{1}{q})$ lies in the interior of $\mathcal{T}_{l}$. 
\end{enumerate} 
Then there exists a $C\geq0$ such that, for all $u_{0}\in W^{\alpha,p}(\R^d)$ and $u_{1}\in W^{\alpha-1,p}(\R^d)$, the solution $u$ to \eqref{eq:RoughWaveEquation} satisfies 
\begin{equation}\label{eq:mainLpintro}
\|u\|_{L^{q}([0,1]\times \R^d)}\leq C\big(\|u_{0}\|_{W^{\alpha,p}(\R^d)}+\|u_{1}\|_{W^{\alpha-1,p}(\R^d)}).
\end{equation}
\end{theorem}

For $r\in \N$, it suffices to assume in Theorem \ref{thm:mainLpintro} that $(g_{ij})_{i,j=1}^{d}\subseteq C^{r-1,1}(\Rd)$, and in this case one may allow $r>2s(p)+1$ with $r=\alpha(p,q)$ (see Remark \ref{rem:sqinfty}). In the special case where $r=2$, the condition $2s(p)+1<r$ leads to the interval $2<p<\frac{2(d-1)}{d-3}$. For $r>\frac{d+1}{2}$, one has $r>\max(2s(p)+1,\alpha(p,q))$ for all $2<p<q<\infty$.

In \cite{HassellRozendaal2023}, the sharp fixed-time estimates in \eqref{eq:fixedtimeintro} and \cite{SeSoSt91} were extended to coefficients as above, under the same condition $r>2s(p)+1$ on $1<p<\infty$ as in Theorem \ref{thm:mainLpintro} (one then automatically has $r>\alpha(p,p)$). See also \cite{LiRoSo25a} for the endpoint cases $p=1$ and $p=\infty$. In particular, combined with standard Sobolev embeddings, those results show that a unique solution
\[
u\in L^{\infty}([0,1];W^{\alpha-2s(p)-d(\frac{1}{p}-\frac{1}{q}),q}(\Rd))
\]
to \eqref{eq:RoughWaveEquation} exists under the assumptions of Theorem \ref{thm:mainLpintro}. The latter result essentially yields additional spatial regularity of order $\frac{1}{q}$ for $(\frac{1}{p},\frac{1}{q})\in \mathcal{T}_{l}$, and of order $s(p)+s(q)$ for $(\frac{1}{p},\frac{1}{q})\in \mathcal{T}_{u}$, at the cost of decreased time regularity.

For $1\leq p\leq 2$, sharp estimates as in \eqref{eq:mainLpintro} follow from existing results.
Indeed, loosely speaking, one can extend the sharp bounds from the  smooth setting by 
interpolating between the fixed-time estimates from \cite{HassellRozendaal2023,LiRoSo25a}, the dispersive estimate from \cite{Geba-Tataru05}, the Strichartz estimate from \cite{Tataru2001} for $p=2$ and $q=\frac{2(d+1)}{d-1}$, and the trivial bound for $p=2$ and $q=\infty$, which follows from a Sobolev embedding. 
Moreover, if $q=\infty$ then the fixed-time estimates and Sobolev embeddings also yield sharp bounds for $2<p\leq \infty$. Theorem \ref{thm:mainLpintro} further enlarges the set of $(\frac{1}{p},\frac{1}{q})\in[0,1]^{2}$ for which the bounds from the smooth setting extend to rough coefficients, with only the non-shaded region below the diagonal in Figure \ref{fig:roughsmoothing} now remaining open.

In fact, as is shown in Corollary \ref{cor:extrabutweak}, by combining finite speed of propagation and H\"{o}lder's inequality, one can obtain from Theorem \ref{thm:mainLpintro} nontrivial local smoothing for all $(\frac{1}{p},\frac{1}{q})$ with $2<p\leq q<\infty$, and in particular for $p=q$. For simplicity, here we only state the result in the special case where $p=q=p_{c}=\frac{2d}{d-1}$, in which case the coefficients are merely required to have $C^{1,1}(\Rd)$ regularity.

\begin{corollary}\label{cor:extrabutweakintro}
Let $r=2$, $p=\frac{2d}{d-1}$ and $\alpha>2s(p)-\frac{1}{p^{2}}$. Then there exists a $C\geq0$ such that, for all $u_{0}\in W^{\alpha,p}(\R^d)$ and $u_{1}\in W^{\alpha-1,p}(\R^d)$, the solution $u$ to \eqref{eq:RoughWaveEquation} satisfies 
\[
\|u\|_{L^{p}([0,1]\times \R^d)}\leq C\big(\|u_{0}\|_{W^{\alpha,p}(\R^d)}+\|u_{1}\|_{W^{\alpha-1,p}(\R^d)}).
\]
\end{corollary}

\subsection{Hardy spaces for Fourier integral operators}

Theorem \ref{thm:mainLpintro} falls into a line of research concerning rough wave equations that goes back at least to Smith's work in \cite{Smith1998}. He obtained Strichartz estimates for wave equations with $C^{1,1}$ coefficients in dimensions $d=2,3$, by relying on a combination of paradifferential calculus, wave packet transforms and Hamiltonian flows. This approach was later used in modified forms to derive properties of rough wave equations, such as Strichartz estimates in all dimensions \cite{Tataru00,Tataru2001,Tataru02}, spectral cluster estimates \cite{Smith06} and propagation of singularities \cite{Smith14}. Moreover, the assumption of $C^{1,1}$ regularity is typically sharp in these results, cf.~\cite{SmithSogge1994,SmithTataru2002}.

A common thread in these works is the use of a paradifferential smoothing procedure, to split a rough differential operator into a sum of a smooth pseudodifferential operator, and a rough pseudodifferential operator of lower differential order. The latter can then be treated as an inhomogeneous term in the equation. Although the smooth pseudodifferential operator is not homogeneous and therefore not directly amenable to the theory of Fourier integral operators, which constitutes a key tool in the smooth setting, one can still construct a parametrix for the smooth part of the equation using wave packet transforms.

However, in the results mentioned above the relevant parametrix is $L^{2}$ based, and a problem arises when dealing with initial data in Sobolev spaces over $L^{p}$. Namely, due to Duhamel's principle, when treating the rough lower-order part of the relevant operator as an inhomogeneous term, the solution operator associated with the smooth equation is applied to this rough term. Already in case of the Euclidean wave equation, this inevitably leads to an additional loss of derivatives outside of $L^{2}$, cf.~\eqref{eq:fixedtimeintro}. As a result, in $L^{p}(\Rd)$ for $p\neq 2$ one cannot iterate to get rid of the error term and close the loop.

In \cite{HassellRozendaal2023,LiRoSo25a}, this issue was resolved by working with a different space of initial data, relying on another fundamental contribution due to Smith. In \cite{Smith98a}, he introduced a space $\HT^{1}_{FIO}(\Rd)$ which is invariant under Fourier integral operators and which satisfies Sobolev embeddings that allow one to recover the fixed-time results from \cite{SeSoSt91} for smooth wave equations. In \cite{HassellPortalRozendaal2020}, Smith's construction was extended to a scale $(\Hp)_{1\leq p\leq \infty}$ of Hardy spaces for Fourier integral operators, invariant under Fourier integral operators and satisfying the embeddings
\begin{equation}\label{eq:Sobolevintro}
W^{s(p),p}(\Rd)\subseteq \Hp\subseteq W^{-s(p),p}(\Rd)
\end{equation}
for $1<p<\infty$, with the natural modifications involving the local Hardy space $\HT^{1}(\Rd)$ and $\bmo(\Rd)$ for $p=1$ and $p=\infty$, respectively. 
This invariance under the solution operators to smooth wave equations is the crucial tool which allows one to use an iterative procedure on $\Hp$ to deal with the Duhamel term arising from a rough wave equation, after which \eqref{eq:Sobolevintro} relates the results to $W^{s,p}(\Rd)$.

In this article we follow a similar approach as in \cite{HassellRozendaal2023,LiRoSo25a}, by proving the following theorem involving the Sobolev spaces $\HT^{s,p}_{FIO}(\Rd):=(1-\Delta)^{-s/2}\Hp$ over the Hardy spaces for Fourier integral operators from Definition \ref{def:HpFIO}. Once combined with \eqref{eq:Sobolevintro}, this result directly yields Theorem \ref{thm:mainLpintro}.

\begin{theorem}\label{thm:mainHpFIOintro}
Under the conditions of Theorem \ref{thm:mainLpintro}, there exists a $C\geq0$ such that, for all $u_{0}\in \HT^{\alpha-s(p),p}_{FIO}(\R^d)$ and $u_{1}\in \HT^{\alpha-s(p)-1,p}_{FIO}(\R^d)$, the solution $u$ to \eqref{eq:RoughWaveEquation} satisfies 
\[
\|u\|_{L^{q}([0,1]\times \R^d)}\leq C\big(\|u_{0}\|_{\HT^{\alpha-s(p),p}_{FIO}(\R^d)}+\|u_{1}\|_{\HT^{\alpha-s(p)-1,p}_{FIO}(\R^d)}).
\]
\end{theorem}

\subsection{Overview of the proof}\label{subsec:proof}
\subsubsection{Paradifferential decomposition}
To prove Theorem \ref{thm:mainHpFIOintro}, we first apply the paradifferential smoothing procedure to reduce to the half-wave equation associated with a smooth pseudodifferential operator, for frequency-localized initial data. This step is where the invariance of the Hardy spaces for Fourier integral operators comes into play, and it is the reason for the restriction $r>\max(2s(p)+1,\alpha(p,q))$ in Theorems \ref{thm:mainLpintro} and \ref{thm:mainHpFIOintro}, which is only relevant for $d\geq 4$.
To absorb the rough part of the equation into the Duhamel term, we have to deal with rough pseudodifferential operators on the Hardy spaces for Fourier integral operators. Their mapping properties were considered in \cite{Rozendaal23a,Rozendaal22,LiRoSo25a}; the relevant result for this article is Lemma \ref{lem:pseudoHpFIO} (see also Remark \ref{rem:sqinfty}).

Once one is allowed to work with a smooth evolution $S_{N}$ 
and with initial data frequency-localized in an annulus of height $N>0$, the proof globally follows the argument pioneered in 
the Euclidean setting, cf.~\cite{TaoVargas2000}, albeit with several twists. 
\subsubsection{Whitney decomposition and almost orthogonality}
An angular Whitney decomposition from \cite{TaoVargas2000}  reduces the problem to obtaining estimates for suitable bilinear operators $B_{\nu}$, at dyadic scales 
$N^{-1/2}\leq \nu\leq 1$, the components of which involve frequency localization to cones of aperture $\nu$ that are separated by angle $\nu$.  

An almost orthogonality statement shows that it suffices to bound  
terms of the form $\|(S_{N}f)(S_{N}g)\|_{L^{q/2}([0,1]\times\Rd)}$, under the same frequency support conditions. 
In the constant-coefficient case, due to the preservation of Fourier support, this statement is straightforward. Here we require an additional spatial localization, after which the almost orthogonality can be inferred from kernel bounds.

\subsubsection{Bilinear estimates in $\mathcal{H}_{FIO}^p$-spaces}
We rely on two ingredients to bound the terms $\|(S_{N}f)(S_{N}g)\|_{L^{q/2}([0,1]\times\Rd)}$. One is the extension by the second author and Tataru 
of Wolff's bilinear restriction estimate, in \cite{SchippaTataru2025}. This result yields, for $q=q_{b}=\frac{2(d+3)}{d+1}$, an essentially sharp bound in terms of the $L^{2}(\Rd)$ norms of $f$ and $g$. 
While the linear version of this estimate fails even for the Euclidean wave equation, angular separation of the Fourier supports yields cancellation and improved integrability. In fact, the bilinear estimate in \cite{SchippaTataru2025} concerns evolutions associated with homogeneous symbols, and  
we need to modify some of the arguments. An $N^{\veps}$ loss occurs here, which we show can be removed for $q>q_{b}$. Crucially, the resulting expression is not summable in $\nu$. 

The second ingredient is an estimate for $q=\infty$ in terms of the $\HT^{\infty}_{FIO}(\Rd)$ norms of $f$ and $g$, 
relying only on the fact that these functions have Fourier support in a cone of aperture $\nu\geq N^{-1/2}$. Here we again use 
kernel bounds from \cite{GebaTataru2007}. This estimate is summable in $\nu$ but too weak in terms of the dependence on $N$.

For $(\frac{1}{p},\frac{1}{q})$ on the line between $(\frac{1}{2},\frac{1}{q_{b}})$ and $(0,0)$, we can interpolate these two ingredients to obtain a bound 
in terms of the $\Hp$ norms of $f$ and $g$. This line intersects the Schlag--Sogge line at the point $\mathfrak{Q}$, 
below which the expression becomes summable in $\nu$ and sharp in $N$. To the right of the first line one can remove the $N^{\veps}$ loss, and it is then straightforward to combine the pieces 
and prove Theorem \ref{thm:mainHpFIOintro}.

\subsection{Fourier decoupling}

Some of our refined decoupling results appear to be new even for 
the Euclidean wave equation.

In \cite{BourgainDemeter2015}, Bourgain and Demeter obtained essentially sharp 
$\ell^{2}$ decoupling bounds for the paraboloid.
Via the Pramanik--Seeger argument from \cite{Pramanik-Seeger07}, their $\ell^{2}$ decoupling estimates imply $\ell^{p}L^{p}$-$L^{p}$ decoupling bounds for the Euclidean wave equation, which in turn yield $W^{\alpha,p}(\Rd)$-$L^{p}([0,1]\times\Rd)$ local smoothing. 

It was observed in \cite{Rozendaal22b} that $\ell^{p}L^{p}$-$L^{p}$ decoupling bounds for the Euclidean wave equation are in fact equivalent to $\HT^{\alpha,p}_{FIO}(\Rd)$-$L^{p}([0,1]\times\Rd)$ local smoothing estimates. This equivalence was extended in \cite{LiRoSoYa24} to the variable-coefficient $\ell^{p}$ decoupling inequalities from \cite{BeltranHickmanSogge2020}, and similar connections hold between related function spaces and other decoupling inequalities, cf.~\cite{RozendaalSchippa2023,HaPoRoYu26,Schippa22}. 

On the other hand, the Tao--Vargas approach to $W^{\alpha,p}(\Rd)$-$L^{q}([0,1]\times\Rd)$ local smoothing upon which we rely does not appear to connect to decoupling theory. Moreover, there are also $W^{\alpha,p}(\Rd)$-$L^{p}([0,1]\times\Rd)$ local smoothing estimates which do not follow directly from bounds involving $\Hp$, most notably at the critical index $p=p_{c}$ (see e.g.~\cite[Section 5]{Rozendaal22b}). This is particularly relevant because, to implement the strategy described above of dealing with the rough Duhamel term, we require bounds for the paradifferentially smoothed evolution involving $\HT^{\alpha-s(p),p}_{FIO}(\Rd)$ initial data; it would not suffice to obtain the weaker $W^{\alpha,p}(\Rd)$-$L^{q}([0,1]\times\Rd)$ estimates in Theorem \ref{thm:mainLpintro} for the smooth equation.

Fortunately, as is evident from Theorem \ref{thm:mainHpFIOintro}, one can indeed derive $\HT^{\alpha-s(p),p}_{FIO}(\Rd)$-$L^{q}([0,1]\times\Rd)$ local smoothing estimates which improve upon the classical bounds from \cite{TaoVargas2000,Lee2006} involving $W^{\alpha,p}(\Rd)$ initial data. This is in turn possible because the bilinear restriction estimate with initial data in $W^{\alpha,p}(\Rd)$ can be upgraded to a bilinear bound involving $\HT^{\alpha-s(p),p}_{FIO}(\Rd)$ initial data (see Theorem \ref{thm:BilinearEstimates}). And, cf.~the proof sketch above, to do so one merely needs to rely on the improved fixed-time bounds for $p=\infty$. In fact, another property of the Hardy spaces for Fourier integral operators plays a hidden role here; the behavior of the $\Hp$ norm of the Knapp example is sufficiently beneficial for $p>2$ that it allows for the kind of improvements we are after (see Lemma \ref{lem:Knapp}). Even in the case of the Euclidean wave equation, similar results cannot hold for $q<\frac{2(d+1)}{d-1}$ when relying on an $\ell^{2}$ decoupling inequality, again due to the Knapp example (see Remark \ref{rem:ell2decoupling}).

Finally, the arguments from \cite{Rozendaal22b,LiRoSoYa24} show that the bounds in Theorem \ref{thm:mainHpFIOintro} are equivalent in a quantitative sense to $\ell^{p}L^{p}$-$L^{q}$ decoupling inequalities. Here it is relevant to note that, for $(\frac{1}{p},\frac{1}{q})$ in the interior of $\mathcal{T}_{l}$, the bounds are sharp up to the endpoint index, whereas decoupling estimates typically involve at least some loss. 

\subsection{Organization}

In Section \ref{sec:Preliminaries} we introduce the relevant function spaces, and we collect some basics on rough symbol classes and pseudodifferential operators. 

With these in hand, in Section \ref{sec:BilinearSmoothing} we prove Theorem \ref{thm:mainHpFIOintro}. 
In Section \ref{subsec:main} we first state a slightly more precise version of Theorem \ref{thm:mainHpFIOintro}, and we prove Corollary \ref{cor:extrabutweakintro}. Section \ref{subsec:reduction} then contains various reduction steps, most notably the paradifferential smoothing. Section \ref{subsec:wavedecomp} extends 
tools from \cite{SchippaTataru2025} to the present setting, and Section \ref{subsec:bilinear} is dedicated to our bilinear estimates. Section \ref{subsec:orthogonal} details the angular Whitney decomposition and contains the proof of the almost orthogonal decomposition. 
In Section \ref{subsec:conclude} these ingredients are combined to conclude the proof of Theorem \ref{thm:mainHpFIOintro}. 

Appendix \ref{sec:kernel} collects kernel bounds that play a key role at several points in the main proof.  Appendix \ref{sec:inter} contains two interpolation results, used to deal with the endpoint cases in Theorems \ref{thm:mainLpintro} and \ref{thm:mainHpFIOintro}. Appendix \ref{sec:noeps} contains the $\veps$ removal argument for the $L^{2}$-based bilinear restriction estimate from \cite{SchippaTataru2025}, and Appendix \ref{sec:smooth} concerns improvements of Theorem \ref{thm:mainHpFIOintro} for the Euclidean wave equation.

\subsection{Notation}
The natural numbers are $\N=\{1,2,\ldots\}$, and $\N_{0}:=\N\cup\{0\}$. 
For $\mathbb{X} \in \{ \N, \N_0,-\N\}$
, we set $2^{\mathbb{X}}:=\{2^{k}\mid k\in\mathbb{X}\}$. For $x \in \R$, we write $x_+ := \max(x,0)$. Throughout this article, we fix $d\in\N$ with $d\geq2$. 

Space-time variables are denoted by $(t,x) \in \R\times \R^d$, and the dual variables are $(\tau,\xi) \in \R\times \R^d$. The open ball in $\Rd$ of radius $r>0$ around $x\in\Rd$ is denoted by $B_{d}(x,r)$, and $S^{d-1}$ is the unit sphere in $\Rd$. The cotangent bundle of $\Rd$ is $T^{*}\Rd=\R^{d}\times\Rd$, and the cosphere bundle is $\Sp=\Rd\times S^{d-1}$.

For $\xi\in\Rd$ we write $\lb\xi\rb:=(1+|\xi|^{2})^{1/2}$, and $\hat{\xi}:=|\xi|^{-1}\xi$ if $\xi\neq 0$. We use multi-index notation, where $\partial_{\xi}=(\partial_{\xi_{1}},\ldots,\partial_{\xi_{d}})$ and 
$\partial^{\alpha}_{\xi}=\partial^{\alpha_{1}}_{\xi_{1}}\ldots\partial^{\alpha_{d}}_{\xi_{d}}$
for $\xi=(\xi_{1},\ldots,\xi_{d})\in\Rd$ and $\alpha=(\alpha_{1},\ldots,\alpha_{d})\in\N_{0}^{d}$. We write $D_j :=- i \partial_{x_j}$ for $1\leq j\leq d$, and $D_t := - i \partial_t$. Moreover, $e_{d}:=(0,\ldots,0,1)$ is the $d$-th basis vector of $\Rn$.

The Fourier transform of a tempered distribution $f\in\Sw'(\Rd)$ is denoted by $\F f$ or $\wh{f}$, and its inverse Fourier transform by $\F^{-1}f$ or 
$\widecheck{f}$. If $f\in L^{1}(\Rd)$ then $\F f(\xi)=\int_{\Rd}e^{-i\lb x,\xi\rb}f(x)\ud x$ for $\xi\in\Rd$. Also, $\ph(D)$ is the Fourier multiplier with symbol $\ph$.

The H\"{o}lder conjugate of $p\in[1,\infty]$ is $p'$, and the indicator function of a subset $\Omega$ of a set $\Omega'$ is denoted by $\ind_{\Omega}$. The space of bounded operators between Banach spaces $X$ and $Y$ is $\La(X,Y)$.

We write $f(s)\lesssim g(s)$ to indicate that $f(s)\leq Cg(s)$ for all $s$ and a constant $C>0$ independent of $s$, and $f(s)\lesssim_{\alpha} g(s)$ means that $C$ depends on a parameter $\alpha$. The same convention applies to $f(s)\gtrsim g(s)$ and $g(s)\eqsim f(s)$.

\section{Preliminaries}\label{sec:Preliminaries}
In this section we collect background on the relevant function spaces and classes of pseudodifferential symbols, as will be needed in later sections.

\subsection{Function spaces}\label{subsec:spaces}

Let 
$\psi_1 \in C^\infty_c(\R^d)$ 
be such that $\psi_1 \equiv 1$ on $B_d(0,1)$ and $\psi_1 \equiv 0$ outside $B_d(0,2)$. For $N\in 2^{\N}$ and $\xi\in\Rd$, set $\psi_N(\xi) := \psi_1(\xi/N) - \psi_1(\xi/(N/2))$. Then $\psi_{N}(\xi)=0$ unless 
$N/2<|\xi|<2N$. We have for all $\xi\in\Rd$
\begin{equation}\label{eq:LittlePaley}
\sum_{N \in 2^{\N_0}} \psi_N(\xi) =1.
\end{equation}
For brevity we let $\psi = \psi_1$.

Recall that, for $p,q\in[1,\infty]$ and $s\in\R$, the (inhomogeneous) Besov space $B^{s}_{p,q}(\Rd)$ consists of all $f\in\Sw'(\Rd)$ such that 
\begin{equation*}
    \| f \|_{B^s_{p,q}(\R^d)}:= \Big(\sum_{N\in 2^{\N_0}} N^{sq} \|\psi_{N}(D) f \|_{L^p(\R^d)}^q  \Big)^{1/q}<\infty,
\end{equation*}
with the usual modification for $q = \infty$. Moreover, we write $C^{s}_{*}(\Rd):=B^{s}_{\infty,\infty}(\Rd)$. 

Next, $H^1(\R^d)$ denotes the classical Hardy space, comprised of integrable functions with mean zero which admit the well-known atomic decomposition. Its dual is $\BMO(\Rd)$. Let $\mathcal{H}^1(\R^d)$ be the local Hardy space from \cite{Goldberg1979}, consisting of $f \in \mathcal{S}'(\R^d)$ such that $\psi(D)f \in L^1(\R^d)$ and $(1-\psi(D))f \in H^1(\R^d)$, endowed with the norm
\begin{equation*}
    \| f \|_{\mathcal{H}^1(\R^d)} := \|\psi(D) f \|_{L^1(\R^d)} + \|(1-\psi(D))f \|_{H^1(\R^d)}.
\end{equation*}
In turn, the dual of $\HT^{1}(\Rn)$ is the space $\bmo(\Rd)$, which consists of all $f \in \mathcal{S}'(\R^d)$ such that $\psi(D)f \in L^{\infty}(\R^d)$ and $(1-\psi(D))f \in \BMO(\R^d)$, with the norm
\begin{equation*}
    \| f \|_{\bmo(\R^d)} := \| \psi(D)f \|_{L^\infty(\R^d)} + \| (1-\psi(D))f \|_{\BMO(\R^d)}.
\end{equation*}
Throughout, for notational convenience, we will write \begin{equation}\label{eq:localHardy}
\HT^{s,p}(\Rn):=
\begin{cases}
W^{s,p}(\Rn)=\lb D\rb^{-s}L^{p}(\Rn)&\text{for }p\in(1,\infty),\\
\lb D\rb^{-s}\HT^{1}(\Rn)&\text{for }p=1,\\
\lb D\rb^{-s}\bmo(\Rn)&\text{for }p=\infty,
\end{cases}
\end{equation}
for $s\in\R$. Here $\langle D \rangle^s$ is the Bessel potential of order $s\in\R$, which is defined as a Fourier multiplier by $\F(\langle D \rangle^s f)(\xi) = \langle \xi \rangle^s \F f (\xi)$ for $\xi\in\Rd$.

To relate our main result in Section \ref{subsec:main} to Theorem \ref{thm:mainLpintro}, and to interpret some more subtle aspects of our results, e.g.~in Remark \ref{rem:sqinfty}, it is relevant to note that $\HT^{s,\infty}(\Rd)\subsetneq C^{s}_{*}(\Rd)$ for all $s>0$, that $C^{s}(\Rd)\subsetneq C^{s-1,1}(\Rd)\subsetneq \HT^{s,\infty}(\Rd)$ for $s\in\N$, and that $C^{s}_{*}(\Rd)=C^{s}(\Rd)$ for $s\in (0,\infty)\setminus \N$.

Next, fix a family $(\ph_{\w})_{\w\in S^{d-1}}\subseteq C^{\infty}(\Rn)$ of non-negative functions with the following properties:
\begin{enumerate}
\item\label{it:phiproperties1} For all $\w\in S^{d-1}$ and $\xi\neq0$, it holds $\ph_{\w}(\xi)=0$ if $|\xi|<\frac{1}{4}$ or $|\hat{\xi}-\w|>|\xi|^{-1/2}$.
\item\label{it:phiproperties2} For all $\alpha\in\N_{0}^{d}$ and $\beta\in\N_{0}$, there exists a $C_{\alpha,\beta}\geq0$ such that, for all $\w\in S^{d-1}$ and $\xi\neq0$, one has  
$|\lb \w,\partial_{\xi}\rb^{\beta}\partial^{\alpha}_{\xi}\ph_{\w}(\xi)|\leq C_{\alpha,\beta}|\xi|^{\frac{d-1}{4}-\frac{|\alpha|}{2}-\beta}$.
\item\label{it:phiproperties3}
The map $(\xi,\w)\mapsto \ph_{\w}(\xi)$ is measurable on $\Rn\times S^{d-1}$, and $\int_{S^{d-1}}\ph_{\w}(\xi)^{2}\ud\w=1$ for all $\xi\in\Rn$ with $|\xi|\geq 1$.
\end{enumerate}

We can now define the Hardy spaces for Fourier integral operators.

\begin{definition}\label{def:HpFIO}
Let $p\in[1,\infty)$. Then $\Hp$ consists of those $f\in\Sw'(\Rn)$ such that $\psi(D/2)f\in L^{p}(\Rd)$, $\ph_{\w}(D)f\in \HT^{p}(\Rd)$ for almost all $\w\in S^{d-1}$, and $(\int_{S^{d-1}}\|\ph_{\w}(D)f\|_{\HT^{p}(\Rd)}^{p}\ud \w)^{1/p}<\infty$, endowed with the norm
\[
\|f\|_{\Hp}:=\|\psi(D/2)f\|_{L^{p}(\Rd)}+\Big(\int_{S^{d-1}}\|\ph_{\w}(D)f\|_{\HT^{p}(\Rd)}^{p}\ud \w\Big)^{1/p}.
\] 
Moreover, $\HT^{\infty}_{FIO}(\Rd):=\HT^{1}_{FIO}(\Rd)^{*}$, and $\HT^{s,p}_{FIO}(\Rd):=\lb D\rb^{-s}\Hp$ for all $p\in[1,\infty]$ and $s\in\R$.
\end{definition}

We recall some key properties of these spaces. Firstly,
\begin{equation}\label{eq:Sobolev}
\mathcal{H}^{s+s(p),p}(\R^d) \hookrightarrow \mathcal{H}^{s,p}_{\text{FIO}}(\R^d) \hookrightarrow \mathcal{H}^{s-s(p),p}(\R^d)
\end{equation}
for all $p\in[1,\infty]$ and $s\in\R$, by \cite[Theorem~7.4]{HassellPortalRozendaal2020}.  The Schwartz functions are contained in $\Hps$, and they are dense if $p<\infty$, by \cite[Proposition 6.6]{HassellPortalRozendaal2020}. The spaces form a complex interpolation scale in the natural manner, by \cite[Proposition 6.7]{HassellPortalRozendaal2020},  
and they behave as expected under duality:
$(\Hps)^{*}=\HT^{-s,p'}_{FIO}(\Rn)$ 
for all $p\in[1,\infty)$ and $s\in\R$, cf.~\cite[Proposition 6.8]{HassellPortalRozendaal2020}.

Finally, an important role will be played by the following result from \cite[Proposition 6.4]{FaLiRoSo23} (see also \cite[Lemma 2.2]{Rozendaal22b}) on the $\Hp$ norm of the Knapp example.

\begin{lemma}\label{lem:Knapp}
Let $p\in[1,\infty]$ and $c\geq 1$. Then there exists a $C\geq 1$ such that the following holds. Let $N\geq 1$, $\w\in S^{d-1}$ and $f\in\Hp$ be such that
\[
\supp(\wh{f}\,)\subseteq\{\xi\in\Rd\mid c^{-1}N\leq |\xi|\leq cN, |\hat{\xi}-\w|\leq cN^{-1/2}\}.
\]
Then
\[
C^{-1}\|f\|_{L^{p}(\Rd)}\leq N^{\frac{d-1}{2}(\frac{1}{p}-\frac{1}{2})}\|f\|_{\Hp}\leq C\|f\|_{L^{p}(\Rd)}.
\]
\end{lemma}

\subsection{Symbol classes and quantization}\label{subsec:symbols}

For $m\in\R$ and $\rho,\delta\in[0,1]$, H\"{o}rmander's class $S^{m}_{\rho,\delta}$ consists of all $a\in C^{\infty}(\R^{2d})$ such that 
\[
\sup_{(x,\xi)\in\R^{2d}}\lb \xi\rb^{-m+|\alpha|\rho-|\beta|\delta}|\partial_{x}^{\beta}\partial_{\xi}^{\alpha}a(x,\xi)|<\infty
\]
for all $\alpha,\beta\in\N_{0}^{d}$. 
We will be interested in rough versions of these symbols,  
and in elements of $S^{m}_{1,\delta}$ which have additional smoothness, in the following sense.

\begin{definition}\label{def:symbolrough}
Let $r>0$, $m\in\R$ and $\delta\in[0,1]$. Then $C^{r}_{*}S^{m}_{1,\delta}$ consists of all $a:\R^{2n}\to\C$ such that the following properties hold for all $\alpha\in \N_{0}^{d}$:
\begin{enumerate}
\item\label{it:symbolrough1} $a(x,\cdot)\in C^{\infty}(\Rn)$ for all $x\in\Rn$, and 

$\sup_{(x,\xi)\in\R^{2d}}\lb\xi\rb^{-m+|\alpha|}|\partial_{\xi}^{\alpha}a(x,\xi)|<\infty$; 
\item\label{it:symbolrough2} $\partial_{\xi}^{\alpha}a(\cdot,\xi)\in C^{r}_{*}(\Rn)$ for all $\xi\in\Rn$,  
and
$\sup_{\xi\in\Rd}\lb\xi\rb^{-m+|\alpha|-\delta r}\|\partial_{\xi}^{\alpha}a(\cdot,\xi)\|_{C^{r}_{*}(\Rd)}\linebreak<\infty$.
\end{enumerate} 
Moreover, $C^{1,1}S^{m}_{1,\delta}$ consists of all $a\in C^{2}_{*}S^{m}_{1,\delta}$ such that, for all $\alpha\in \N_{0}^{d}$ with $|\alpha|=2$,
\[
\sup_{(x,\xi)\in\R^{2d}}\lb\xi\rb^{-m+2-2\delta}|\partial_{\xi}^{\alpha}a(x,\xi)|\linebreak<\infty.
\]
Finally, for $r\in\N$, the collection $\A^{r}S^{m}_{1,\delta}$ consists of those $a\in S^{m}_{1,\delta}$ such that 
\[
\sup_{\xi\in\Rd}\lb\xi\rb^{-m+|\alpha|-\delta(s-r)_{+}}\|\partial_{\xi}^{\alpha}a(\cdot,\xi)\|_{C^{s}_{*}(\Rd)}<\infty
\] 
for all $\alpha\in\N_{0}^{d}$ and $s\geq0$, and, if $s\in\N_{0}$, then 
\[
\sup_{x,\xi\in\Rd}\lb\xi\rb^{-m+|\alpha|-\delta(s-r)_{+}}|\partial_{x}^{\beta}\partial_{\xi}^{\alpha}a(x,\xi)|<\infty
\]
for all $\beta\in\N_{0}^{d}$  such that $|\beta|=s$.
\end{definition}

An $a\in C^{r}_{*}S^{m}_{1,\delta}$ is \emph{homogeneous of degree $m\in\R$ for $|\xi|\geq 1$} if 
$a(x,\la\xi)=\la^{m} a(x,\xi)$ for all $\lambda\geq 1$ and $(x,\xi)\in\R^{2d}$ with $|\xi|\geq1$. 
However, the symbols which we work with merely satisfy the following weaker version of homogeneity. 

\begin{definition}\label{def:asymphom}
Let  
$m\in\R$, $\delta\in[0,1]$ and $b\in S^{m}_{1,\delta}$, and let $a\in C^{1,1} S^{m}_{1,0}$ be homogeneous of degree $m$ for $|\xi|\geq1$. Then $b\in S^{m}_{1,\delta}$ is \emph{asymptotically homogeneous of degree $m$ with limit} $a$ if the following conditions hold:
\begin{subequations}\label{eq:asymphom}
\begin{align*}
& [(x,\xi)\mapsto (\xi\cdot \partial_{\xi}-m)b(x,\xi)]\in S^{m-1}_{1,\delta};\\
& a-b\in C^{1,1} S^{m-1}_{1,\delta}.
\end{align*}
\end{subequations}
\end{definition}

In Section \ref{subsec:reduction} we will apply a symbol smoothing procedure to $a\in C^{r}_{*}S^{m}_{1,\delta}$, defined using the Littlewood--Paley decomposition from \eqref{eq:LittlePaley}. For $(x,\xi)\in\R^{2d}$, set
\[
a^{\sharp}(x,\xi):=\sum_{M\in 2^{\N_{0}}}\psi(M^{-1/2}D)a(x)\psi_{M}(\xi)
\]
and
\begin{equation}\label{eq:flat}
a^{\flat}(x,\xi):=a(x)-a^{\sharp}(x,\xi)=\sum_{M\in 2^{\N_{0}}}(1-\psi)(M^{-1/2}D)a(x)\psi_{M}(\xi).
\end{equation}
The following is a special case of \cite[Lemma 3.8]{LiRoSo25a}. 

\begin{lemma}\label{lem:smoothing}
Let $r>0$, $m\in\R$ and $a\in C^{r}_{*}S^{m}_{1,0}$. Then $a^{\flat}\in C^{r}_{*}S^{m-r/2}_{1,1/2}$. Moreover, if $a\in C^{1,1}S^{m}_{1,0}$, then $a^{\sharp}\in \A^{2}S^{m}_{1,1/2}$, and if $a$ is additionally homogeneous of degree $m$ for $|\xi|\geq 1$, 
then $a^{\sharp}$ is asymptotically homogeneous of degree $m$ with limit $a$. If $a$ is elliptic, then $a^{\sharp}$ is elliptic as well.
\end{lemma}

We define the standard quantization of an $a\in C^{r}_{*}S^{m}_{1,\delta}$ by
\begin{equation*}
    a(x,D) f(x) := \frac{1}{(2\pi)^d} \int_{\R^{2d}} e^{i \langle x-y,\xi \rangle} a(x,\xi) f(y) \ud y\ud \xi,
\end{equation*}
for $r>0$, $m\in\R$, $\delta\in[0,1]$, $f\in\Sw(\Rd)$ and $x\in\Rd$. 
Finally, the following is a special case of \cite[Theorem 3.19]{LiRoSo25a}. 

\begin{lemma}\label{lem:pseudoHpFIO}
Let $r>0$, $m\in\R$, $p\in[1,\infty]$ and $a\in C^{r}_{*}S^{m}_{1,1/2}$. For $\veps\in(0,\frac{r}{2}]$, set
\[
\sigma:=\begin{cases}
0&\text{if }r>4s(p),\\
2s(p)-\frac{r}{2}+\veps&\text{if }0<r\leq 4s(p).
\end{cases}
\]
Then $a(x,D):\HT^{s+\sigma+m,p}_{FIO}(\Rn)\to \Hps$ for all $-\tfrac{r}{2}+s(p)-\sigma<s<r-s(p)$. 
\end{lemma}

When applying this lemma, we will often use that the operator norm of $a(x,D)$  only depends on $a$ through a fixed finite number of its $C^{r}_{*}S^{m}_{1,1/2}$ seminorms.

\section{Local smoothing}\label{sec:BilinearSmoothing}

This section is devoted to Theorem \ref{thm:mainHpFIOintro}. We first state a more general version in Theorem \ref{thm:mainrough}; the proof takes up the rest of the section. 

\subsection{Main results}\label{subsec:main}

Let $(g_{ij})_{i,j=1}^{d}\subseteq C^{1,1}(\R^d)$ be a uniformly elliptic collection of bounded real-valued functions, with bounded first and second derivatives. We also assume that $(g_{ij})_{i,j=1}^{d}\subseteq C^{r}_{*}(\R^d)$ for some $r\geq 2$. Note that $C^{1,1}(\R^d)\subseteq C^{2}_{*}(\R^d)$, so this condition is only relevant for $r>2$. Let
\begin{equation}\label{eq:defL}
Lf(x):=\sum_{i,j=1}^{d}D_{i}(g_{ij}(x)D_{j}f)(x)
\end{equation}
be the associated divergence-form operator, defined for suitable $f$ and $x\in\Rd$. 

The following result is obtained by combining \cite[Theorem 4.7 and Corollary 4.13]{LiRoSo25a}. It contains sharp fixed-time bounds for the wave equation associated with $L$.

\begin{proposition}\label{prop:roughfixed}
There exist unique collections $(U_{0}(t))_{t\in\R},(U_{1}(t))_{t\in\R}$ such that, for all $1\leq p<\infty$ and $s\in\R$ with $2s(p)+1<r$ and $-r+s(p)+1< s< r-s(p)$, and for all $u_{0}\in\Hps$, $u_{1}\in \Hpsm$ and $F\in L^{1}_{\loc}(\R;\Hpsm)$, the following properties hold:
\begin{enumerate}
\item\label{it:roughfixed1} $U_{k}(t):\HT^{s-k,p}_{FIO}(\R^d)\to\Hps$ is bounded for all $t\in\R$ and $k\in\{0,1\}$, and for all $t_{0}>0$ one has $\sup_{|t|\leq t_{0}}\|U_{k}(t)\|_{\La(\HT^{s-k,p}_{FIO}(\R^d),\Hps)}<\infty$; 
\item\label{it:roughfixed2} $[t\mapsto U_{0}(t)u_{0}],[t\mapsto U_{1}(t)u_{1}]\in C^{l}(\R;\HT^{s-l,p}_{FIO}(\R^d))$ for $l\in\{0,1,2\}$;
\item\label{it:roughfixed3} Set 
\begin{equation}\label{eq:defsol}
u(t):=U_{0}(t)u_{0}+U_{1}(t)u_{1}-\int_{0}^{t}U_{1}(t-\tau)F(\tau)\ud \tau
\end{equation}
for $t\in\R$. Then 
\begin{equation}\label{eq:regularityu}
u\in C(\R;\Hps)\cap C^{1}(\R;\Hpsm)\cap W^{2,1}_{\loc}(\R;\HT^{s-2,p}_{FIO}(\R^d)),
\end{equation}
$u(0)=u_{0}$, $\partial_{t}u(0)=u_{1}$, and $(D_{t}^{2}-L)u(t)=F(t)$ 
for almost all $t\in \R$.
\end{enumerate}
For $p=\infty$, under the above assumptions on $s$ and $r$, \eqref{it:roughfixed1} remains true, and \eqref{it:roughfixed2} and \eqref{it:roughfixed3} hold in the weak-star sense. 
\end{proposition}

\begin{remark}\label{rem:unique}
In fact, we will use that a stronger uniqueness statement is contained in the proofs of \cite[Theorem 4.7 and Corollary 4.13]{LiRoSo25a}. Namely, for given $p$, $s$, $u_{0}$, $u_{1}$ and $F$ as in Proposition \ref{prop:roughfixed}, there is a unique $u$ as in \eqref{eq:regularityu} satisfying $u(0)=u_{0}$, $\partial_{t}u(0)=u_{1}$, and $(D_{t}^{2}-L)u(t)=F(t)$ in $\HT^{s-2,p}_{FIO}(\R^d)$ for almost all $t\in \R$.
\end{remark}

The following theorem is the main result of this section. It implies Theorem \ref{thm:mainHpFIOintro} and, by using \eqref{eq:Sobolev}, 
also Theorem \ref{thm:mainLpintro}. Recall the definition of $\mathfrak{Q}$ from \eqref{eq:Q}, of the closed triangles $\mathcal{T}_{l}$ and $\mathcal{T}_{u}$, and of $\alpha(p,q)$, in \eqref{eq:alphapq}.

\begin{theorem}\label{thm:mainrough}
Let $2\leq p<q\leq \infty$ be such that $(\frac{1}{p},\frac{1}{q})\in \mathcal{T}_{l}\cup \mathcal{T}_{u}$, and let $\alpha\geq \alpha(p,q)-s(p)$. Suppose that $r>\max(2s(p)+1,\alpha(p,q))$, and that one of the following conditions holds:
\begin{enumerate}
\item\label{it:mainrough1} $\alpha>\alpha(p,q)-s(p)$; 
\item\label{it:mainrough2} $(\frac{1}{p},\frac{1}{q})$ lies in the interior of $\mathcal{T}_{l}$. 
\end{enumerate} 
Then there exists a $C\geq0$ such that, for all $u_{0}\in \HT^{\alpha,p}_{FIO}(\R^d)$, $u_{1}\in \HT^{\alpha-1,p}_{FIO}(\R^d)$ and $F\in L^{1}_{\loc}(\R;\HT^{\alpha-1,p}_{FIO}(\R^d))$, the solution $u$ from \eqref{eq:defsol} satisfies $u\in L^{q}([0,1]\times\R^d)$ and
\[
\|u\|_{L^{q}([0,1]\times \R^d)}\leq C\big(\|u_{0}\|_{\HT^{\alpha,p}_{FIO}(\R^d)}+\|u_{1}\|_{\HT^{\alpha-1,p}_{FIO}(\R^d)}+\|F\|_{L^{1}([0,1];\HT^{\alpha-1,p}_{FIO}(\R^d))}\big).
\]
\end{theorem}

\begin{remark}\label{rem:dyadicestimates}
The endpoint $\alpha=\alpha(p,\infty)-s(p)=\frac{d-1}{2}+\frac{1}{p}-s(p)$ is also allowed for all $p\geq 2$ and $q=\infty$ if one replaces $L^{\infty}([0,1]\times\Rd)$ by $L^{\infty}([0,1];\HT^{\infty}(\Rd))$; this follows directly from Proposition \ref{prop:roughfixed} and a Sobolev embedding.
\end{remark}

\begin{remark}\label{rem:sqinfty}
If $(g_{ij})_{i,j=1}^{d}\subseteq \HT^{r,\infty}(\Rd)$ and $q<\infty$ in Theorem \ref{thm:mainrough}, then the conclusion also holds under the assumption that $r>2s(p)+1$ and $r\geq \alpha+s(p)$. This follows from the proof, and in particular from the proof of Proposition \ref{prop:reduction}, upon noting that Lemma \ref{lem:pseudoHpFIO} holds for $s=r-s(p)$ if the coefficients of the pseudodifferential symbol have $\HT^{r,\infty}(\Rd)$ regularity in space (see \cite[Theorem 3.19]{LiRoSo25a}).
\end{remark}

\begin{remark}\label{rem:Sobolevreg}
One can also obtain estimates of the form
\[
\|u\|_{L^{q}([0,1];W^{s,q}(\Rd))}\lesssim\|u_{0}\|_{\HT^{s+\alpha,p}_{FIO}(\R^d)}+\|u_{1}\|_{\HT^{s+\alpha-1,p}_{FIO}(\R^d)}+\|F\|_{L^{1}([0,1];\HT^{s+\alpha-1,p}_{FIO}(\R^d))}
\]
in Theorem \ref{thm:mainrough}, for $s$ in an $r$-dependent interval. To this end, one mainly has to extend Proposition \ref{prop:reduction} below, for which one can use the same template, with Lemma \ref{lem:pseudoHpFIO} leading to the restriction on $s$, $p$ and $q$. See also the duality arguments in \cite{LiRoSo25a} for the case where $s<0$.
For simplicity of the presentation, we will not consider such Sobolev estimates.
\end{remark}

\begin{remark}\label{rem:lowerorder}
One could allow lower-order terms in \eqref{eq:defL}, as long as they have enough regularity to be absorbed in the Duhamel term in the proof of Proposition \ref{prop:reduction}. The same $C^{r}_{*}(\Rd)$ regularity as the principal terms suffices, but less is also possible. For simplicity of the presentation, we will not consider lower-order terms.
\end{remark}

By additionally relying on finite speed of propagation, we can also derive from Theorem \ref{thm:mainrough} nontrivial local smoothing outside of $\mathcal{T}_{l}\cup \mathcal{T}_{u}$. Set
\[
p_{q}:=
\begin{cases}
\frac{2dq}{2+(d-1)q}&\text{for }2\leq q\leq \frac{2(d^{2}+2d-1)}{d^{2}-1},\\
\frac{(d+1)q}{d+3}&\text{for }\frac{2(d^{2}+2d-1)}{d^{2}-1}\leq q\leq \infty.
\end{cases}
\]
Note that $2\leq p_{q}\leq p$ is such that $(\frac{1}{p_{q}},\frac{1}{q})$ lies either on the line between $\mathfrak{Q}$ and $(0,0)$, or on the line between $\mathfrak{Q}$ and $(\frac{1}{2},\frac{1}{2})$, with the latter case given by $\frac{1}{q}-\frac{1}{2}=d(\frac{1}{p_{q}}-\frac{1}{2})$. Note also that $\alpha(p_{q},q)=
2s(q)(1-\frac{d-1}{2d})$ for $2\leq q\leq \frac{2(d^{2}+2d-1)}{d^{2}-1}$, and that $\alpha(p_{q},q)=2s(q)-\frac{d-1}{q(d+1)}$ for $\frac{2(d^{2}+2d-1)}{d^{2}-1}\leq q\leq \infty$.

The following result implies implies Corollary \ref{cor:extrabutweakintro}.

\begin{corollary}\label{cor:extrabutweak}
Let $2\leq p\leq q\leq \infty$ be such that $(\frac{1}{p},\frac{1}{q})\notin \mathcal{T}_{l}\cup \mathcal{T}_{u}$, and let $\alpha>\alpha(p_{q},q)$. Suppose that $r>\max(2s(p_{q})+1,\alpha(p_{q},q))$. Then there exists a $C\geq0$ such that, for all $u_{0}\in W^{\alpha,p}(\R^d)$, $u_{1}\in W^{\alpha-1,p}(\R^d)$ and $F\in L^{1}_{\loc}(\R;W^{\alpha-1,p}(\R^d))$, the solution $u$ from \eqref{eq:defsol} satisfies $u\in L^{q}([0,1]\times\R^d)$ and
\[
\|u\|_{L^{q}([0,1]\times \R^d)}\leq C\big(\|u_{0}\|_{W^{\alpha,p}(\R^d)}+\|u_{1}\|_{W^{\alpha-1,p}(\R^d)}+\|F\|_{L^{1}([0,1];W^{\alpha-1,p}(\R^d))}\big).
\]
\end{corollary}
\begin{proof}
We show that $U_{j}:W^{\alpha-j,p}(\Rd)\to L^{q}([0,1]\times\Rd)$, for $j\in\{0,1\}$. It suffices to prove this in the special case where $q=\frac{2(d^{2}+2d-1)}{d^{2}-1}$, but this makes no difference for the argument.

First consider $f\in W^{\alpha-j,p}(\Rd)$ such that $\supp(f)\subseteq B$ for some ball $B\subseteq \Rd$ of radius $2$.   
Hence Theorem \ref{thm:mainrough} yields
\begin{equation}\label{eq:localsmoothcompact}
\|U_{j}f\|_{L^{q}([0,1]\times\Rd}\lesssim \|f\|_{\HT^{\alpha-s(p_{q})-j,p_{q}}_{FIO}(\Rd)}\lesssim \|f\|_{W^{\alpha-j,p_{q}}(\Rd)}\lesssim \|f\|_{W^{\alpha-j,p}(\Rd)},
\end{equation}
where we applied \eqref{eq:Sobolev} in the second step and finally \cite[Lemma 3.3.1]{Triebel10}.

To deal with a general $f\in W^{\alpha-j,p}(\Rd)$, recall from \cite[Proposition 7.2]{Triebel92} that 
\begin{equation}\label{eq:localization}
\|f\|_{W^{\alpha-j,p}(\Rd)}\eqsim \Big(\sum_{\kappa\in \Z^{d}}\|\theta_{\kappa}f\|_{W^{\alpha-j,p}(\Rd)}^{p}\Big)^{1/p},
\end{equation}
Here $\theta_{\kappa}f$ denotes multiplication of $f$ by $\theta_{\kappa}\in C^{\infty}_{c}(\Rd)$, and the latter is given by $\theta_{\kappa}(x):=\theta(x-\kappa)$ for $x\in\Rd$ and $\kappa\in \Z^{d}$, where $\theta\in C^{\infty}_{c}(\Rd)$ satisfies $\supp(\theta)\subseteq B_{d}(0,2)$, $0\leq \theta\leq 1$ and  $\sum_{\kappa\in \Z^{d}}\theta(x-\kappa)=1$ for all $x\in\Rd$. 

Note that, due to finite speed of propagation (see e.g.~\cite[Remark 1.8]{Colombini-Metivier08}), there exists a compact set $K\subseteq \Z^{d}$, independent of $f$, such that $\theta_{\kappa}U_{j}(t)(\theta_{\kappa'}f)=0$ for all $0\leq t\leq 1$ and $\kappa,\kappa'\in\Z^{d}$ such that $\kappa-\kappa'\notin K$. 
Hence we can apply H\"{o}lder's inequality and \eqref{eq:localsmoothcompact} to write
\begin{align*}
\|U_{j}f\|_{L^{q}([0,1]\times\Rd)} &\eqsim \Big(\sum_{\kappa\in \Z^{d}}\|\theta_{\kappa}U_{j}f\|_{L^{q}([0,1]\times\Rd)}^{q}\Big)^{1/q}\\  
&\lesssim  \Big(\sum_{\kappa,\kappa'\in \Z^{d}}\ind_{K}(\kappa-\kappa')\|\theta_{\kappa}U_{j}(\theta_{\kappa'}f)\|_{L^{q}([0,1]\times\Rd)}^{q}\Big)^{1/q}\\
&\lesssim \Big(\sum_{\kappa'\in \Z^{d}}\|\theta_{\kappa'}f\|_{W^{\alpha-j,p}(\Rd)}^{q}\Big)^{1/q} \eqsim \|f\|_{{W^{\alpha-j,p}(\Rd)}}.
\end{align*}
On the last line we also used the embedding $\ell^{p}\subseteq \ell^{q}$. This concludes the proof.
\end{proof}

\begin{remark}\label{rem:nontrivialsmoothing}
One can also obtain a version of Corollary \ref{cor:extrabutweak} where $W^{\alpha,p}(\Rd)$ is replaced by $\HT^{\alpha,p}_{FIO}(\Rd)$, for $\alpha>s(q)$. For $(\frac{1}{p},\frac{1}{q})\notin \mathcal{T}_{l}\cup \mathcal{T}_{u}$, this result neither implies Corollary \ref{cor:extrabutweak} nor follows from it. Moreover, while this version is nontrivial for all $p<q$, it is weaker than Proposition \ref{prop:roughfixed} for $p=q$.

To prove such a version of Corollary \ref{cor:extrabutweak}, one proceeds in the same way as above, using a natural replacement of the final inequality in \eqref{eq:localsmoothcompact} and applying \cite[Theorem 3.6]{LiRoSoYa24} instead of \eqref{eq:localization}.
\end{remark}

\subsection{Reduction steps}\label{subsec:reduction}

In this subsection we reduce Theorem \ref{thm:mainrough} to frequency-localized estimates for a collection of smooth first-order equations. 

In fact, we first show that it suffices to deal with coefficients with small $\dot{C}^{1,1}$ norm. This will allow for sufficient control of the Hamiltonian flow, as required in \cite{SchippaTataru2025}.  
More precisely, small $\dot{C}^{1,1}$-norm of the coefficients ensures the absence of caustics for $t \in [0,1]$, i.e., the validity of dispersive effects up to $t=1$ (cf.~\cite{Geba-Tataru05}).

\begin{lemma}\label{lem:absencecaustics}
Suppose that there exist $\veps_{0},C_{0}>0$ such that the conclusion of Theorem \ref{thm:mainrough} holds with a uniform constant $C\geq0$ for all $(g_{ij})_{1\leq i,j\leq d}$ as above, satisfying $\max_{1\leq i,j\leq d}\|g_{ij}\|_{C^{r}_{*}(\Rd)}\leq C_{0}$, $\sum_{i,j=1}^{d}g_{ij}(x)\xi_{i}\xi_{j}\geq \veps_{0}|\xi|^{2}$ for all $x,\xi\in\Rd$, and
\begin{equation}
\label{eq:AbsenceCaustics}
\max_{1\leq i,j\leq n}\| \partial_x g_{ij} \|_{L^{\infty}(\Rd)} + \| \partial^2_{x} g_{ij} \|_{L^{\infty}(\Rd)} \leq \tfrac{1}{4}\epsilon_0.
\end{equation}
Then the conclusion of Theorem \ref{thm:mainrough} holds without the additional assumption \eqref{eq:AbsenceCaustics}.
\end{lemma}
\begin{proof}
For $\ka>0$ sufficiently small, we can rescale $t \to t/\kappa$, $x \to x/\kappa$, which leads us to consider a wave equation satisfying \eqref{eq:AbsenceCaustics} on the large time interval $[0,\kappa^{-1}]$:
\begin{equation*}
    \partial_t^2 u^{\kappa}(t,x)= \partial_i g^{ij}_{\kappa} \partial_j u^{\kappa}(t,x) + F^{\kappa}(t,x), \quad (t,x) \in [0,\kappa^{-1}] \times \R^d.
\end{equation*}
We can partition $[0,\kappa^{-1}]$ into $\lceil \kappa^{-1} \rceil$ intervals $I \in \mathcal{I}$ 
of length at most $1$. 

By assumption,
\[
\begin{split}
    &\| u^{\kappa} \|_{L^q([0,\kappa^{-1}] \times \R^d)}=\Big( \sum_{I \in \mathcal{I}} \| u^{\kappa} \|^q_{L^q(I \times \R^d)} \Big)^{1/q} \\
    &\lesssim\Big(\sum_{I \in \mathcal{I}} \| u^{\kappa}(a_I) \|^q_{\mathcal{H}^{\alpha,p}_{\text{FIO}}} + \| \partial_t u^{\kappa}(a_I) \|^q_{\mathcal{H}^{\alpha-1,p}_{\text{FIO}}(\R^d)} + \| F^{\kappa} \|^q_{L^1(I;\mathcal{H}^{\alpha-1,p}_{\text{FIO}}(\R^d)} \big)^{1/q}.
\end{split}
\]
We use the invariance of the $\mathcal{H}^{\alpha,p}_{FIO}(\Rd)$ norm under time evolution, as contained in Proposition \ref{prop:roughfixed}, to sum over $I \in \mathcal{I}$:
\begin{equation*}
\Big( \sum_{I \in \mathcal{I}} \| u^{\kappa} \|^q_{L^q(I \times \R^d)} \Big)^{1/q}\lesssim_{\ka} 
\| u^{\kappa}_0 \|_{\mathcal{H}^{\alpha,p}_{FIO}} + \| u_1^{\kappa} \|_{\mathcal{H}^{\alpha-1,p}_{FIO}} + \| F^{\kappa} \|_{L^1([0,\kappa^{-1}],\mathcal{H}^{\alpha-1,p}_{FIO})}.
\end{equation*}
For the inhomogeneity we relied on the embedding $\ell^1 \hookrightarrow \ell^q$. 
Finally, the operation of dilation is a Fourier integral operator as in \cite[Proposition 3.3]{Rozendaal22b}, so it acts boundedly on $\mathcal{H}^p_{FIO}(\Rd)$. This suffices.
\end{proof}

In the remainder we will assume that \eqref{eq:AbsenceCaustics} is satisfied.
For $N\in 2^{\N}$, $x\in\Rd$ and suitable $f$, set 
\[
L_{N}f(x):=\sum_{i,j=1}^{d}D_{i}(\psi(N^{-1/2}D)g_{ij}(x)D_{j}f)(x).
\]
Note that the coefficients $(\psi(N^{-1/2}D)g_{ij})_{i,j=1}^{d}\subseteq C^{\infty}(\Rd)$ are real-valued, have uniformly bounded (in $N$) derivatives up to order $r$, and are uniformly bounded in $C^{r}_{*}(\Rd)$. Moreover, the coefficients are \emph{uniformly elliptic}, in the sense that 
there exist $c>0$ and $N_{0}\in 2^{\N}$ such that 
\[
\Big|\sum_{i,j=1}^{d}\psi(N^{-1/2}D)g_{ij}(x)\xi_{i}\xi_{j}\Big|\geq c|\xi|^{2}
\]
for all $N\geq N_{0}$ and $(x,\xi)\in\R^{2d}$. This follows from the ellipticity of $(g_{ij})_{i,j=1}^{d}$, since by Young's inequality and dyadic summation, 
\begin{align*}
&|(1-\psi(N^{-1/2}D))g_{ij}(x)|
\leq \sum_{M\geq N^{1/2}/4}\|\psi_{M}(D)(1-\psi(N^{-1/2}D))g_{ij}\|_{L^{\infty}(\Rd)}\\
&\leq C N^{-r/2}\|(1-\psi(N^{-1/2}D))g_{ij}\|_{C^{r}_{*}(\Rd)}\lesssim N^{-r/2}\|g_{ij}\|_{C^{r}_{*}(\Rd)}
\end{align*}
for all 
$1\leq i,j\leq d$. 
Hence, for $N\in 2^{\N}$ sufficiently large, there exist unique collections 
as in Proposition \ref{prop:roughfixed}, associated with $L_{N}$. Given that the coefficients $(\psi(N^{-1/2}D)g_{ij})_{i,j=1}^{d}$ are smooth, these collections have the properties in Proposition \ref{prop:roughfixed} for all $p$ and $s$, but the bounds in \eqref{it:roughfixed1} will not be independent of $N$, unless $2s(p)+1<r$ and $-r+s(p)+1<s<r-s(p)$. 

We record that the smallness assumption \eqref{eq:AbsenceCaustics} is propagated too. That is,
\begin{equation*}
\max_{1\leq i,j\leq d}    \| \partial_x \psi(N^{-1/2}(D)) g^{ij} \|_{L^\infty(\Rd)} + \| \partial^2_x \psi(N^{-1/2}(D)) g^{ij} \|_{L^\infty(\Rd)} \leq \epsilon_0
\end{equation*}
since $\| \psi_N\|_{L^1(\Rd)} \leq 4$.

The operators $L_{N}$ arise when dealing with frequency-localized initial data. However, it will be convenient to work with symbols and operators that are actually uniformly smooth. Hence we decompose each $L_{N}$ as $L_{N}=L_{1,N}+L_{2,N}$, where
\[
L_{1,N}:=\sum_{i,j=1}^{d}D_{i}(\psi(N^{-1/2}D)g_{ij})^{\sharp}(x,D)D_{j},
\]
and 
\[
L_{2,N}:=\sum_{i,j=1}^{d}D_{i}(\psi(N^{-1/2}D)g_{ij})^{\flat}(x,D)D_{j}.
\]
This decomposition has the following properties.

\begin{lemma}\label{lem:Ldecomp}
There exist $N_{0}\in 2^{\N}$ and $c\in\R$ such that the following holds for all $N\geq N_{0}$, $p\in[1,\infty]$ and $s\in\R$:
\begin{enumerate}
\item\label{it:Ldecomp1} There exist a real-valued elliptic $b_{N}\in\A^{2}S^{1}_{1,1/2}$ and an $e_{N}\in \A^{1}S^{1}_{1,1/2}$ such that $L_{1,N}=(b_{N}(x,D)+ic)^{2}+e_{N}(x,D)$. One may choose $b_{N}$ such that
\begin{equation}\label{eq:defsymbolb}
b_{N}(x,\xi)=\Big(\sum_{i,j=1}^{d}(\psi(N^{-1/2}D)g_{ij})^{\sharp}(x,\xi)\xi_{i}\xi_{j}\Big)^{1/2}
\end{equation}
for all $(x,\xi)\in\R^{2d}$ with $|\xi|\geq N_{0}/2$, and such that $b_{N}$ is asymptotically homogeneous of degree $1$ with limit given by $(\sum_{i,j=1}^{d}\psi(N^{-1/2}D)g_{ij}(x)\xi_{i}\xi_{j})^{1/2}$ for all $(x,\xi)\in \R^{2d}$ with $|\xi|\geq 1$; 
\item\label{it:Ldecomp2}
$b_{N}(x,D)+ic:\HT^{s,p}_{FIO}(\Rn)\to \HT^{s-1,p}_{FIO}(\Rd)$ is an isomorphism;
\item\label{it:Ldecomp3} If $r>2s(p)+1$ and $0\leq s< r-s(p)$, then $L_{2,N}:\HT^{s,p}_{FIO}(\Rd)\to \HT^{s-1,p}_{FIO}(\Rd)$ is bounded.
\end{enumerate}
\end{lemma}
The symbol estimates and operator norms in Lemma \ref{lem:Ldecomp} are uniform in $N$.
\begin{proof}
\eqref{it:Ldecomp1}: 
This is essentially contained in \cite[Proposition 5.1]{HassellPortalRozendaal2020}, and uses Lemma \ref{lem:smoothing}. In fact, the statement there merely implies that $e_{N}\in S^{1}_{1,1/2}$, but the additional regularity follows from the proof. 
See also the proof of \cite[Proposition 4.9]{HassellPortalRozendaal2020} for \eqref{eq:defsymbolb}. 

\eqref{it:Ldecomp2}: This is \cite[Lemma 4.10]{HassellRozendaal2023}. 
\eqref{it:Ldecomp3}:  Combine the first statement of 
Lemma \ref{lem:smoothing} with Lemma \ref{lem:pseudoHpFIO}. 
\end{proof}

Next, we record the fixed-time regularity properties of the solution operators to the half-wave equation associated with the $b_{N}(x,D)$.

\begin{lemma}\label{lem:fixedtimesmooth}

Let $N_{0}\in 2^{\N}$ be as in Lemma \ref{lem:Ldecomp}. Then, for each $N\geq N_{0}$, there exists a unique collection $(S_{N}(t))_{t\in\R}$ of operators on $\Sw'(\Rd)$ such that the following holds for all $p\in[1,\infty)$ and $s\in\R$:
\begin{enumerate}
\item\label{it:firstorder0} $S_{N}(t):\Hps\to\Hps$ is bounded for all $t\in\R$ and $N\geq N_{0}$, and for each $t_{0}>0$ one has $\sup\{\|S_{N}(t)\|_{\Hps}\mid t\in[-t_{0},t_{0}], N\geq N_{0}\}<\infty$;
\item\label{it:firstorder1} $[t\mapsto S_{N}(t)f]\in C^{k}(\R;\HT^{s-k,p}_{FIO}(\R^d))$ for all $f\in\Hps$ and $N\geq N_{0}$; 
\item\label{it:firstorder3} $D_{t}S_{N}(t)f=b_{N}(x,D)f
$ for all $t\in\R$, $f\in\Hps$ and $N\geq N_{0}$.
\end{enumerate}
For $p=\infty$, under the above assumptions on $s$ and $r$, \eqref{it:firstorder0} remains true, and \eqref{it:firstorder1} and \eqref{it:firstorder3} hold in the weak-star sense.
\end{lemma}
\begin{proof}
For $p<\infty$, combine Proposition \ref{lem:Ldecomp} and \cite[Theorem 4.3]{LiRoSo25a}. 

For $p=\infty$, observe that \cite[Theorem 4.3]{LiRoSo25a} in fact applies with $b_{N}$ replaced by $a_{1}+a_{2}$ for any $a_{1}\in \A^{2}S^{1}_{1,1/2}$ real-valued, elliptic and asymptotically homogeneous, and for any $a_{2}\in S^{0}_{1,1/2}$. Keeping in mind that $b_{N}(x,D)^{*}=b_{N}(x,D)+\tilde{b}_{N}(x,D)$ for some $\tilde{b}_{N}\in S^{0}_{1,1/2}$,  one can thus rely on duality here.  
\end{proof}

We can now state and prove the main result of this subsection. Throughout, we write $S_{N}f(t,x):=S_{N}(t)f(x)$ for $f\in\Sw(\Rd)$ and $(t,x)\in\R^{1+d}$.

\begin{proposition}\label{prop:reduction}
Let $1\leq p\leq q\leq \infty$ and $0\leq \alpha<r-s(p)$ be such that $2s(p)+1<r$, $p<\infty$ and $\alpha>s(p)+d(\frac{1}{p}-\frac{1}{q})-1$. Suppose that there exist $N_{0}\in 2^{\N}$ 
and $C_{0}\geq0$ such that the following holds. For all $N\in 2^{\N}$ with $N\geq N_{0}$, 
and for all $f\in\HT^{\alpha,p}_{FIO}(\Rd)$ satisfying $\supp(\wh{f}\,)\subseteq \{\xi\in\Rn\mid N/4\leq |\xi|\leq 4N\}$, 
one has 
\begin{equation}\label{eq:reductionbound}
\|S_{N} f \|_{L^{q}([0,1]\times\Rd)}
\leq C_{0}\|f\|_{\HT^{\alpha,p}_{FIO}(\Rd)}.
\end{equation}
Then there exists a $C\geq0$ such that, for all $u_{0}\in \HT^{\alpha,p}_{FIO}(\R^d)$, $u_{1}\in \HT^{\alpha-1,p}_{FIO}(\R^d)$ and $F\in L^{1}_{\loc}(\R;\HT^{\alpha-1,p}_{FIO}(\R^d))$, and for all $N\in 2^{\N_{0}}$, the solution $u$ from \eqref{eq:defsol} satisfies 
\[
\|\psi_{N}(D)u\|_{L^{q}([0,1]\times\R^d)}\leq C\big(\|u_{0}\|_{\HT^{\alpha,p}_{FIO}(\R^d)}+\|u_{1}\|_{\HT^{\alpha-1,p}_{FIO}(\R^d)}+\|F\|_{L^{1}([0,1];\HT^{\alpha-1,p}_{FIO}(\R^d))}\big).
\]
\end{proposition}
Note that,  for $p=\infty$,  Proposition \ref{prop:roughfixed} already contains the relevant sharp bounds. Moreover, in Theorem \ref{thm:mainrough} one automatically has $\alpha>s(p)+d(\frac{1}{p}-\frac{1}{q})-1$.

\begin{proof}
Throughout, fix $N\in 2^{\N_{0}}$ and 
write $v(t):=\psi_{N}(D)u(t)$ for $t\in\R$. 

\subsubsection*{Preliminary work}

Firstly, by a standard Sobolev embedding and \eqref{eq:Sobolev},  
\begin{align*}
\|v\|_{L^{q}([0,1]\times\Rd)}&\lesssim \|\lb D\rb^{d(\frac{1}{p}-\frac{1}{q})}v\|_{L^{q}([0,1];L^{p}(\Rd))}\\
&\lesssim \|\lb D\rb^{d(\frac{1}{p}-\frac{1}{q})+s(p)-\alpha}v\|_{L^{\infty}([0,1];\HT^{\alpha,p}_{FIO}(\Rd))}.
\end{align*}
This suffices if $N<N_{0}$, due to Proposition \ref{prop:roughfixed} and because $\psi_{N}(D)$ is smoothing. In particular, we may suppose that $N_{0}\geq 16$ is such that Lemma \ref{lem:Ldecomp} applies to all $N\geq N_{0}$. Note that then $N^{1/2}\leq N/4$. 

Next, consider $N\geq N_{0}$. 

The argument above shows that $v\in L^{q}([0,1]\times\Rd)$, but we have to obtain norm bounds independent of $N$. To this end, let $b_{N}$, $e_{N}$ and $c$ be as in Lemma \ref{lem:Ldecomp}, and set $B:=b_{N}(x,D)+ic$. Then 
$D_{t}^{2}v(t)=B^{2}v(t)+G(t)$ for all $t\in\R$, where 
\[
G(t):=\big([\psi_{N}(D),L_{N}]+e_{N}(x,D)\psi_{N}(D)\big)u(t).
\]
Set
\[
\cs_{t}=\tfrac{1}{2}(S_{N}(t)+S_{N}(-t))\ \text{ and }\  \sn_{t}=\tfrac{1}{2i}(S_{N}(t)-S_{N}(-t))
\]
for $t\in\R$. 
Then Lemma \ref{lem:fixedtimesmooth} and Remark \ref{rem:unique} imply that
\begin{equation}\label{eq:uDuhamel}
v(t)=\cs_{t}\psi_{N}(D)u_{0}+B^{-1}\sn_{t}\psi_{N}(D)u_{1}-\int_{0}^{t}B^{-1}\sn_{t-\tau}G(\tau)\ud\tau.
\end{equation}
To bound the first two terms on the right-hand side, we can use \eqref{eq:reductionbound}, which by construction also holds if one replaces $S_{N}(t)$ by $\cs_{t}$ or, with a gain of one derivative due to Lemma \ref{lem:Ldecomp} \eqref{it:Ldecomp2}, by $B^{-1}\sn_{t}$.

\subsubsection*{The Duhamel term}

It remains to deal with the right-most term in \eqref{eq:uDuhamel}. Write
\begin{equation}\label{eq:Duhamelsplit}
G(\tau)=[\psi_{N}(D),L_{N}]u(\tau)+\psi_{N}(D)e_{N}(x,D)u(\tau)+[e_{N}(x,D),\psi_{N}(D)]u(\tau)
\end{equation}
for $\tau\in\R$. We will bound the contribution of each of these terms separately.

Firstly, since $e_{N}\in \A^{1}S^{1}_{1,1/2}$, pseudodifferential calculus yields an $e_{1,N}\in S^{0}_{1,1/2}$ such that $[e_{N}(x,D),\psi_{N}(D)]=e_{1,N}(x,D)$. By Lemma \ref{lem:pseudoHpFIO} and Proposition \ref{prop:roughfixed},
\begin{align*}
&\|[e_{N}(x,D),\psi_{N}(D)]u(\tau)\|_{\HT^{\alpha,p}_{FIO}(\Rd)}\lesssim \|u(\tau)\|_{\HT^{\alpha,p}_{FIO}(\Rd)}\\
&\lesssim \|u_{0}\|_{\HT^{\alpha,p}_{FIO}(\R^d)}+\|u_{1}\|_{\HT^{\alpha-1,p}_{FIO}(\R^d)}+\|F\|_{L^{1}([0,1];\HT^{\alpha-1,p}_{FIO}(\R^d))}
\end{align*}
for all $0\leq \tau\leq 1$. For $\veps>0$ small and for all $\tau\leq t\leq 1$, one can combine this with a classical Sobolev embedding, \eqref{eq:Sobolev}, Lemma \ref{lem:fixedtimesmooth} and Lemma \ref{lem:Ldecomp} \eqref{it:Ldecomp2}, as follows:
\begin{align*}
&\big\|B^{-1}\sn_{t-\tau}[e_{N}(x,D),\psi_{N}(D)]u(\tau)\big\|_{L^{q}([0,1]\times\Rd)}\\
&\lesssim \big\|B^{-1}\sn_{t-\tau}[e_{N}(x,D),\psi_{N}(D)]u(\tau)\big\|_{\HT^{s(p)+d(\frac{1}{p}-\frac{1}{q})+\veps,p}_{FIO}(\Rn)}\\
&\lesssim \big\|[e_{N}(x,D),\psi_{N}(D)]u(\tau)\big\|_{\HT^{s(p)+d(\frac{1}{p}-\frac{1}{q})-1+\veps,p}_{FIO}(\Rn)}\\
&\lesssim \|u_{0}\|_{\HT^{\alpha,p}_{FIO}(\R^d)}+\|u_{1}\|_{\HT^{\alpha-1,p}_{FIO}(\R^d)}+\|F\|_{L^{1}([0,1];\HT^{\alpha-1,p}_{FIO}(\R^d))}.
\end{align*}
For the final inequality, we also used that $\alpha>s(p)+d(\frac{1}{p}-\frac{1}{q})-1$. Upon integrating in $\tau$, this takes care of the contribution of the right-most term in \eqref{eq:Duhamelsplit}.

Next, we decompose the first term on the right-hand side of \eqref{eq:Duhamelsplit} as 
\[
[\psi_{N}(D),L_{N}]u(\tau)=[\psi_{N}(D),L_{1,N}]u(\tau)+[\psi_{N}(D),L_{2,N}]u(\tau).
\]
By Lemma \ref{lem:Ldecomp} \eqref{it:Ldecomp3}, $L_{2,N}:\HT^{\alpha,p}_{FIO}(\Rd)\to \HT^{\alpha-1,p}_{FIO}(\Rd)$, and thus $[\psi_{N}(D),L_{2,N}]:\HT^{\alpha,p}_{FIO}(\Rd)\to \HT^{\alpha-1,p}_{FIO}(\Rd)$ as well. On the other hand, combining Lemma \ref{lem:Ldecomp} \eqref{it:Ldecomp1} with standard pseudodifferential calculus and the embedding $\A^{1}S^{1}_{1,1/2}\subseteq \A^{2}S^{2}_{1,1/2}$, one sees that $L_{1,N}=a(x,D)$ for some $a\in \A^{2}S^{2}_{1,1/2}$.
In turn, pseudodifferential calculus also yields an $\tilde{a}\in S^{1}_{1,1/2}$ such that  $[\psi_{N}(D),L_{1,N}]=\tilde{a}(x,D)$. By Lemma \ref{lem:pseudoHpFIO}, we thus conclude that 
\begin{equation}\label{eq:Duhamelreg}
\begin{aligned}
&\|[\psi_{N}(D),L_{N}]u(\tau)+\psi_{N}(D)e_{N}(x,D)u(\tau)\|_{\HT^{\alpha-1,p}_{FIO}(\Rd)}\lesssim\|u(\tau)\|_{\HT^{\alpha,p}_{FIO}(\Rd)}\\
&\lesssim \|u_{0}\|_{\HT^{\alpha,p}_{FIO}(\R^d)}+\|u_{1}\|_{\HT^{\alpha-1,p}_{FIO}(\R^d)}+\|F\|_{L^{1}([0,1];\HT^{\alpha-1,p}_{FIO}(\R^d))}
\end{aligned}
\end{equation}
for $0\leq \tau\leq 1$, where we used Proposition \ref{prop:roughfixed} for the second inequality.

Clearly, the Fourier transforms of $\psi_{N}(D)L_{N}u(\tau)$ and $\psi_{N}(D)e_{N}(x,D)u(\tau)$ are supported in $\{\xi\in\Rd\mid N/2\leq |\xi|\leq 2N\}$, for each $\tau\in\R$. Moreover, given that $N^{1/2}\leq N/4$, the Fourier transform of $L_{N}\psi_{N}(D)u(\tau)$ is supported in $\{\xi\in\Rd\mid N/4\leq |\xi|\leq 4N\}$. Hence the Fourier transform of 
\[
H(\tau):=[\psi_{N}(D),L_{N}]u(\tau)+\psi_{N}(D)e_{N}(x,D)u(\tau)
\]
is also supported in $\{\xi\in\Rd\mid N/4\leq |\xi|\leq 4N\}$. 

We can now apply \eqref{eq:reductionbound}, which by Lemma \ref{lem:Ldecomp} \eqref{it:Ldecomp2} yields an analogous bound with a gain of one derivative for $t\mapsto B^{-1}\sn_{t-\tau}$. Combining this estimate with \eqref{eq:Duhamelreg} and with Minkowski's inequality, we can bound the contribution to \eqref{eq:uDuhamel} of the remaining term in \eqref{eq:Duhamelsplit}: 
\begin{align*}
&\Big(\int_{0}^{1}\Big\|\int_{0}^{t}B^{-1}\sn_{t-\tau}H(\tau)\ud\tau\Big\|_{L^{q}(\Rn)}^{q}\ud t\Big)^{1/q}\\
&\leq \Big(\int_{0}^{1}\Big(\int_{0}^{t}\|B^{-1}\sn_{t-\tau}H(\tau)\|_{L^{q}(\Rd)}\ud\tau\Big)^{q}\ud t\Big)^{1/q}\\
&\leq \int_{0}^{1}\Big(\int_{\tau}^{1}\|B^{-1}\sn_{t-\tau}H(\tau)\|_{L^{q}(\R^d)}^{q}\ud t\Big)^{1/q}\ud \tau\lesssim \int_{0}^{1}\|H(\tau)\|_{\HT^{\alpha-1,p}_{FIO}(\R^d)}\ud \tau\\
&\lesssim \|u_{0}\|_{\HT^{\alpha,p}_{FIO}(\R^d)}+\|u_{1}\|_{\HT^{\alpha-1,p}_{FIO}(\R^d)}+\|F\|_{L^{1}([0,1];\HT^{\alpha-1,p}_{FIO}(\R^d))},
\end{align*}
with the natural modification for $q=\infty$.
\end{proof}

\subsection{Wave packet decompositions and flow comparisons}\label{subsec:wavedecomp}

The rest of this section is concerned with proving \eqref{eq:reductionbound}. 

To this end, we require an isotropic wave packet decomposition which played a crucial role in the proof of the sharp bilinear estimates in \cite{SchippaTataru2025}. In this subsection we show that the relevant wave packet decomposition can be used in our present setting, which is not automatic because 
the symbol $b_{N}$ from \eqref{eq:defsymbolb} does not satisfy the required size and regularity estimates for $|\xi| \not\eqsim N$. However, we will argue that for frequency-localized initial data any contribution away from the initial localization is smoothing, cf.~\eqref{eq:ModifiedEnergyEstimate}.
Later we will also need to transfer the argument to prove bilinear estimates from homogeneous flows to asymptotically homogeneous flows. To do so, we will use that the resulting wave packets are indistinguishable, which is the content of Lemma \ref{lem:Comparison} below. 

Throughout, fix $N\in 2^{\N}$ large. 
All the implicit bounds will be independent of $N$.

\subsubsection{Rescaling and Hamiltonian flows}

We will rely on the wave packet decomposition provided by Theorem \ref{thm:WavePacketDecomposition} below. The  relevant 
wave packets are centered along bicharacteristics in phase space, $T^{*}\Rd=\R^{2d}$, and propagate along the Hamiltonian flow associated with a suitable symbol $a:\R^{2d}\to \R$:
\begin{equation}\label{eq:flow}
    \begin{cases}
        \dot{x}(t)\!&= \partial_{\xi} a(x(t),\xi(t)),  \\
        \dot{\xi}(t)\!&= - \partial_x a(x(t),\xi(t)),
    \end{cases}
\end{equation} 
for $t$ in a maximal time interval of existence. 

With frequencies of size approximately $N$ being distinguished, to enter the outset of \cite{SchippaTataru2025}, it is natural to rescale and set $\tilde{a}_N(x,\xi) := N^{-1} b_N(N^{-1} x, N \xi)$ for $(x,\xi)\in\R^{2d}$. 
Denote by $(\tilde{S}_{N}(t))_{t\in\R}$ the solution operators to the abstract Cauchy problem
\begin{equation}\label{eq:halfwaveatilde}
        \begin{cases}
           (D_t+ \tilde{a}_N(x,D))u(t,x)= 0,\\
            u(0,x) = u_{0}(x).
        \end{cases}
    \end{equation}
Write $\tilde{S}_{N}f(t,x):=\tilde{S}_{N}(t)f(x)$ for $f\in\Sw(\Rd)$ and $(t,x)\in\R\times \Rd$, and note that
\begin{equation}\label{eq:rescalingkey}
S_{N}f(t,x)=(\tilde{S}_{N}(Nt)f_{N})(Nx),
\end{equation}
where $f_{N}(y):=f(N^{-1}y)$ for $y\in\Rd$. Hence, to obtain estimates for $S_{N}f$ in the case where $\supp(\wh{f}\,)\subseteq\{\xi\in\Rd\mid |\xi|\eqsim N\}$, we may derive bounds for $\tilde{S}_{N}g$ if $\supp(\wh{g})\subseteq\{\xi\in\Rd\mid |\xi|\eqsim 1\}$. However, we also have to consider $0\leq t \leq N$. This motivates the following arguments.

Fix $c>1$. With $\psi_1 \in C^{\infty}_{c}(\Rd)$ as in Section \ref{subsec:spaces}, set
\begin{equation*}
    \tilde{\psi}_1(\xi):= \psi\big(\tfrac{1}{2c}\xi)-\psi(c\xi)
    \text{ and } a_N(x,\xi) := \tilde{a}_N(x,\xi) \tilde{\psi}_1(\xi)
\end{equation*}
for $(x,\xi)\in\R^{2d}$. Note that $\tilde{\psi}_{1}(\xi)=1$ if $\frac{1}{c}\leq |\xi|\leq c$, and $\tilde{\psi}_{1}(\xi)=0$ unless $\frac{1}{2c}< |\xi|< 2c$. Hence 
\[
a_N(x,\xi)\!=\!\Big(\!\sum_{i,j=1}^{d}\!\sum_{\frac{N}{2c}<M<8Nc}\!\psi(M^{-1/2}D)\psi(N^{-1/2}D)g_{ij}(N^{-1}x)\psi_{M}(N\xi)\xi_{i}\xi_{j}\!\Big)^{1/2}\tilde{\psi}_{1}(\xi).
\]
Keeping in mind that
\begin{equation}\label{eq:BoundsCoefficientsParadifferential}
\sup_{x\in\Rd}\big|\partial_x^{\alpha} \big(\psi(N^{-1/2}D)g_{ij}\big)(N^{-1}x)\big| \lesssim_{\alpha,\beta} N^{-|\alpha| + \frac{(|\alpha|-2)_+}{2}}
\end{equation}
for all $\alpha,\beta\in \N_{0}^{d}$, one obtains the following estimates:
\begin{equation}\label{eq:SizeRegularitypsharp}
\sup_{(x,\xi)\in\R^{2d}}|\partial_x^{\alpha} \partial_{\xi}^{\beta} a_{N}(x,\xi) | \lesssim_{\alpha,\beta} N^{-|\alpha| + \frac{(|\alpha|-2)_+}{2}}.
\end{equation}
These bounds need to be satisfied in order to directly appeal to \cite[Theorem 1.2]{SchippaTataru2025}.

However, simply replacing $\tilde{a}_N$ by $a_{N}$ and estimating the remainder with commutator estimates leads to difficulties outside of the $L^{2}$-based setting. 
Instead, we argue that the wave packet decomposition provided by \cite{SchippaTataru2025} 
is still at our disposal, albeit with a different Hamiltonian. To this end, we will use that 
 the relevant initial data is localized at frequencies $|\xi| \eqsim N$ 
and that the evolution $(S_{N}(t))_{t\in\R}$ generated by $i b_N(x,D)$ essentially maintains the frequency localization, as we now show. 
\begin{proposition}
\label{prop:ModifiedEnergy}
Let $m\geq0$. Then there exists a $C\geq0$ such that
\begin{equation}
\label{eq:ModifiedEnergyEstimate}
\| \psi_{M}(D)S_{N}(t)f \|_{L^{2}(\Rd)} \leq C \big(\max\big(\tfrac{N}{M},\tfrac{M}{N}\big)\big)^{-m} \|f\|_{L^{2}(\Rd)},
\end{equation}
for all $M\in 2^{\N_{0}}$, $0\leq t\leq 1$ and $f\in L^{2}(\Rd)$ such that $\supp(\wh{f}\,)\subseteq\{\xi\in\Rd\mid N/c\leq |\xi|\leq cN\}$. \end{proposition}
\begin{proof}
We employ the wave packet transform $W$ from \eqref{eq:defW}, which is an isometry from $L^{2}(\Rd)$ into $L^{2}(\R^{2d})$. Due to the latter property, which implies that $W^{*}W$ is the identity, it suffices to show that 
\[
\|W\psi_{M}(D)S_{N}(t)W^{*}F\|_{L^{2}(\R^{2d})} \lesssim_{m} \big(\max\big(\tfrac{N}{M},\tfrac{M}{N}\big)\big)^{-m} \|F\|_{L^{2}(\R^{2d})}
\]
for all $F\in L^{2}(\R^{2d})$. This is in turn a consequence of the kernel bounds from Proposition \ref{prop:kernelbounds}, combined with Schur's inequality. Here one also uses that 
\[
W\psi_{M}(D)S_{N}(t)W^{*}F(x,\xi)=\phi_{\xi}(D)\psi_{M}(D)S_{N}(t)W^{*}F(x)=0
\]
unless $|\xi|\eqsim M$, that $\psi_{M}(D)$ and $W$ commute, and that $\psi_{M}(D)$ also satisfies  suitable kernel bounds.
\end{proof}

When considering the smoothed and rescaled symbol $\tilde{a}_N$, we note that frequencies $|\xi| \eqsim 1 $ do not essentially change their size under the flow governed by \eqref{eq:flow}.

\begin{lemma}\label{lem:conservationclassical}
There exists an $M_{0}\geq1$ such that the following holds. 
Let $(x_{0},\xi_0)\in T^{*}\Rd$ be such that $\xi_{0}\in \supp(\tilde{\psi}_1)$. Then the bicharacteristic curve $t\mapsto (x(t),\xi(t))$ associated with $\tilde{a}_N$, and satisfying $(x(0),\xi(0))=(x_0,\xi_0)$, 
has the property that $M_{0}^{-1}\leq |\xi(t)|\leq M_{0}$ for all $t\in\R$.
\end{lemma}
\begin{proof}
    This is immediate from conservation of the Hamiltonian, the bounds on the symbol and the uniform ellipticity of the coefficients.
\end{proof}

In what follows, solely to ease notation, we will suppose that $c>1$ is sufficiently small that  
the flow in Lemma \ref{lem:conservationclassical} keeps frequencies $\xi_0 \in \text{supp}(\tilde{\psi}_1)$ in the set
\begin{equation*}
\{ \xi \in \R^d \mid \sum_{M=N/2C}^{2CN} \psi_{M}(N\xi)  = 1 \}.
\end{equation*}
For initial data frequency localized with $\tilde{\psi}_1(D)$, to analyze the propagation of singularities along the Hamiltonian flow associated with $\tilde{a}_N$, it suffices to instead consider the Hamiltonian flow associated by $a_N$, given that the phase space region $|\xi| \not\eqsim 1$ where the flows differ does not essentially contribute, by \eqref{eq:ModifiedEnergyEstimate}. This statement will be quantified in terms of the wave packet decomposition below.

\subsubsection{Isotropic wave packet decompositions}

We go over the wave packet construction in \cite{SchippaTataru2025} and explain how it extends to a different outset. For example, the position variable in \cite{SchippaTataru2025} was centered. Fortunately, this amounts to a harmless linear change of variables. A more meaningful difference with respect to \cite{SchippaTataru2025} is that we consider the evolution generated by a real-valued symbol in the standard quantization, as opposed to the Weyl quantization.

We construct isotropic wave packets using the modified FBI transform $T_N:L^2(\R^d) \to L^2(T^* \R^d)$, given by
\begin{equation*}
T_N f(x,\xi) := c_N \int_{\Rd} e^{i \xi (x-y)} e^{-\frac{\rho}{2}|x-y|^2} f(y) \ud y,
\end{equation*}
for $f \in L^2(\R^d)$ and $(x,\xi) \in \Tp$. Here $\rho:= N^{-1}$ and $c_N := N^{-\frac{d}{4}} 2^{-\frac{d}{2}} \pi^{-\frac{3d}{4}}$. Note that $T_{N}$ is an isometry. In terms of the coherent states given by
\begin{equation*}
\ph_{x,\xi}(y):= e^{-i \xi (x-y)} e^{-\frac{\rho}{2} |x-y|^2}
\end{equation*}
for $y\in\Rd$, one clearly has
\begin{equation*}
T_N f(x,\xi) = c_N \int_{\R^d} \overline{\ph_{x,\xi}(y)} f(y) \ud y. 
\end{equation*}
Moreover, 
\[
T_{N}^{*}F(y):=c_{N}\int_{\R^{2d}}\ph_{x,\xi}(y)F(x,\xi)\ud x\ud\xi
\]
for $F\in L^{2}(\R^{2d})$ and $y\in\Rd$, defines an adjoint of $T_{N}$. 

The base points of our wave packets will be placed on the lattice $\Lambda_N:= N^{\frac{1}{2}} \Z^d \times N^{-\frac{1}{2}} \Z^d$ in phase space. Consider a partition of unity $(\psi_{(x_0,\xi_0)})_{(x_0,\xi_0) \in \Lambda_N}\subseteq C^{\infty}(\R^{2d})$ associated with this 
lattice. 
That is, one has $\supp(\psi_{(x_{0},\xi_{0})})\subseteq B_{d}(x_{0},3N^{1/2})\times B_{d}(\xi_{0},3N^{-1/2})$ for all $(x_{0},\xi_{0})\in\Lambda_{N}$, and
$\sum_{(x_0,\xi_0) \in \Lambda_N} \psi_{(x_0,\xi_0)}(x,\xi)= 1$ 
for all $(x,\xi)\in\Tp$, and $\psi_{(x_{0},\xi_{0})}$ and their derivatives satisfy appropriate bounds, independent of $(x_{0},\xi_{0})$ and $N$.

We also recall the bi-Lipschitz property of the Hamiltonian flow, which in fact motivates the shape of the support of the $\psi_{(x_{0},\xi_{0})}$ above. 
Define a rescaled isotropic metric on phase space by
\begin{equation*}
    d_N((x,\xi),(y,\eta)):= N^{-\frac{1}{2}}|x-y| + N^{\frac{1}{2}} |\xi-\eta|,
\end{equation*}
for $(x,\xi),(y,\eta)\in T^{*}\Rd$. The Hamiltonian flow $(\chi_t)_{t \in \R}$ generated by $\tilde{a}_N$ satisfies the bi-Lipschitz bounds for $|t| \leq N$ for points $(x_{1},\xi_{1}),(x_2,\xi_2) \in T^* \R^d$ with
\begin{equation*}
    d_N(\chi_t(x_1,\xi_1),\chi_t(x_2,\xi_2)) \eqsim d_N((x_1,\xi_1),(x_2,\xi_2)).
\end{equation*}
This is proved in \cite[Lemma~2.2]{SchippaTataru2025}, relying on the $C^{1,1}$ regularity of the coefficients. Under the assumption \eqref{eq:AbsenceCaustics}, the implicit constants can be taken as uniform.

Finally, we recall the size and regularity assumptions on regions in phase space which are permissible for phase space concentration.
Let $\nu \in [N^{-\frac{1}{2}},1]$ be an aperture parameter, and 
let $u_0 \in L^2(\R^d)$ and $(x_*,\xi_{*}) \in \Sp$. We define smooth localizing functions: 
\begin{equation*}
    \theta_{x_*}(x):= \theta((\nu N)^{-1}(x-x_*)), \quad \psi_{\xi_*}(\xi):= \psi( \nu^{-1}( \hat{\xi} - \xi_*)) \tilde{\psi}_1(\xi)
\end{equation*} 
for $x\in\Rd$. We require that $\supp(\widehat{\theta}) \subseteq B_d(0,c)$, and $|\theta(x)| \lesssim_{m} (1+|x|)^{-m}$ for all $x\in\Rd$. Choosing a spatial localization with compact Fourier support will be a minor simplification (see Proposition \ref{prop:reduction}). We refer to Subsection \ref{subsection:KernelBoundsPhaseSpace} for more details. Let $\theta_1 \in \mathcal{S}(\R)$ be a version in one dimension.

Then 
\[
u_{x_*,\xi_*}(x):= \theta_{x_*}(x) \psi_{\xi_*}(D) u_0
\]
is concentrated in phase space on the region
\begin{equation*}
    \mathcal{Y}_{x_*,\xi_*}:= \{(x,\xi) \in \R^{2d} \mid x \in B_d(x_*,\nu N), \, \tfrac{1}{2c}\leq |\xi| \leq 2c, \, |\hat{\xi}-\xi_*| \leq \nu \}
\end{equation*}
a statement which will be quantified in Theorem \ref{thm:WavePacketDecomposition} below. 
We thicken $\mathcal{Y}_{x_*,\xi_*}$ 
as follows:
\begin{equation*}
\overline{\mathcal{Y}}_{x_*,\xi_*} = \{(x,\xi) \in \R^{2d} \mid |x| < \ka N \nu, \tfrac{1}{\ka2c} \leq |\xi| \leq \ka 2c, \, |\hat{\xi}-\xi_{*}| < \ka\nu \},
\end{equation*}
using a large constant $\ka\geq 1$.

We can now construct the wave packet decomposition, as in \cite{SchippaTataru2025}. 
Recall that $(\tilde{S}_{N}(t))_{t\in\R}$ solves \eqref{eq:halfwaveatilde}, and write
\begin{equation}\label{eq:packetdecomp}
\begin{split}
&\tilde{S}_N(t) T_N^* T_N u_{x_*,\xi_*} = \tilde{S}_N(t) T_N^*\Big( \sum_{(x_0,\xi_0)\in\Lambda_{N}} \psi_{x_0,\xi_0} (T_N u_{x_*,\xi_*})\Big) \\
&= \sum_{(x_0,\xi_0) \in \Lambda_N} c_N \int_{\Tp} (\tilde{S}_N(t) \ph_{x,\xi}) \psi_{(x_0,\xi_0)}(x,\xi) T_N u_{x_*,\xi_*}(x,\xi) \ud x \ud\xi
\end{split}
\end{equation} 
for $t\in\R$. 
Let
\begin{equation*}
    u_{(x_0,\xi_0)}(t) = c_N \int_{\Tp} (\tilde{S}_N(t) \ph_{x,\xi}) \psi_{(x_0,\xi_0)}(x,\xi) T_N u_{x_*,\xi_*}(x,\xi) \ud x \ud\xi.
\end{equation*}
For $(x_{0},\xi_{0})\in \Lambda_{N}\cap \overline{\mathcal{Y}}_{x_*,\xi_*}$ and $y\in\R^{d}$, set
\begin{align*}
\alpha_{(x_{0},\xi_{0})}&:= \| u_{(x_0,\xi_0)}(0) \|_{L^{2}(\Rd)} ,\\
\phi_{(x_{0},\xi_{0})}(t,y)&:= \| u_{(x_0,\xi_0)}(0) \|_{L^{2}(\Rd)}^{-1}u_{(x_0,\xi_0)}(t,y) .
\end{align*}
Moreover,
\[
g(t,y):=\tilde{S}_N(t) T_N^* T_N u_{x_*,\xi_*}-\sum_{(x_{0},\xi_{0})\in \Lambda_{N}\cap \overline{\mathcal{Y}}_{x_*,\xi_*}}\alpha_{(x_{0},\xi_{0})}\phi_{(x_{0},\xi_{0})}(t,y)
\]
collects the remaining terms in the sum in \eqref{eq:packetdecomp}.

This decomposition has the following properties.

\begin{theorem}
\label{thm:WavePacketDecomposition}
Let 
$\nu \in [N^{-\frac{1}{2}},1]$, $u_0 \in L^2(\R^d)$ and $(x_*,\xi_{*}) \in \Sp$ be given.  
Then there exist sequences $(\alpha_{(x_{0},\xi_{0})})_{(x_{0},\xi_{0})\in \Lambda_N \cap \overline{\mathcal{Y}}_{x_*,\xi_{*}}}\subseteq \C$ and \\ $(\phi_{(x_{0},\xi_{0})})_{(x_{0},\xi_{0})\in\Lambda_N \cap \overline{\mathcal{Y}}_{x_*,\xi_{*}}}\subseteq C(\R;L^{2}(\Rd))$, as well as a $g\in C(\R;L^{p}(\Rd))$, such that the following holds. For all $t\in[-N,N]$ and $x\in\Rd$, one has
\begin{equation*}
\tilde{S}_{N}(t)u_{x_*,\xi_*}(x) = \sum_{(x_0,\xi_0) \in \Lambda_N \cap \overline{\mathcal{Y}}_{\alpha,N}} \alpha_{(x_0,\xi_0)} \phi_{(x_0,\xi_0)}(t,x) + g(t,x).
\end{equation*}
There exists a $C\geq0$, independent of 
$\nu$, $u_{0}$ and $(x_{*},\xi_{*})$, such that  
\[
\sum_{(x_0,\xi_0) \in \Lambda_N \cap \overline{\mathcal{Y}}_{x_*,\xi_{*}}} |\alpha_{x_0,\xi_0}|^2 \leq C \| u_0 \|_{L^2(\Rd)}^2.
\]
For each $m\geq0$ there exists a $C_{m}\geq0$, independent of 
$\nu$, $u_{0}$ and $(x_{*},\xi_{*})$, such that  the following bounds hold:
\begin{align*}
    |\phi_{(x_0,\xi_0)}(t,y)| &\leq\!C_{m} N^{-\frac{d}{4}} (1+N^{-\frac{1}{2}} |x_0(t) - y|)^{-m},\\
    |\F\phi_{(x_0,\xi_0)}(t,\cdot)(\eta)| &\leq\!C_{m} N^{\frac{d}{4}} (1+N^{\frac{1}{2}} |\xi_0(t) - \eta|)^{-m},\\
\big|\mathcal{F}\big(\theta_1(N^{\frac{1}{2}}(\cdot-t_0)) \phi_{(x_0,\xi_0)}(\cdot,y)\big)(\tau)\big| &\leq\!C_{m} N^{\frac{1}{2}-\frac{d}{4}} (1+N^{\frac{1}{2}}|\tau  + \tilde{a}_N(x_{0}(t_0),\xi_{0}(t_0))|)^{-m},\\
\| g(t,\cdot) \|_{L^p(\Rd)} &\leq\!C_m 
(N^{\frac{1}{2}} \nu)^{-m} \nu^{(d-1)(\frac{1}{2}-\frac{1}{p})} \| u_0 \|_{L^2(\Rd)},
\end{align*}
for all $t,t_{0}\in[-N,N]$, $y,\eta\in\Rd$, $\tau\in\R$ and $(x_{0},\xi_{0})\in \Lambda_N \cap \overline{\mathcal{Y}}_{x_*,\xi_{*}}$. Here $t\mapsto (x_{0}(t),\xi_{0}(t))$ is the bicharacteristic curve associated with $\tilde{a}_{N}$ and satisfying $(x_{0}(0),\xi_{0}(0))=(x_{0},\xi_{0})$. 
\end{theorem}

\begin{proof}
The second sum comprises the remainder term $g$, whose bounds we find by conjugating the pseudo-differential operator $\psi_{x_*,\xi_*}(x,D) = \theta_{x_*}(x) \psi_{\xi_{*}}(D)$ with the FBI transform to phase space. The rapid decay is obtained by 
\begin{equation*}
\begin{split}
    &\quad (\bigcup_{(x_0,\xi_0) \in \Lambda_N \cap \overline{\mathcal{Y}}_{x_*,\xi_{*}}^c} \text{supp}(\psi_{(x_0,\xi_0)}) + B(0,N^{\frac{1}{2}+\delta}) \times B(0,N^{-\frac{1}{2}+\delta})) \\
    &\cap \text{supp}(\theta_{x_*}(x) \psi_{\xi_{*}}(\xi))= \emptyset,
    \end{split}
\end{equation*}
and invoking conjugation formulae (cf. \cite[Theorem~5]{Tataru2001}) for pseudo-differential operators with respect to the FBI transform.

We turn to the first term, which comprises the wave packets. The regularity properties are obtained from weighted $L^2$-estimates for the Green's function:
\begin{equation*}
G(t,x,\xi,y) = c_N^2 e^{-i \xi^t (y-x^t)} e^{-i \psi(x,t,\xi)} (\tilde{S}_N(t) \phi_{x,\xi})(y).
\end{equation*}
These estimates depend on:
\begin{enumerate}
\item the regularity properties of the Hamiltonian flow generated by $\tilde{a}_N$, which are like in \cite{SchippaTataru2025}, upon considering the phase space region $|\xi| \eqsim 1$, to which the relevant bicharacteristics are confined,
\item the size and regularity estimates of the symbol $\tilde{a}_N$ \cite[Eq.~(4.44)]{SchippaTataru2025}.
\end{enumerate}
With these at hand, the proof of the weighted $L^2$-estimates of the Green function are in reach; cf. \cite[Eq.~(4.46)]{SchippaTataru2025}:
\[
\begin{aligned}
    &\| (N^{-\frac{1}{2}} (x^t - y))^{\gamma} \partial_t^{\sigma} \partial_x^{\alpha} \partial_{\xi}^{\beta} \partial_y^{\nu} G(t,s,x,\xi,y) \|_{L^2_y} \\
    &\leq C_{\nu,\alpha,\beta,\gamma} N^{- \frac{|\alpha|-|\beta|+|\nu|+|\sigma|}{2}} C_N^2 \| \phi_{x,\xi} \|_2
\end{aligned}
\]
for $|t-s| \leq N$. By the regularity of the Hamiltonian flow, it suffices to consider
\begin{equation*}
    G_1(t,x,\xi,y) = C_N^2 e^{-i \xi_0^t \cdot (y-x^t_0)} e^{-i \psi(x_0,t,\xi_0)} (S(t,0) e^{i \xi_0 (x-x_0)} \phi_{x,\xi})(y).
\end{equation*}
$G_1$ is translated to the origin by
\begin{equation*}
    G_2(t,x,\xi,y) = G_1(t,x_0+x,\xi_0+\xi,x_0^t+y).
\end{equation*}
We have the following (cf. \cite[Lemma~4.13]{SchippaTataru2025}), which is straight-forward:
\begin{lemma}
    $G_2$ solves
    \begin{equation*}
        (D_t + a_2(y,t,D_y)) G_2 = 0, \quad G_2(0) = C_N^2 \phi_{x,\xi},
    \end{equation*}
    where
    \begin{equation*}
        a_2(y,t,\eta) = a(x_0^t +y,t,\xi_0^t + \eta) - a(x_0^t,t,\xi_0^t) - y a_x(x_0^t,t,\xi_0^t) - \eta a_{\xi}(x_0^t,t,\xi_0^t).
    \end{equation*}
\end{lemma}
Indeed, centering the Green function in phase space along the bicharacteristic, does not depend on the quantization. So, we can turn to energy estimates for $L^2$-normalized Schwartz functions
\begin{equation*}
    (D_t + a_2(y,t,D_y))v = 0, \quad v(0) = v_0.
\end{equation*}
The weighted energy estimates are obtained by Gr\o nwall arguments like in \cite{SchippaTataru2025}. 
We prove
\begin{align*}
&\| (N^{\frac{1}{2}} \partial_t )^{\gamma} (N^{-\frac{1}{2}} y)^{\alpha} (N^{\frac{1}{2}} \partial_y)^{\beta} v(t) \|_2\\
&\lesssim \sum_{|\alpha'|+|\beta'| \leq |\alpha| + |\beta| + 2 |\gamma|} \| (N^{-\frac{1}{2}} y)^{\alpha'} (N^{\frac{1}{2}} \partial_y)^{\beta'} v(0) \|_2.
\end{align*}
We start with the estimates for $\gamma = 0$, which are obtained by induction on $|\alpha| + |\beta| $. For $|\alpha| + |\beta| = 0$ we note that
\begin{equation*}
    \frac{d}{dt} \|v(t) \|_2^2 = \langle i (a_2(y,D_y) - a_2^*(y,D_y) ) v, v \rangle.
\end{equation*}
Recall the asymptotic expansion of the symbol of $a_2^*$ (cf. \cite[Theorem~3.13]{Sogge2017}):
\begin{equation*}
    a_2^*(y,t,\xi) = \bar{a}_2 + \sum_{|\alpha| \geq 1} \frac{1}{\alpha!} D_{\xi}^{\alpha} \partial_x^{\alpha} \bar{a}_2(y,t,\xi)
\end{equation*}
From the size and regularity of $a_2$ we find
\begin{equation*}
    \langle (a_2(y,t,D_y) - a_2^*(y,t,D_y))v, v \rangle \lesssim \frac{\epsilon_0}{N} \| v(t) \|_2^2.
\end{equation*}
Consequently, for $|t| \leq N$,
\begin{equation*}
    \| v(t) \|_2^2 \lesssim \| v(0) \|_2^2.
\end{equation*}
For $|\alpha| + |\beta| = 1$ we let $A(t) = \| N^{-\frac{1}{2}} yv(t) \|_2^2 + \| N^{\frac{1}{2}} \partial_y v \|_2^2 + \| v(t) \|_2^2$ and control $A(t)$ by Gr\o nwall's inequality. Precisely, we shall prove
\begin{align}
\label{eq:EnergyEstimateA} \| (D_t + a_2(y,D_y)) (N^{-\frac{1}{2}} y v) \|_2 &\lesssim N^{-1} A(t),  \\
\label{eq:EnergyEstimateB} \| (D_t + a_2(y,D_y)) (N^{\frac{1}{2}} y v) \|_2 &\lesssim N^{-1} A(t).
\end{align}
With these at hand, it is straightforward
\begin{equation*}
    \begin{split}
    \frac{d}{dt} \| N^{-\frac{1}{2}} y v(t) \|_2^2 &= 2 \Re \langle N^{-\frac{1}{2}} y v(t) , \big( \frac{d}{dt} + i a_2^w(y,D)) (N^{-\frac{1}{2}} y v) \rangle \\
    &= 2 \Re ( \langle N^{-\frac{1}{2}} y v(t), \big( \frac{d}{dt} + i a_2(y,D)+ i b_2(y,D) \big) (N^{-\frac{1}{2}} y v) \rangle \big) \\
    &\lesssim C N^{-1} A(t) + N^{-1} \| N^{-\frac{1}{2}} y v(t) \|_2^2.
    \end{split}
\end{equation*}
The second bound follows from comparing standard and Weyl quantization like above. Similarly, the bound follows
\begin{equation*}
    \frac{d}{dt} \| N^{-\frac{1}{2}} \partial_y v(t) \|_2^2 \lesssim N^{-1} A(t).
\end{equation*}
For the proof of \eqref{eq:EnergyEstimateA} and \eqref{eq:EnergyEstimateB} we compute
\begin{equation}
\label{eq:EnergyCommutatorEstimate}
\begin{split}
    (D_t + a_2(y,D_y)) (N^{-\frac{1}{2}} y v) &= - i \partial_{\eta} a_2(y,D_y) N^{-\frac{1}{2}} v, \\
    (D_t + a_2(y,D_y)) (N^{-\frac{1}{2}} \partial_y v) &= i N^{\frac{1}{2}} \partial_y a_2(y,D_y) v.
    \end{split}
\end{equation}
The symbols on the right-hand side are favorably estimated by the symbol $a_2$ being quadratic in $(y,\eta)$. The estimates in \cite{SchippaTataru2025} are already obtained in standard quantization. This concludes the case $|\alpha| + |\beta| = 1$. 

For the induction step, we modify the argument in \cite{SchippaTataru2025} detailed in Weyl quantization like before. The key tools are commutator estimates like in \eqref{eq:EnergyCommutatorEstimate} and comparing Weyl and standard quantization. The terms differening from the ones in Weyl quantization are of lower order and can readily be absorbed into the Gr\o nwall argument.
We finally remark that, although the size and regularity estimates for $\tilde{a}_N$ are only valid for $|\xi| \eqsim 1$, this suffices by the essential preservation of frequencies - which follows again from Proposition \ref{prop:ModifiedEnergy}.
\end{proof}

\subsubsection{Comparing the homogeneous and asymptotically homogeneous flow}
\label{subsection:Comparison}
Next, we will argue that the evolution $(\tilde{S}_{N}(t))_{t\in\R}$ is indistinguishable on the level of wave packet decompositions from the one generated by the simplified symbol 
\begin{equation}\label{eq:aNh}
a_N^{h}(x,\xi) := \Big(\sum_{i,j=1}^{d}\psi(N^{-1/2}D)g_{ij}(N^{-1} x) \xi_i \xi_j\Big)^{1/2} \tilde{\psi}_{1}(\xi),
\end{equation}
for $(x,\xi)\in\R^{2d}$. Let $g'_{ij}:= \psi(N^{-1/2}D) g_{ij}(N^{-1} \cdot)$ for brevity and $g'_{ij, \leq M^{1/2}} = \psi(M^{-1/2} D) g'_{ij}$. This will allow us to use results 
which were proved in \cite{SchippaTataru2025} for $a_N^h$, by passing through wave packet decompositions. 

We start off by proving the following lemma, using Taylor expansion. We use the convention $\text{sgn}(x) = 1$ for $x\in(0,\infty)$ and $\text{sgn}(0) = 0$.

\begin{lemma}
\label{lem:DifferenceEstimates}
Let $\alpha,\beta\in \N_{0}^{d}$. Then there exists a $C_{\alpha,\beta}\geq0$ such that
\begin{equation}
\label{eq:FlowComparison}
|\partial_x^{\alpha} \partial_{\xi}^{\beta}(a_{N}^{h}-a_{N})(x,\xi)| \leq C_{\alpha,\beta} \epsilon_0^{\text{sgn}(|\alpha|)} N^{-1-\frac{|\alpha|}{2}}.
\end{equation}
for all 
$(x,\xi)\in\Rd$ with $1/(2C) \leq |\xi|\leq 2C$. 
\end{lemma}
\begin{proof}
We can write
\begin{equation*}
a_N^2 = \big( g'_{ij} \xi_i \xi_j + \sum_{M=N/2C}^{2CN} \psi_M(N \xi) \xi_i \xi_j \cdot r_M \big) \tilde{\psi}_c^2(\xi)
\end{equation*}
with $s_M = g'_{ij,\leq M^{1/2}} - g'_{ij}$. Consequently, under the assumptions on $\xi$,
\begin{equation*}
   a_N - (g'_{ij} \xi_i \xi_j)^{\frac{1}{2}} = \frac{\sum_{M=N/2C}^{2CN} \psi_M(N \xi) r_m \xi_i \xi_j}{(a_N^h + (g'_{ij} \xi_i \xi_j)^{1/2}) }.
\end{equation*}
By Young's inequality and \eqref{eq:AbsenceCaustics} we find for $M \sim N$
\begin{equation}
\label{eq:CoefficientEstimate}
\| g'_{ij,M^{1/2}} \|_{\infty} \lesssim \epsilon_0 M^{-1}.
\end{equation}
and similarly $\| s_M \|_{\infty} \lesssim \epsilon_0 M^{-1}$.

We let based on our assumption on $\xi$:
\begin{equation*}
\begin{split}
a_N(x,\xi) = a_N^{h}(x,\xi) + r_N(x,\xi)
\end{split}
\end{equation*}
with $|\partial_{\xi}^{\beta} r_N| \lesssim_{\beta} N^{-1}$ by  \eqref{eq:CoefficientEstimate} and the ellipticity of the denominator for normalized frequencies.

Next, we estimate that 
\begin{equation*}
|\partial_x \partial_{\xi}^{\beta} r_N(x,\xi)| \lesssim_{\beta} \epsilon_0 N^{-\frac{3}{2}}.
\end{equation*}
This is immediate from 
\begin{equation*}
|\partial_x g'_{ij,M^{\frac{1}{2}}}| \lesssim N^{-1} | (\partial_x g_{ij,M^{\frac{1}{2}}})(M^{-1} x)| \lesssim \epsilon_0 N^{-\frac{3}{2}}.
\end{equation*}
For $|\alpha| = 2$, \eqref{eq:FlowComparison} follows from \eqref{eq:BoundsCoefficientsParadifferential},  \eqref{eq:CoefficientEstimate}, and the ellipticity of the denominator for normalized frequencies.
\end{proof}

We turn to comparing the flows of $a_N^h$ and $a_{N}$. The chief result is that on the level of wave packet decompositions the flows are indistinguishable. We note that, upon assuming $\frac{1}{2c} \leq |\xi| \leq 2c$, we find that the frequencies under the Hamiltonian flow (see below) satisfy $\frac{1}{2C} \leq |\xi^t| \leq 2C$.

\begin{lemma}
\label{lem:Comparison}
The flows generated by $a_N^h$ and $a_N$ are indistinguishable for normalized frequencies in the following sense.
Let $t\mapsto (x_{i}(t),\xi_i(t))$ be bicharacteristics governed by $a_N^h$ or $a_N$ with $(x_1^0,\xi_1^0) = (x_2^0,\xi_2^0)$ and
with $\xi_i(t)$ satisfying
\begin{equation*}
\frac{1}{2C} \leq |\xi_i(t)| \leq 2C \text{ for } i \in \{ 1, 2 \}.
\end{equation*}
Then it holds for $0 \leq |t| \leq N$:
\begin{equation}
\label{eq:SmallFlowDifference}
|x_1(t) - x_2(t) | \leq 5  N^{\frac{1}{2}}, \quad |\xi_1(t) - \xi_2(t)| \leq 10 \epsilon_0 N^{-\frac{1}{2}}.
\end{equation}
\end{lemma}
\begin{proof}
This will follow from bootstrapping \eqref{eq:SmallFlowDifference} for $0 \leq |t| \leq N$, noting that for small times $|t| \leq d$ \eqref{eq:SmallFlowDifference} follows from continuity.

To close the bootstrap, we consider the difference in integrated form:
\begin{equation*}
\left\{\!\begin{array}{cl}
x_1(t) - x_2(t) &= \int_0^t \partial_{\xi} a_N(x_1(s),\xi_1(s)) - \partial_{\xi} a_N^h(x_2(s),\xi_2(s)) \ud s, \\
\xi_1(t) - \xi_2(t) &= \int_0^t (-\partial_x a_N(x_1(s),\xi_1(s)) + \partial_x a_N^h(x_2(s),\xi_2(s))) \ud s.
\end{array} \right.
\end{equation*}
We compare derivatives of $a_N$ to derivatives of $a_N^h$, invoking Lemma \ref{lem:DifferenceEstimates}:
\begin{equation*}
\left\{\!\begin{array}{cl}
x_1(t) - x_2(t) &= \int_0^t \partial_{\xi} a_N^{h}(x_1(s),\xi_1(s)) - \partial_{\xi} a_N^h(x_2(s),\xi_2(s)) + s_N \ud s, \\
\xi_1(t) - \xi_2(t) &= \int_0^t (-\partial_x a_N^h(x_1(s),\xi_1(s)) + \partial_x a_N^h(x_2(s),\xi_2(s))) + t_N \ud s
\end{array} \right.
\end{equation*}
with bounds $|s_N| \lesssim N^{-1}$ and $|t_N| \lesssim \epsilon_0 N^{-\frac{3}{2}}$, for which reason the integrated error is acceptable. Now we invoke the mean-value theorem to write
\begin{equation*}
\left\{ \begin{array}{cl}
x_1(t) - x_2(t) &= \int_0^t ( \partial^2_{x \xi} a_N^h \cdot (x_1(s) - x_2(s)) + \partial^2_{\xi} a_N^h \cdot (\xi_1(s)-\xi_2(s)) + s_N) \ud s, \\
\xi_1(t) - \xi_2(t) &= \int_0^t (\partial^2_{x} a_N^h \cdot (x_1(s) - x_2(s)) + \partial^2_{x \xi} a_N^h \cdot (\xi_1(s) - \xi_2(s)) + t_N) \ud s.
\end{array} \right.
\end{equation*}
Suppose that the bounds \eqref{eq:SmallFlowDifference} hold for $|t| \leq C^* < N$. By continuity we can suppose that for $c > 0$ we have the bounds for $0 \leq |t| \leq C^* + c \leq N$ 
\begin{equation*}
|x_1(t) - x_2(t) | \leq 8 N^{\frac{1}{2}}, \quad |\xi_1(t) -\xi_2(t) | \leq 20 \epsilon_0 N^{-\frac{1}{2}}.
\end{equation*}
By the bootstrap assumption we find
\begin{equation*}
\begin{split}
|x_1(t) - x_2(t)| &\leq |t| ( 8 \epsilon_0 N^{-\frac{1}{2}} + 20 \epsilon_0 N^{-\frac{1}{2}} + D N^{-1} ) \leq 5 N^{\frac{1}{2}}, \\
|\xi_1(t) - \xi_2(t)| &\leq |t| (\epsilon_0 N^{-2} \cdot 8 N^{\frac{1}{2}} + \epsilon_0^2 N^{-1} \cdot 20 N^{-\frac{1}{2}} + \epsilon_0 N^{-\frac{3}{2}} ) \leq 10 \epsilon_0 N^{-\frac{1}{2}},
\end{split}
\end{equation*}
which finishes the proof.
\end{proof}

\subsection{Bilinear restriction estimates}\label{subsec:bilinear}

In this subsection we prove bilinear restriction estimates involving the Hardy spaces for Fourier integral operators, after first extending the $L^{2}$-based estimates for homogeneous symbols from \cite{SchippaTataru2025} to the present setting. Throughout, fix $N\in 2^{\N}$ large and $c_{0},c_{1},c_{2},c_{3}>0$. Let 
\begin{equation}\label{eq:thetazero}
\Theta_{0}:=\{\w\in S^{n-1}\mid |\w-e_{d}|\leq c_{0}\}
\end{equation}
be a small spherical cap. 
Let $\phi \in C^\infty_c(\R)$ be such that $\phi(\xi) =1$ for $|\xi| \leq 1$, and such that $\supp(\phi) \subseteq [-2,2]$. Set
\begin{equation}
\label{eq:AngularLocalization}
    \psi^{\nu}_{\omega}(\xi) = \phi(\nu^{-1}|\hat{\xi}-\omega|)
\end{equation}
for $\omega \in S^{d-1}$.

The following proposition will be used for the case $q=\infty$ in the main result of this subsection, and for the bilinear terms which involve no angular separation. 

\begin{proposition}\label{prop:bilinearinfty}
Let $1\leq p\leq q\leq \infty$. 
Then there exists a $C\geq0$ such that the following holds. Let $\Theta\subseteq S^{d-1}$ be a spherical cap of angular width $c_{2}N^{-1/2}$. Let $f\in\Sw(\Rd)$ be such that
\begin{equation}\label{eq:bilinearinftyassume}
\supp(\wh{f}\,)\subseteq\{\xi\in\Rd\mid N/c_{3}\leq |\xi|\leq c_{3}N,\hat{\xi}\in \Theta\}.
\end{equation}
Then
\[
\sup_{0\leq t\leq 1}\| S_{N}(t)f\|_{L^{q}(\R^d)} \leq C  N^{\frac{d+1}{2}(\frac{1}{p}-\frac{1}{q})+\frac{d-1}{2}(\frac{1}{p}-\frac{1}{2})}\|f\|_{\HT^{p}_{FIO}(\Rd)}.
\]
\end{proposition}
\begin{proof}
For $q\leq 2$, we can apply \eqref{eq:Sobolev}, Lemma \ref{lem:fixedtimesmooth}, Lemma \ref{lem:Knapp}, Bernstein's inequality, and then Lemma \ref{lem:Knapp} again:
\begin{align*}
\sup_{0\leq t\leq 1}\| S_{N}f\|_{L^{q}(\R^d)}&\lesssim \sup_{0\leq t\leq 1}\| S_{N}(t)f\|_{\HT^{s(q),q}_{FIO}(\Rd)}\lesssim \|f\|_{\HT^{s(q),q}_{FIO}(\Rd)}\eqsim \|f\|_{\HT^{q}(\Rd)}\\
&\lesssim N^{\frac{d+1}{2}(\frac{1}{p}-\frac{1}{q})}\|f\|_{\HT^{p}(\Rd)}\eqsim N^{\frac{d+1}{2}(\frac{1}{p}-\frac{1}{q})+\frac{d-1}{2}(\frac{1}{p}-\frac{1}{2})}\|f\|_{\HT^{p}_{FIO}(\Rd)},
\end{align*}
where in the last step we also used that $p\leq 2$.

For $q>2$, due to the same combination of Bernstein's inequality and Lemma \ref{lem:Knapp} as above, we may
consider the case where $q=p$. In fact, by interpolation and Lemma \ref{lem:Knapp}, it suffices to show that
\begin{equation}\label{eq:bilinearinftyinfty}
\sup_{0\leq t\leq 1}\|S_{N}(t)\psi_{\omega}^{\nu}(D) \psi_N(D) f\|_{L^{\infty}(\Rd)}\lesssim \|f\|_{L^{\infty}(\Rd)}
\end{equation}
for all $f\in\Sw(\Rd)$ and $\w\in S^{d-1}$, where $\nu:=N^{-1/2}$. 

To prove \eqref{eq:bilinearinftyinfty}, we rely on the kernel bounds from Appendix \ref{sec:kernel}, which are in fact implicit in the reasoning for $q\leq 2$ as well.  
For $t\in[0,1]$ fixed, let $\tilde{K}$ the kernel of $S_{N}(t)\psi_{\omega}^{\nu}(D) \psi_N(D)$. By Proposition \ref{prop:KernelEstimatesFrequencyLocalization} and \eqref{eq:metric}, for all $m\geq 0$ and $x,y\in\Rd$,
\begin{align*}
|\tilde{K}(x,y)|&\lesssim_{m} N^{\frac{d+1}{2}}\big(1+N\text{dist}_S^2((x,\omega(t)), \hat{\chi}_t(\hat{\mathcal{Y}}_{y,\omega}))\big)^{-m}\\
&\eqsim N^{\frac{d+1}{2}}\big(1+Nd_S((x,\omega(t)), \hat{\chi}_t(y,\omega))^{2}\big)^{-m}\\
&\eqsim N^{\frac{d+1}{2}}\big(1+N(|x-y|^{2}+|(x-y)\cdot\omega(t)|)\big)^{-m}.
\end{align*}
Hence $\sup_{x\in\Rd}\|\tilde{K}(x,\cdot)\|_{L^{1}(\Rd)}<\infty$ and $\sup_{y\in\Rd}\|\tilde{K}(\cdot,y)\|_{L^{1}(\Rd)}<\infty$. By Young's inequality, this suffices.

\end{proof}

In Theorem \ref{thm:bilinearL2} we give an $L^{2}$-based bilinear restriction estimate for asymptotically $1$-homogeneous symbols. We write $S_{N}f(t,x):=S_{N}(t)f(x)$ for $f\in\Sw(\Rd)$ and $(t,x)\in\R^{1+d}$, where  $(\tilde{S}_N(t))_{t\in\R}$ solves \eqref{eq:halfwaveatilde}. 

It will be useful to observe that bilinear estimates inherit a natural spatial localization by the finite speed of propagation. Localizing functions $\theta$ and rescalings $\tilde{\theta}_{K,k}$ with compact Fourier support are introduced in the Appendix \ref{subsection:KernelBoundsPhaseSpace}.

For $K \in 2^{\N_{0}}$, $\lambda=(\lambda_{1},\ldots,\lambda_{d})\in \Z^{d}$ and $x\in\Rd$, set $\theta_{K,\lambda}(x) := \theta(K(x-\lambda))$, and write 
\begin{equation*}
\tilde{\theta}_{K,k}(x) := \sum_{\substack{\la=(\la_{1},\ldots,\la_{d}) \in K^{-1} \Z^d, \\ \la_i \in (k_i-1/2,k_i+1/2]\text{ for }1\leq i\leq d}} \theta_{K,\lambda}^2(x)
\end{equation*}
for $k \in\Z^{d}$. Then 
$\sum_{k \in \Z^d} \tilde{\theta}_{K,k}(x)=1$ 
for all $x\in\Rd$, and $\supp(\F\tilde{\theta}_{K,k})\subseteq B_{d}(0,K/4)$ for all $k \in \Z^{d}$. Note that $\tilde{\theta}_{K,k} \geq 0$. Let $Q_{k} = k + (-1/2,1/2]^d$.
Moreover, for each $m\geq 0$, one has $|\tilde{\theta}_{N,k}(x)|\lesssim (1+K \text{dist}(x,Q_k))^{-m}$ for all $K \in 2^{\N_{0}}$, $k \in\Z^{d}$ and $x\in\Rd$. Let $\Theta_0$ be like in \eqref{eq:thetazero}.

We have the following:
\begin{proposition}
\label{prop:BilinearSpatialLocalization}
    Let $f,g \in \mathcal{S}(\R^d)$ be two frequency-localized functions with $\hat{f}, \hat{g} \subseteq \{ \xi \in \R^d : \hat{\xi} \in \Theta_0, \, |\xi| \sim N\}$. For $q \geq 2$ it holds
    \begin{equation*}
    	\begin{split}
        &\quad \| S_N(t) f S_N(t) g \|_{L^{q/2}([0,1] \times \R^d)} \\
        &\lesssim \big( \sum_{ \substack{m, k_i \in \Z^d \\ |k_i-m| \leq 5}} \| \tilde{\theta}_{N^{1/2},m} ( S_N(t) \tilde{\theta}_{N^{1/2},k_1} f) (S_N(t) \tilde{\theta}_{N^{1/2},k_2} g) \|_{L^{q/2}([0,1] \times \R^d)}^{q/2} \big)^{2/q} \\
       &\quad + C_m N^{-m} \| f \|_2 \| g \|_2.
       \end{split}
    \end{equation*}
\end{proposition}
\begin{proof}
To lighten the notation, the $N^{1/2}$-subscript of $\hat{\theta}$ indicating its decay will be omitted.
The following estimate is immediate from the argument in Corollary \ref{cor:extrabutweak}:
\begin{equation*}
    \| S_N(t) f S_N(t) g \|_{L^{q/2}} \lesssim \big( \sum_{m \in \Z^d} \| \tilde{\theta}_m (S_N(t) f S_N(t) g) \|_{L^{q/2}}^{q/2} \big)^{2/q}.
\end{equation*}
We decompose $f = \sum_{k \in \Z^d} \tilde{\theta}_k f$, $g = \sum_{k \in \Z^d} \tilde{\theta}_k g$ (decay of order $N^{-1/2}$).
\begin{equation*}
\begin{split}
    \| \tilde{\theta}_m S_N(t) f S_N(t) g \|_{L^{q/2}} &\leq \big( \sum_{|k_i-m| \leq 5} + \sum_{\substack{|k_1-m| \leq 5, \\ |k_2-m| > 5}} + \sum_{\substack{|k_1-m|>5, \\ |k_2-m| \leq 5}} +  \sum_{|k_i-m| >5} \big) \\
    &\quad \times \| \tilde{\theta}_m (S_N(t) \tilde{\theta}_{k_1} f) (S_N(t) \tilde{\theta}_{k_2} g) \|_{L^{q/2}}.
\end{split}
\end{equation*}
We will obtain decay for the off-diagonal expressions. For $|k_1-m| > 5$, $|m-k_2| \leq 5$ we find
\begin{equation}
    \label{eq:FirstOffDiagoanl} \| \tilde{\theta}_m (S_N(t) \tilde{\theta}_{k_1} f) (S_N(t) \tilde{\theta}_{k_2} g) \|_{L^{q/2}} \leq C_l N^{-l} (1+N^{1/2} |k_1-m|)^{-l} \| \tilde{\theta}_{k_1}^{1/4} f \|_{L^q} \| \tilde{\theta}_{k_2} g \|_{L^q}.
\end{equation}
An analogous decay is obtained for $|k_1-m| \leq 5$ and $|m-k_2| > 5$. For $|m-k_i| \geq 5$, $i=1,2$ we find
\begin{equation}
    \label{eq:SecondOffDiagonal}
    \begin{split}
    &\quad \| \tilde{\theta}_m (S_N(t) \tilde{\theta}_{k_1} f) (S_N(t) \tilde{\theta}_{k_2} g) \|_{L^{q/2}} \\
    &\leq C_l N^{-l} (1+N^{1/2} |k_1-m|)^{-l} (1+N^{1/2} |k_2-m|)^{-l} \| \tilde{\theta}_{k_1}^{1/4} f \|_{L^q} \| \tilde{\theta}^{1/4}_{k_2} g \|_{L^q}.
    \end{split}
\end{equation}
We show how the off-diagonal terms can be assembled to the rapidly decaying error term. For instance, \eqref{eq:SecondOffDiagonal} yields by dominating $\tilde{\theta}_k^{1/4}$ with a smooth and decaying function $\theta_k$ off $Q_k$ and Sobolev embedding:
\begin{equation}
\label{eq:offdiagonal}
\begin{split}
    &\quad C_l N^{-l} (1+N^{1/2} |k_1-m|)^{-l} (1+N^{1/2} |k_2-m|)^{-l} \| \tilde{\theta}_{k_1}^{1/4} f \|_{L^q} \| \tilde{\theta}^{1/2}_{k_4} g \|_{L^q} \\
    &\leq C_l N^{-l/2} (1+N^{1/2} |k_1-m|)^{-l} (1+N^{1/2} |k_2-m|)^{-l} \| \langle D \rangle^{\alpha} ( \theta_{k_1} f) \|_{L^2} \|  \langle D \rangle^{\alpha} ( \theta_{k_2} g ) \|_{L^2}.
    \end{split}
\end{equation}
Finally, carrying out the $\ell^{q/2}-$sum over $m$ yields
\begin{equation*}
   \eqref{eq:offdiagonal} \leq  C_l N^{-l/2} (1+N^{1/2} |k_1-k_2|)^{-l} \| \langle D \rangle^{\alpha} ( \theta_{k_1} f) \|_{L^2} \|  \langle D \rangle^{\alpha} ( \theta_{k_2} g ) \|_{L^2}.
\end{equation*}
The claim is concluded by Young's inequality.

To show the claimed off-diagonal estimates, we record that for $|k_1-m| \geq 5$
\begin{equation}
\label{eq:RapidOffDiagonalDecay}
    \| \tilde{\theta}_m^{1/2} S_N(t) \tilde{\theta}_{k_1} f \|_{L^q} \leq C_l N^{-l} (1+ N^{1/2} |k_1-m|)^{-l} \| \tilde{\theta}_{k_1}^{1/4} f \|_{L^q}.
\end{equation}
This is a consequence of Proposition \ref{prop:PhaseSpaceLocalization}.
\end{proof}

\begin{theorem}\label{thm:bilinearL2}
Let $q:= \frac{2(d+3)}{d+1}$ and $\veps>0$. Then there exists a $C\geq0$ such that the following holds. Let $N^{-1/2}\leq \nu\leq 1$ and let $\Theta_{1},\Theta_{2}\subseteq \Theta_{0}$ be spherical caps of angular width $c_{2}\nu$, distance at least $c_{1}\nu$ apart. Let $f_{1},f_{2}\in\Sw(\Rd)$ be such that  
\begin{equation}\label{eq:bilinearL2assume}
\supp(\wh{f_{j}})\subseteq \{ \xi \in \R^d \mid c_{3}^{-1}\leq |\xi|\leq c_{3}, \; \hat{\xi}\in\Theta_{j}\}   
\end{equation}
for $j\in\{1,2\}$. Then 
\begin{equation}
\label{eq:BilinearEstimate}
\| \tilde{S}_N  f_1 \tilde{S}_N f_2 \|_{L^{q/2}([0,N] \times \R^d)} \leq CN^{\varepsilon} \nu^{-\frac{4}{d+3}} \prod_{j=1}^2 \| f_j \|_{L^{2}(\Rd)}.
\end{equation}
\end{theorem}
\begin{proof}
Throughout, fix $\veps>0$ small and set $\delta:=\veps^{2}$. 

We first consider the case of unit transversality, $\nu\eqsim 1$. We then split the general case $N^{-\frac{1}{2}}\leq \nu\leq 1$ into two intervals: $N^{-\frac{1}{2}}\leq \nu\leq N^{-\frac{1}{2}+\delta}$ and $N^{-\frac{1}{2}+\delta}<\nu\leq 1$. The required estimates in the first interval are a direct consequence of Proposition \ref{prop:bilinearinfty}, whereas the second interval is dealt with by rescaling to $\nu\eqsim 1$.

By Proposition \ref{prop:BilinearSpatialLocalization} it suffices to prove the estimate
\begin{equation*}
    \| \tilde{S}_N (\tilde{\theta}_i f_1) \tilde{S}_N (\tilde{\theta}_j f_2) \|_{L^{q/2}([0,N] \times B_d(0,CN)))} \leq C_\varepsilon N^{\varepsilon} \nu^{-\frac{4}{d+3}} \prod_{j=1}^2 \| f_j \|_{L^2},
\end{equation*}
after translation, for $|i|+|j| \leq 10$. We abuse notation and let $i=j=0$ in the following.

\subsubsection*{The case where $\nu \eqsim 1$} 
We can employ the isotropic wave packet decompositions provided by Theorem \ref{thm:WavePacketDecomposition} for the individual terms to obtain wave packets concentrated in space on scale $N^{\frac{1}{2}}$ and in frequency on scale $N^{-\frac{1}{2}}$ along the bicharacteristics of $a_N$.

Next, we apply the wave packet decomposition provided by Theorem \ref{thm:WavePacketDecomposition}
\begin{equation*}
\tilde{S}_N(t) \big( \tilde{\theta}_0(x) \psi^{\nu}_{\omega_{i}}(D) f_i \big) = \sum_{T \in \T} a_T \phi_T + g
\end{equation*}
with $\T \subseteq ( B_d(0,CN) \times \{ \xi \in \R^d: (2c_{3})^{-1}\leq |\xi|\leq 2c_{3}, \; \hat{\xi}\in 2\Theta_{j} \} ) \cap \Lambda_N^*$ and for $p \geq 2$ we let $\| g \|_{L^p} \lesssim_m N^{-m} \| f \|_{L^2}$.

Also, we remark that in \cite{SchippaTataru2025} the space-time domain was $B_{d+1}(0,N)$. We can readily translate physical space coordinates to the origin, which does not change the symbol bounds, and mildly enlarge the set to $[0,N] \times B_d(0,N)$ once we ensure that the required bounds on the Hamiltonian hold on this slightly larger set.

We explain the statement and the proof of the key Kakeya estimate \eqref{eq:KeyKakeyaBilinear}. In the following $\delta=\varepsilon^2$ will be a constant chosen much smaller than $\varepsilon$. The idea is that by the kernel bounds for the wave packets and their number limited by $N^{10d}$, allowing for a margin of $N^{\delta}$ in phase space, will give inessential contribution of the wave packets away from the margin.

The focusing relation $T \sim Q$ between curved tubes and $N^{1-\delta}$-cubes was introduced in \cite{SchippaTataru2025} (going back to \cite{Wolff2001,Tao2003}). We define the energy difference function:
\begin{equation*}
F^z_{\xi_1,\xi_2'}(\eta) = a_N(z,\xi_1) + a_N(z,\eta + \xi'_2 - \xi_1) - a_N(z,\eta) - a_N(z,\xi_2'),
\end{equation*}
where $z=(x,t)$, $\hat{\xi}_1,\hat{\eta} \in \Theta_1$, $ \hat{\xi_2'} \in \Theta_2$, $\angle(\Theta_1,\Theta_2) \sim 1$ such that $|x_{T_1}(t) - x| \lesssim N^{\frac{1}{2}+\delta}$.

The energy shell is defined as
\begin{equation*}
\mathcal{E}^q_{\xi_1,\xi_2'} = \{ \eta_{T_1}(t) \mid T_1 \in \mathbb{T}_1, |F^{z_q}_{\xi_1,\xi_2'}(\eta_{T_1}(t_q))| \lesssim N^{-\frac{1}{2}+\delta} \},
\end{equation*}
which supports the frequencies almost annihilating the energy difference function.

Dyadic pigeonholings of the cubes and tubes are defined like in \cite{SchippaTataru2025}.
For $q \in \mathfrak{q}[\mu_1,\mu_2]$ let $\mathbb{T}_1^{\not\sim S}(q,\lambda,\mu_1,\mu_2)$ be the tubes $T_1 \in \mathbb{T}_1[\lambda,\mu_1,\mu_2]$ associated to $q$, i.e., with $T_1 \not\sim S$, $N^{\delta} T_1 \cap q \neq \emptyset$. $\mathbb{T}_1^{\not\sim S}(\mathcal{E}^{z_q}_{\xi_1,\xi_2'},\lambda,\mu_1,\mu_2)$ is defined like in \cite{SchippaTataru2025}

Note that the spatial parameter in \cite{SchippaTataru2025} was denoted by $R$. Here we use $N$ to relate it directly to the input frequencies.
We need to prove for $q \in \mathfrak{q}[\mu_1,\mu_2]$, $q \in S$:
\begin{equation}
\label{eq:KeyKakeyaBilinear}
| \mathbb{T}_1^{\not\sim S}(\mathcal{E}^{z_q}_{\xi_1,\xi_2'},\lambda,\mu_1,\mu_2)| \lesssim N^{C \delta} \frac{\# \T_2}{\lambda \mu_2},
\end{equation}
which combinatorial estimate is the key point in the proof of bilinear estimate \cite[Eq.~(3.22)]{SchippaTataru2025}. This is reduced like in \cite{SchippaTataru2025} to the estimate
\begin{equation*}
|\mathcal{Q}_q| \lesssim N^{C \delta},
\end{equation*}
where
\begin{equation*}
\begin{split}
\mathcal{Q}_q &= \{ (q',T_1) \in \mathfrak{q} \times \mathbb{T}_1(\mathcal{E}^q_{\xi_1,\xi_2'}), \text{dist}(q,q') \geq N^{1-\delta}, \\
&\quad N^{\delta} T_2 \cap q' \neq \emptyset, \, N^{\delta} T_1 \cap q \neq \emptyset, \, N^{\delta} T_1 \cap q' \neq \emptyset \}.
\end{split}
\end{equation*}
The proof consists of two steps:
\begin{enumerate}
\item Arguing that for fixed $q'$, which is at large distance $\geq N^{1-\delta}$, $T_1 \in \mathbb{T}_1(\mathcal{E}^q_{\xi_1,\xi_2'})$ intersecting $q$ and $q'$ is essentially uniquely determined.
\item Arguing that tubes emanating from an energy shell at $q$ intersect far away tubes $T_2$ transversely.
\end{enumerate}

By Lemma \ref{lem:Comparison} the wave packet decompositions are indistinguishable, for which reason both claims extend to the asymptotically homogeneous case. 

\subsubsection*{The case where $N^{-\frac{1}{2}+\delta}<\nu\leq 1$}

This was in the $1$-homogeneous case, which was handled in \cite{SchippaTataru2025}, reduced to unit transversality by an anisotropic rescaling. In the following we suppose that the Fourier supports of $f_i$ are constrained to the same sector of aperture $3 \nu$ centered at $e_d$, which is no loss of generality by rotation. We carry out the almost orthogonality on the level of $S_N$ to use notation from Appendix \ref{subsection:KernelBoundsPhaseSpace}. Let
\begin{equation*}
    f_i = \sum_{k \in \Z^d} \tilde{\theta}^{\nu}_{k,e_d} f_i.
\end{equation*}
We argue by the kernel bounds that
\begin{equation*}
\begin{split}
    \| S_N f_1 S_N f_2 \|_{L^{q/2}([0,1] \times B_d(0,C))} &\leq \sum_{|k_1-k_2| \leq 5} \| S_N f_{1 k_1} S_N f_{2 k_2} \|_{L^{q/2}([0,1]\times \R^d)} \\
    &\quad + C_m N^{-m} \| f_1 \|_2 \| f_2 \|_2.
\end{split}
\end{equation*}
Let $y_k = \nu(k_1 e_1 + \ldots + k_{d-1} e_{d-1}) + k_d \nu^2 e_d$ and
\begin{equation*}
    \hat{\mathcal{Y}}_{k,e_d} = \{ (y,\hat{\xi}) : d_S(y,\hat{\xi},y_k,e_d) \leq 3 \nu \}.
\end{equation*}
Let $(y_k^t,\omega^t) = \hat{\chi}_t (y_k,e_d)$. The kernel bounds imply the following pointwise estimate:
\begin{equation*}
    |S_N (\tilde{\theta}^{\nu}_{k,e_d} f_i)(x)| \leq C N^d (1+ N \text{dist}_S^2(x,\omega^t,\hat{\chi}_t(\hat{\mathcal{Y}}_{k,e_d})))^{-m} \| ( \tilde{\theta}_{k,e_d}^{\nu})^{1/2} f_i \|_2.
\end{equation*}

By the bi-Lipschitz property of the flow, the estimate for the off-diagonal term is immediate:
\begin{equation*}
   \sum_{|k_1-k_2| \geq 5} \| S_N f_{1 k_1} S_N f_{2 k_2} \|_{L^{q/2}([0,1]\times B_d(0,C))} \leq C_m N^{-m} \| f_{1} \|_2 \| f_2 \|_2.
\end{equation*}

We carry out the centering in phase space for the diagonal terms $|k_1-k_2| \leq 5$:
\begin{equation*}
    \| \tilde{S}_N f_{1 k_1} \tilde{S}_N f_{2 k_2} \|_{L^{q/2}([0,N]\times B_d(0,CN))}.
\end{equation*}
We abuse notation and denote the rescalings of $f_i$ again by $f_i$. In \cite[Subsection~5.3.2]{SchippaTataru2025} it is detailed how in the $1$-homogeneous case, after changing to angular frequencies, $\omega = \frac{\xi'}{\xi_d}$ the flow is centered around the bicharacteristic ray $\{ (\bar{x}(t),\xi_d \bar{\omega}(t),\xi_d) : \xi_d \in (1/2,3/2) \}$ by the phase space transformation $x \to x+ \bar{x}(t)$ and $\xi' \to \xi' + \bar{\omega}(t) \xi_d$. $(\bar{x}(t),\bar{\omega}(t))$ denote the central bicharacteristic for the $\nu$-sector we are considering. The linear transformation $\mathcal{T}$ of the symbol is detailed in \cite[Subsection~5.3.2]{SchippaTataru2025} and applying $\mathcal{T}$ yields 
\begin{equation*}
\mathcal{T} a_N = \mathcal{T} a_N^h + \mathcal{T} r_N.
\end{equation*}
Then the anisotropic rescaling 
\begin{equation*}
\mathcal{S}: x' \to \nu^{-1} x',\, x_d \to x_d, \, \xi' \to \nu \xi', \, t \to \nu^{-2} t
\end{equation*}
can be carried out which yields functions $\tilde{f}_1$, $\tilde{f}_2$ of unit aperture separation governed by the symbol
\begin{equation*}
\nu^{-2} \mathcal{S} \mathcal{T} a_N = \nu^{-2} \mathcal{S} \mathcal{T} a_N^h + \nu^{-2} \mathcal{S} \mathcal{T} r_N.
\end{equation*}
For $\nu^{-2} \mathcal{S} \mathcal{T} a_N^h$ was argued in \cite{SchippaTataru2025} that wave packet decompositions on the scale $R = \nu^{2} N$ are at disposal by virtue of
\begin{equation*}
|\partial_x^{\alpha} \partial_{\xi}^{\beta} (\mathcal{S} \mathcal{T} a_N^h) | \lesssim_{\alpha,\beta} (\nu^{2} N)^{-|\alpha|+  \frac{(|\alpha|-2)_+}{2}}.
\end{equation*}
The implicit constants only depend on the implicit constants in the symbol bounds for $a_N$.

However, the time regularity is reduced, which is salvaged by a local change of coordinates. This is detailed in \cite[Section~5.3.2]{SchippaTataru2025}. Here we show that the wave packet decompositions are still indistinguishable for which we turn to bounds for $\mathcal{S} \mathcal{T} r_N$. By Lemma \ref{lem:DifferenceEstimates} we find
\begin{equation*}
|\partial_{x}^{\alpha} \partial_{\xi}^{\beta} (\nu^{-2} \mathcal{S} \mathcal{T} r_N) | \lesssim_{\alpha,\beta} (\nu^2 N)^{-1 - \frac{|\alpha|}{2}}.
\end{equation*}
The above derivative bounds follow directly from the estimates for $r_N$.

Hence, by the argument of Lemma \ref{lem:DifferenceEstimates} applied on the scale $\nu^2 N$ the wave packet decompositions of $\nu^{-2} \mathcal{S} \mathcal{T} a_N$ and $\nu^{-2} \mathcal{S} \mathcal{T} a_N^h$ are again indistinguishable. Modifying the case of unit transversality, recalled above, like in \cite{SchippaTataru2025} for reduced time regularity completes the proof of this case.

\subsubsection*{The case where $N^{-\frac{1}{2}}\leq \nu\leq N^{-\frac{1}{2}+\delta}$.}

Using a partition of unity as in e.g.~\cite[p.~11]{Rozendaal22b}, 
we can decompose each $f_{j}$ into a sum $f_{j}=\sum_{\w} f_{j\w}$ of approximately $(\nu N^{1/2})^{d-1}\leq N^{\veps}$ terms $f_{j\w}$, each of which satisfies the support condition in \eqref{eq:bilinearL2assume} with $\nu=N^{-1/2}$,  as well as the estimate $\|f_{j\w}\|_{L^{2}(\Rd)}\lesssim \|f_{j}\|_{L^{2}(\Rd)}$. Writing $(f_{j\w})_{1/N}(y):=f_{j\w}(Ny)$ for $y\in\Rd$, we can thus simply combine \eqref{eq:rescalingkey} and Proposition \ref{prop:bilinearinfty}:
\begin{align*}
&\| \tilde{S}_N  f_1 \tilde{S}_N f_2 \|_{L^{q/2}([0,N] \times \R^d)}\leq \prod_{j=1}^{2}\| \tilde{S}_N  f_j\|_{L^{q}([0,N]\times\Rd)}\\
&\leq \prod_{j=1}^{2}\sum_{\w}\| \tilde{S}_N  f_{j\w}\|_{L^{q}([0,N]\times\Rd)}=N^{\frac{2(d+1)}{q}}\prod_{j=1}^{2}\sum_{\w}\|S_N (f_{j\w})_{1/N}\|_{L^{q}([0,1]\times\Rd)}\\
&\lesssim N^{\frac{2(d+1)}{q}+(d+1)(\frac{1}{2}-\frac{1}{q})}\prod_{j=1}^{2}\sum_{\w}\| (f_{j\w})_{1/N}\|_{L^{2}(\Rd)}=N^{\frac{d+1}{q}-\frac{d-1}{2}+\veps}\prod_{j=1}^{2}\| f_{j}\|_{L^{2}(\Rd)}.
\end{align*}
Since $\frac{d+1}{q}-\frac{d-1}{2}=\frac{2}{d+3}$, this suffices.

\end{proof}

\begin{corollary}\label{cor:bilinearL2main}
Let $\frac{2(d+3)}{d+1}\leq q\leq \infty$ and $\delta,\veps\geq 0$. 
Suppose that one of the following conditions holds:
\begin{enumerate}
    \item\label{it:bilinearL2main1} $q=\frac{2(d+3)}{d+1}$ and $\veps>0$; 
    \item\label{it:bilinearL2main2} $q>\frac{2(d+3)}{d+1}$ and $\delta>0$;
\end{enumerate}
Then there exists a $C\geq0$ such that the following holds. Let $N^{-1/2}\leq \nu\leq 1$ and let $\Theta_{1},\Theta_{2}\subseteq \Theta_{0}$ be spherical caps of angular width $c_{2}\nu$, distance at least $c_{1}\nu$ apart. Let $f_{1},f_{2}\in\Sw(\Rd)$ be such that  
\begin{equation*}
\supp(\widehat{f_{j}})\subseteq \{ \xi \in \R^d \mid N/c_{3}\leq |\xi|\leq c_{3}N, \hat{\xi}\in\Theta_{j}\}    
\end{equation*}
for $j\in\{1,2\}$. 
Then 
\[
\|(S_{N}f_{1})(S_{N}f_{2})\|_{L^{q/2}([0,1] \times \R^d)}\leq C N^{d - \frac{2(d+1)}{q}+\veps} \nu^{d-1 - \frac{2(d+1)}{q}-\delta} \prod_{j=1}^{2}\| f_{j}\|_{L^{2}(\Rd)}. 
\]
\end{corollary}
\begin{proof}
For $q=\frac{2(d+3)}{d+1}$, the statement follows from Theorem \ref{thm:bilinearL2} after rescaling, cf.~\eqref{eq:rescalingkey}. It is shown in Appendix \ref{sec:noeps} that one can remove the $\veps$ loss away from the endpoint exponent. That is, \eqref{it:bilinearL2main2} follows from Proposition \ref{prop:SharpBilinearEstimate} after rescaling. 
\end{proof}

We can now prove the main result of this subsection.

\begin{theorem}\label{thm:BilinearEstimates}
Let $2\leq p\leq \infty$, $\frac{d+3}{d+1}p\leq q\leq \infty$ and $\delta,\veps\geq0$. 
Suppose that one of the following conditions holds:
\begin{enumerate}
\item\label{it:bilinearmain1}  $q=\frac{(d+3)}{d+1}p$ and $\veps>0$;
\item\label{it:bilinearmain2} $q>\frac{(d+3)}{d+1}p$ and $\delta>0$;
\item\label{it:bilinearmain3} $q=\infty$.
\end{enumerate}
Then there exists a $C\geq0$ such that the following holds. Let $N^{-1/2}\leq \nu\leq 1$ and let $\Theta_{1},\Theta_{2}\subseteq \Theta_{0}$ be spherical caps of angular width $c_{2}\nu$, distance at least $c_{1}\nu$ apart if $\nu\geq 2N^{-1/2}$. 
Let
$f_{1},f_{2}\in\Sw(\Rd)$ be such that 
\begin{equation}\label{eq:bilinearassume}
\supp(\wh{f_{j}})\subseteq \{ \xi \in \R^d \mid N/c_{3}\leq |\xi|\leq c_{3}N, \hat{\xi}\in\Theta_{j}\}    
\end{equation}
for $j\in\{1,2\}$. Then 
\begin{equation}\label{eq:bilinearprove}
\begin{split}
&\|(S_{N}f_{1})(S_{N}f_{2})\|_{L^{q/2}([0,1] \times \R^d)}\\
&\leq C N^{d-1-\frac{2(d+1)}{q}+\frac{2}{p}-2s(p)+\veps} \nu^{2(d-1-\frac{d+1}{q}-\frac{d-1}{p})-\delta} \prod_{j=1}^2 \| f_{j} \|_{\mathcal{H}^{p}_{FIO}(\Rd)}.
\end{split}
\end{equation}
\end{theorem}
\begin{proof}
Let $\chi_{1},\chi_{2}\in C^{\infty}_{c}(\Rd)$ equal $1$ on the region on the right-hand side of \eqref{eq:bilinearassume}, with support in a slightly bigger region, and such that $\|\F^{-1}(\chi_{j})\|_{L^{1}(\Rd)}$ is bounded uniformly in $N$. From \eqref{eq:bilinearprove} 
one then obtains a norm bound for the bilinear operator $(f_{1},f_{2})\mapsto (S_{N}\chi_{1}(D)f_{1})(S_{N}\chi_{2}(D)f_{2})$. 
By \cite[Theorem 4.4.1]{Bergh-Lofstrom76}, such a norm bound can be interpolated. Hence we only need to consider the case where $p=2$ and $q<\infty$, and the case where $q=\infty$. 

For $p=2$ and $q<\infty$, the statement is contained in Corollary \ref{cor:bilinearL2main} for  $\nu\geq 2N^{-1/2}$. 
For $N^{-1/2}\leq \nu<N^{-1/2}$, Cauchy--Schwarz and Proposition \ref{prop:bilinearinfty} yield
\begin{align*}
\|(S_{N}f_{1})(S_{N}f_{2})\|_{L^{q/2}([0,1] \times \R^d)}&\leq \|S_{N}f_{1}\|_{L^{q}([0,1] \times \R^d)}\|S_{N}f_{2}\|_{L^{q}([0,1] \times \R^d)}\\
&\lesssim N^{(d+1)(\frac{1}{2}-\frac{1}{q})}\| f_{1} \|_{L^{2}(\Rd)}\| f_{2} \|_{L^{2}(\Rd)}.
\end{align*}
This suffices, because $(d+1) \big( \frac{1}{2} - \frac{1}{q} \big) = d-1-\frac{2(d+1)}{q}+1-(d-1-\frac{d+1}{q}-\frac{d-1}{2})$.

Next, consider the case where $q=\infty$. 
Here we do not need the assumption of angular separation, and one may let $\delta=\veps=0$. 
As in the proof of Theorem \ref{thm:bilinearL2}, 
we can decompose each $f_{j}$ into a sum $f_{j}=\sum_{\w} f_{j\w}$ of approximately $(\nu N^{1/2})^{d-1}$ terms $f_{j\w}$, each of which satisfies the support condition in \eqref{eq:bilinearinftyassume}, with finitely overlapping Fourier supports. 
Now the required statement is a consequence of Proposition \ref{prop:bilinearinfty}, H\"{o}lder's inequality and \cite[Proposition 7.1]{HaPoRoYu26}, for each $j\in\{1,2\}$ separately:
\begin{align*}
&
\|(S_{N}f_{j}) \|_{L^\infty([0,1]\times\Rd)}
\leq 
\sum_{\w}\|S_{N}f_{j\w}\|_{L^\infty([0,1]\times\Rd)}
\lesssim N^{\frac{d+1}{2p}-s(p)}
\sum_{\w}\|f_{j\w}\|_{\HT^{p}_{FIO}(\Rd)}\\
&\lesssim N^{\frac{d+1}{2p}-s(p)}(\nu N^{1/2})^{\frac{d-1}{p'}}
\Big(\sum_{\w}\|f_{j\w}\|_{\HT^{p}_{FIO}(\Rd)}^{p}\Big)^{1/p}\\
&\lesssim N^{\frac{d-1}{2}+\frac{1}{p}-s(p)}\nu^{d-1-\frac{d-1}{p}}\|f_{j}\|_{\Hp}.\qedhere
\end{align*}. 
\end{proof}

\subsection{Angular decomposition and almost orthogonality}\label{subsec:orthogonal}

In this subsection, we first consider an angular Whitney decomposition in frequency. We then prove an almost orthogonality statement which links this decomposition to the bilinear estimates from the previous subsection. 

Fix 
$N\in 2^{\N}$ large, $c_{0}>0$ small, $0<c_{1}\leq 1$ and $c_{2}>c_{1}$. 
Let $\Theta_{0}$ be as in \eqref{eq:thetazero}. 
Using an angular Whitney decomposition as in \cite{TaoVargas2000}, 
write 
\begin{equation}\label{eq:Whitney}
\Theta_{0}\times\Theta_{0}\subseteq \bigcup_{N^{-1/2}\leq \nu\leq c_{1}}\bigcup_{\w_{1}\sim_{\nu} \w_{2}}\Theta^{\nu}_{\w_{1}}\times \Theta^{\nu}_{\w_{2}}.
\end{equation}
Here the first union is over all $\nu\in 2^{-\N}$ such that $N^{-1/2}\leq \nu\leq c_{1}$. 
Moreover, for each $N^{-1/2}\leq \nu\leq c_{1}$, the $\Theta_{\w_{1}}^{\nu}$ and $\Theta_{\w_{2}}^{\nu}$ are spherical caps of angular width $c_{2}\nu$. 
For $\nu\geq 2N^{-1/2}$, the notation $\w_{1}\sim_{\nu}\w_{2}$ means that the distance between $\Theta_{\w_{1}}^{\nu}$ and $\Theta_{\w_{2}}^{\nu}$ is bounded from below by $c_{1}\nu$ and from above by $c_{2}\nu$, 
whereas for $N^{-1/2}\leq \nu<2N^{-1/2}$ 
it merely means that the distance between $\Theta_{k}^{\nu}$ and $\Theta_{k'}^{\nu}$ is bounded from above by $c_{2}\nu$. 
A priori, one may assume that the $\Theta_{\w_{1}}^{\nu}\times \Theta_{\w_{2}}^{\nu}$ have non-overlapping interiors.

Then, by slightly enlarging the caps and by modifying $c_{1}$ and $c_{2}$, we may suppose that the $\Theta_{\w_{1}}^{\nu}\times \Theta_{\w_{2}}^{\nu}$ have only fixed finite overlap, and we may choose a smooth partition of unity $\{\psi^{\nu}_{\w_{1}}\otimes \psi^{\nu}_{\w_{2}}\mid N^{-1/2}\leq \nu\leq c_{1}, \, \w_{1}\sim_{\nu}\w_{2}\}$ subordinate to the cover \eqref{eq:Whitney}. That is, each $\psi^{\nu}_{\w_{j}}\in C^{\infty}(\Rn\setminus \{0\}$ is positively homogeneous of degree $0$ and satisfies $0\leq \psi^{\nu}_{\w_{j}}\leq 1$ and $\supp(\psi^{\nu}_{\w_{j}})\subseteq \{\xi\in\Rd\setminus\{0\}\mid \hat{\xi}\in \Theta_{\w_{j}}^{\nu}\}$, and $\ind_{\Theta_{0}}=\ind_{\Theta_{0}}\sum_{N^{-1/2}\leq \nu\leq c_{1}}\sum_{\w_{1}\sim_{\nu} \w_{2}}\psi^{\nu}_{\w_{1}}\otimes \psi^{\nu}_{\w_{2}}$. Moreover, the $\F^{-1}(\sum_{M=N/8}^{8N}\psi_{M}\psi_{\w_{j}}^{\nu})$ are uniformly bounded in $L^{1}(\Rd)$.  We remark that we overload the notation as the $\psi_{\omega}^{\nu}$ introduced in \eqref{eq:AngularLocalization} are defined slightly differently, but they have the same relevant properties. Write 
\[
f^{\nu}_{\w_{j}}:=\sum_{M=N/8}^{8N}\psi_{M}(D)\psi^{\nu}_{\w_{j}}(D)f
\]
for $f\in\Sw(\Rd)$.

Next, for $N^{-1/2}\leq \nu\leq c_{1}$, $f,g\in\Sw(\Rd)$ and $(t,x)\in\R^{1+d}$, set
\begin{equation}\label{eq:Bnu}
B_{\nu}(f,g)(t,x):=\sum_{\w_{1}\sim_{\nu}\w_{2}}S_{N}(t)f^{\nu}_{\w_{1}}(x)\cdot S_{N}(t)g^{\nu}_{\w_{2}}(x).
\end{equation}
This decomposition is motivated by the fact that 
\begin{equation}\label{eq:bilinear}
S_{N}f\cdot S_{N}g=\sum_{N^{-1/2}\leq \nu\leq c_{1}}B_{\nu}(f,g)
\end{equation}
if 
$\supp(\wh{f}\,),\supp(\wh{g})\subseteq \{\xi\in\Rn\mid N/8\leq |\xi|\leq 8N, |\hat{\xi}-e_{d}|\leq c_{0}\}$.

To estimate the right-hand side of \eqref{eq:bilinear}, we will rely on the bilinear estimates from the previous subsection and on the following almost orthogonality statement.

\begin{proposition}
\label{prop:AlmostOrthogonalityPhaseSpace}
Let $1\leq q\leq p \leq\infty$, $\delta>0$ and $m\geq0$, 
and set $q^{*}:=\min(q,q')$. Then there exists a $C\geq0$ such that
\begin{align*}
&\|B_{\nu}(f,g)\|_{L^{q}([0,1]\times \Rd)}\\
&\leq C\nu^{-\delta}\Big(\sum_{\w_{1}\sim_{\nu}\w_{2}}\|(S_{N} f^{\nu}_{\w_{1}})(S_{N} g^{\nu}_{\w_{2}})\|_{L^{q}([0,1]\times\Rd)}^{q^{*}}\Big)^{1/q^{*}}
+CN^{-m}\|f\|_{\mathcal{H}^{p}_{FIO}} \|g\|_{{\mathcal{H}^p_{FIO}}}
\end{align*}
for all $N^{-1/2}\leq \nu\leq c_{1}$ and $f,g\in\Sw(\Rd)$. 
\end{proposition}
\begin{proof}
For constant coefficients, this follows for $q=2$ from the Fourier support. In the variable coefficient case, an additional spatial localization comes to rescue. To make this effective, we rely on the kernel bounds from Appendix \ref{sec:kernel} instead of integration by parts in oscillatory integrals as in \cite{Lee2006}. A modification leads to a $\nu^{-\delta}$ loss instead of an $N^{\veps}$ loss. 

\subsubsection*{Preliminary work}
Fix $0<\delta<1$. It suffices to prove the statement for fixed $t\in[0,1]$. That is, we may show that
\begin{align*}
&\Big\|\sum_{\w_{1}\sim_{\nu}\w_{2}}(S_{N}(t)f^{\nu}_{\w_{1}})( S_{N}(t)g^{\nu}_{\w_{2}})\Big\|_{L^{q}(\Rd)}\\
&\lesssim \nu^{-\delta}\Big(\sum_{\w_{1}\sim_{\nu}\w_{2}}\|(S_{N}(t) f^{\nu}_{\w_{1}})(S_{N}(t)g^{\nu}_{\w_{2}})\|_{L^{q}(\Rd)}^{q^{*}}\Big)^{1/q^{*}}
+N^{-m}\|f\|_{\mathcal{H}^{p}_{FIO}} \|g\|_{{\mathcal{H}^p_{FIO}}},
\end{align*}
with an implicit constant independent of $t$.

As in Appendix \ref{subsection:KernelBoundsPhaseSpace}, let $\theta \in \mathcal{S}(\Rd)$ be such that $\widehat{\theta}\in C^{\infty}_{c}(\Rd)$ and 
$\sum_{\la \in \Z^{d}} \theta(x-\la)^2 = 1$ for all $x \in \Rd$. Set 
$\theta_{\ka}(x):=\theta(\nu^{-1}(x-\ka))$ for $\ka\in \nu^{1+\delta} \Z^{d}$ and $x\in\Rd$. Moreover, $\text{supp}(\hat{\theta}_k) \subseteq [-C N^{1/2},CN^{1/2}]^d$. We have the following analog of \eqref{eq:localization}:
\begin{equation}\label{eq:localizationtwo}
\|g\|_{L^{q}(\Rd)}\eqsim \Big(\sum_{\kappa\in \nu\Z^{d}}\|\theta^2_{\kappa}g\|_{L^{q}(\Rd)}^{q}\Big)^{1/q}.
\end{equation}
The implicit constant is independent of $N$, $\nu$ and $g\in L^{q}(\Rd)$. 

Throughout, for given $\ka\in\nu\Z^{d}$ and $\w\in S^{d-1}$, let $y_{\w\ka}\in\Rd$ and $\w_{\ka}\in S^{d-1}$ be such that $\hat{\chi}_{t}(y_{\w\ka},\w)=(\ka,\w_{\ka})$, where $\hat{\chi}_{t}$ is the projection of the bicharacteristic flow map of $a_{h}$ at time $t$ (see Appendix \ref{sec:kernel}).

Let $\psi\in C^{\infty}_{c}(\Rd)$ be like in Section \ref{subsec:spaces}, and fix $c>0$ small. For $\w_{0}\in S^{d-1}$ and $\xi\in\Rd\setminus\{0\}$, set $\tilde{\psi}_{\w_{0}}(\xi):=\psi\big(c\nu^{\delta-1}(\hat{\xi}-\w_{0})\big)$. Then $|\partial_{\xi}^{\alpha}\tilde{\psi}_{\w_{0}}(\xi)|\lesssim_{\alpha} \nu^{(\delta-1)|\alpha|}|\xi|^{-|\alpha|}$ 
for all $\alpha\in\N_{0}^{d}$ and $\xi\in\Rd\setminus\{0\}$. Integration by parts with $|\alpha|=d+1$ shows that
\begin{equation}\label{eq:L1norms}
\|\F^{-1}(\tilde{\psi}_{\w_{0}}\psi_{M})\|_{L^{1}(\Rd)}
\lesssim \nu^{-2(1-\delta)}M^{-1} 
\end{equation}
for all $M\in 2^{\N}$, where $(\psi_{M})_{M\in 2^{\N}}$ is as in \eqref{eq:LittlePaley}. Let $M_{0}\in 2^{\N}$ be such that $M_{0} \sim \max(\nu^{-2(1-\delta)},N^{1/2+\delta})$, and for $\xi\in\Rd$ set
\[
\Psi_{\w_{0}}(\xi):=\sum_{M\geq M_{0}}\tilde{\psi}_{\w_{0}}(\xi)\psi_{M}(\xi),
\]
where the sum is taken over $M\in 2^{\N_{0}}$. 
It then follows from \eqref{eq:L1norms} that 
\begin{equation}\label{eq:L1norm}
\|\F^{-1}(\Psi_{\w_{0}})\|_{L^{1}(\Rd)}\lesssim 1.
\end{equation}
Let $\psi^c(M_0^{-1}D) = 1 - \psi(M_0^{-1} D)$.

\subsubsection*{The decomposition}

We give an overview of the decomposition into main and error term. First, we localize the output frequencies into low and high part:
\begin{equation*}
    \sum_{\w_{1}\sim_{\nu}\w_{2}} \big(S_{N}(t)f^{\nu}_{\w_{1}}S_{N}(t)g^{\nu}_{\w_{2}}\big) = (\psi(M_0^{-1} D) + \psi^c(M_0^{-1} D)) \sum_{\w_{1}\sim_{\nu}\w_{2}} \big(S_{N}(t)f^{\nu}_{\w_{1}}S_{N}(t)g^{\nu}_{\w_{2}}\big).
\end{equation*}
The low part can be estimated by kernel estimates as the frequencies essentially retain their dyadic size.

The high part is further decomposed spatially to $\nu$-cubes:
\begin{equation*}
\begin{split}
    &\quad \psi^c(M_0^{-1} D) \sum_{\w_{1}\sim_{\nu}\w_{2}} \big(S_{N}(t)f^{\nu}_{\w_{1}}S_{N}(t)g^{\nu}_{\w_{2}}\big) \\ &= \sum_{\kappa \in \nu \Z^d} \theta_{\kappa}^2 \psi^c(M_0^{-1} D) \sum_{\w_{1}\sim_{\nu}\w_{2}} \big(S_{N}(t)f^{\nu}_{\w_{1}}S_{N}(t)g^{\nu}_{\w_{2}}\big).
\end{split}
\end{equation*}
Now we can use an almost orthogonality similar to the constant-coefficient case, observing that the spatial output localization and the frequency input localization yields a frequency output localization by the bi-Lipschitz property of the Hamiltonian flow, to be detailed below:
\begin{equation*}
    \theta_{\kappa}^2 \psi^c(M_0^{-1} D) \sum_{\w_{1}\sim_{\nu}\w_{2}} \big( \ldots \big) = (\Psi_{\omega_{1 \kappa}}(D) + \Psi_{\omega_{1 \kappa}}^c(D)) \theta_{\kappa}^2 \psi^c(M_0^{-1} D) \sum_{\w_{1}\sim_{\nu}\w_{2}} \big( \ldots \big).
\end{equation*}
Here the angular localization via $\Psi_{\omega_{1 \kappa}}$ is to an aperture of size $\nu^{1-\delta}$: this gives rise to the main term.
The $\delta$ will give us the necessary leeway to treat the complementary angular localization as a smoothing error.

\subsubsection*{The main term}

We apply \eqref{eq:localizationtwo} to write
\begin{equation}\label{eq:decomposekappa}
\begin{aligned}
&\Big\|\sum_{\w_{1}\sim_{\nu}\w_{2}} \psi^c(M_{0}^{-1}D) \big(S_{N}(t)f^{\nu}_{\w_{1}}S_{N}(t)g^{\nu}_{\w_{2}}\big)\Big\|_{L^{q}(\Rd)}\\
&\eqsim \Big(\sum_{\kappa\in\nu\Z^{d}}\Big\|\sum_{\w_{1}\sim_{\nu}\w_{2}}\theta^2_{\ka} \psi^c(M_{0}^{-1}D) \big(S_{N}(t)f^{\nu}_{\w_{1}}S_{N}(t)g^{\nu}_{\w_{2}}\big)\Big\|_{L^{q}(\Rd)}^{q}\Big)^{1/q}.
\end{aligned}
\end{equation}
We bound the terms on the second line of \eqref{eq:decomposekappa}, for $\ka\in\nu\Z^{d}$ momentarily fixed. 

To this end, note that for a given pair $\w_{1}\sim_{\nu}\w_{2}$, there are no more than a fixed finite number of pairs $\w_{1}'\sim_{\nu}\w_{2}'$ such that $|\w_{1\ka}-\w_{1\ka}'|\leq \nu$. Indeed, for 
a given $c>0$, there are at most a fixed finite number of pairs $\w_{1}'\sim_{\nu}\w_{2}'$ such that $|\w_{1}-\w_{1}'|\leq c\nu$. On the other hand, for $c$ large enough, if $|\w_{1}-\w_{1}'|>c\nu$ then 
\[
|\w_{1\ka}-\w_{1\ka}'|=|(\ka,\w_{1\ka})-(\ka,\w_{1\ka}')|\eqsim |(y_{\w_{1}\ka},\w_{1})-(y_{\w_{1}'\ka},\w_{1}')|>c\nu,
\]
due to the fact that $\hat{\chi}_{t}$ is bi-Lipschitz with respect to the standard metric on $\Sp$. 

The previous observation implies that, for a given pair $\w_{1}\sim_{\nu}\w_{2}$, there are at most approximately $\nu^{-\delta(d-1)}$ pairs $\w_{1}'\sim_{\nu}\w_{2}'$ such that $|\w_{1\ka}-\w_{1\ka}'|\leq \nu^{1-\delta}$. As in \cite[Lemma 7.1]{TaoVargas2000}, an argument similar to that used for \eqref{eq:localizationtwo}, involving suitable interpolation between the trivial cases $q=\infty$ and $q=1$, and the case $q=2$ where one can use orthogonality, thus implies that 
\begin{align*}
&\Big\|\sum_{\w_{1}\sim_{\nu}\w_{2}}\Psi_{\w_{1\ka}}(D)\big(\theta^2_{\ka}\psi^c(M_{0}^{-1}D)\big(S_{N}(t)f^{\nu}_{\w_{1}}S_{N}(t)g^{\nu}_{\w_{2}}\big)\big)\Big\|_{L^{q}(\Rd)}\\
&\lesssim \nu^{-\delta(d-1)}\Big(\sum_{\w_{1}\sim_{\nu}\w_{2}}\big\|\Psi_{\w_{1\ka}}(D)\big(\theta^2_{\ka}\psi^c(M_{0}^{-1}D)\big(S_{N}(t)f^{\nu}_{\w_{1}}S_{N}(t)g^{\nu}_{\w_{2}}\big)\big)\big\|_{L^{q}(\Rd)}^{q^{*}}\Big)^{1/q^{*}}\\
&\lesssim \nu^{-\delta(d-1)}\Big(\sum_{\w_{1}\sim_{\nu}\w_{2}}\|\theta^2_{\ka}\psi^c(M_{0}^{-1}D)\big(S_{N}(t)f^{\nu}_{\w_{1}}S_{N}(t)g^{\nu}_{\w_{2}}\big)\|_{L^{q}(\Rd)}^{q^{*}}\Big)^{1/q^{*}}.
\end{align*}
For the last inequality we used \eqref{eq:L1norm}.

Since $q/q^{*}\geq 1$, we can use Minkowski's inequality and again \eqref{eq:localizationtwo}:
\begin{align*}
&\Big(\sum_{\ka\in\nu\Z^{d}}\Big\|\sum_{\w_{1}\sim_{\nu}\w_{2}}\Psi_{\w_{1\ka}}(D)\big(\theta_{\ka}^2 \psi^c(M_{0}^{-1}D)\big(S_{N}(t)f^{\nu}_{\w_{1}}S_{N}(t)g^{\nu}_{\w_{2}}\big)\big)\Big\|_{L^{q}(\Rd)}^{q}\Big)^{\frac{1}{q}}\\
&\lesssim \nu^{-\delta(d-1)}\Big(\sum_{\ka\in\nu\Z^{d}}\Big(\sum_{\w_{1}\sim_{\nu}\w_{2}}\|\theta_{\ka}^2 \psi^c(M_{0}^{-1}D)\big(S_{N}(t)f^{\nu}_{\w_{1}}S_{N}(t)g^{\nu}_{\w_{2}}\big)\|_{L^{q}(\Rd)}^{q^{*}}\Big)^{\frac{q}{q^{*}}}\Big)^{\frac{1}{q}}\\
&\leq \nu^{-\delta(d-1)}\Big(\sum_{\w_{1}\sim_{\nu}\w_{2}}\Big(\sum_{\ka\in\nu\Z^{d}}\|\theta_{\ka}^2 \psi^c(M_{0}^{-1}D) \big(S_{N}(t)f^{\nu}_{\w_{1}}S_{N}(t)g^{\nu}_{\w_{2}}\big)\|_{L^{q}(\Rd)}^{q}\Big)^{\frac{q^{*}}{q}}\Big)^{\frac{1}{q^{*}}}\\
&\eqsim \nu^{-\delta(d-1)}\Big(\sum_{\w_{1}\sim_{\nu}\w_{2}}\|\psi^c(M_{0}^{-1}D) \big(S_{N}(t)f^{\nu}_{\w_{1}}S_{N}(t)g^{\nu}_{\w_{2}}\big)\|_{L^{q}(\Rd)}^{q^{*}}\Big)^{\frac{1}{q^{*}}}\\
&\lesssim \nu^{-\delta(d-1)}\Big(\sum_{\w_{1}\sim_{\nu}\w_{2}}\|S_{N}(t)f^{\nu}_{\w_{1}}S_{N}(t)g^{\nu}_{\w_{2}}\|_{L^{q}(\Rd)}^{q^{*}}\Big)^{\frac{1}{q^{*}}}.
\end{align*}

\subsubsection*{The error term}

Let $\Psi_{\omega_{1 \kappa}}^c = 1-\Psi_{\omega_{1 \kappa}}$. It suffices to prove
\begin{equation}
\label{eq:ErrorTermEstimate}
   \big( \sum_{\kappa \in \nu \Z^d} \| \Psi_{\omega_{1 \kappa}}^{c}(D) (\theta_{\kappa}^2 \psi^c(M_0^{-1}D) (S_N(t) f_{\omega_1}^{\nu} S_N(t) g_{\omega_2}^{\nu} )) \|^q_{L^q} \big)^{1/q} \lesssim N^{-m} \| f_{\omega_1}^{\nu} \|_2 \| g_{\omega_2}^{\nu} \|_2.
\end{equation}
The important observation is that the spatial output and frequency input yields a pointwise localization. This is basically a reprise of Proposition \ref{prop:BilinearSpatialLocalization}. We decompose spatially with functions $\theta_{k_i}^2$ (cf. Appendix \ref{subsection:KernelBoundsPhaseSpace}), which are essentially supported on $\nu^{1-\delta}$-cubes, and have a Fourier support in $[-CN^{1/2},CN^{1/2}]^d$, which does not essentially disrupt the angular frequency localization, see below:
\begin{equation*}
    S_N(t) f_{\omega_1}^{\nu} = \sum_{k_1 \in \nu^{1-\delta} \Z^d} S_N(t) \theta_{k_1}^2 f_{\omega_1}^{\nu}, \quad S_N(t) g_{\omega_2}^{\nu} = \sum_{k_2 \in \nu^{1-\delta} \Z^{d}} S_N(t) \theta_{k_2}^2 g_{\omega_2}^{\nu}.
\end{equation*}
We obtain a corresponding splitting of \eqref{eq:ErrorTermEstimate}:
\begin{equation*}
\begin{split}
   &\quad \big( \sum_{\kappa \in \nu \Z^d} \| \Psi_{\omega_{1 \kappa}}^{c}(D) (\theta_{\kappa}^2 \psi^c(M_0^{-1}D) S_N(t) f_{\omega_1}^{\nu} S_N(t) g_{\omega_2}^{\nu} ) \|^q_{L^q} \big)^{1/q} \\
&\leq \big( \sum_{\kappa \in \nu \Z^d} \| \Psi_{\omega_{1 \kappa}}^{c}(D) \theta_{\kappa}^2 \psi^c(M_0^{-1}D) \sum_{|k_1-k_2| \leq 5} ( S_N(t) f_{k_1,\omega_1}^{\nu} S_N(t) g_{k_2,\omega_2}^{\nu} ) \|^q_{L^q} \big)^{1/q} \\
&\quad + \big( \sum_{\kappa \in \nu \Z^d} \| \Psi_{\omega_{1 \kappa}}^{c}(D) \theta_{\kappa}^2 \psi^c(M_0^{-1}D) \sum_{|k_1-k_2| \geq 5} ( S_N(t) f_{k_1,\omega_1}^{\nu} S_N(t) g_{k_2,\omega_2}^{\nu} ) \|^q_{L^q} \big)^{1/q}.
\end{split}
\end{equation*}
For the second term we reassemble
\begin{equation*}
\begin{split}
    &\quad \big( \sum_{\kappa \in \nu \Z^d} \| \Psi_{\omega_{1 \kappa}}^c(D) \theta_{\kappa}^2 \psi^c(M_0^{-1}D) \sum_{|k_1-k_2| \geq 5} S_N(t) \theta_{k_1}^2 f_{\omega_1}^{\nu} S_N(t) \theta_{k_2}^2 g_{\omega_2}^{\nu} \|_{L^q}^q \big)^{\frac{1}{q}} \\
    &\lesssim \sum_{|k_1-k_2| \geq 5} \| S_N(t) \theta_{k_1} f_{k_1,\omega_1}^{\nu} S_N(t) \theta_{k_2} g_{k_2,\omega_2}^{\nu} \|_{L^q}.
    \end{split}
\end{equation*}
Here we use that the output frequencies are still supported away from the origin.

At this point, crude pointwise estimates suffice (cf. the proof of Proposition \ref{prop:BilinearSpatialLocalization}):
\begin{equation*}
\begin{split}
    &\quad |S_N(t) \psi_{\omega_1}^{\nu^{1-\delta}}(D) \theta_{k_1} \theta_{k_1} f_{\omega_1}^{\nu}(x)| \\
    &\leq C_m N^{2d} (1+ N^{1/2} \text{dist}(x,\pi_x(\hat{\chi}_t(Q_{k_1}^{\nu^{-(1-\delta)}},\Theta_{\omega_1}^{\nu^{1-\delta}}))))^{-m} \| \theta_{k_1} f_{\omega_1}^{\nu} \|_2.
\end{split}
\end{equation*}
For $|k_1 - k_2| \geq 5$ the sum is readily carried out by the bi-Lipschitz property, which proves that the outputs are separted by $\nu^{-(1-\delta)}$. The contribution is smoothing.

\smallskip

We turn to $|k_1-k_2| \leq 5$, i.e. $f_{k_1,\omega_1}^{\nu}$, $g_{k_2,\omega_2}^{\nu}$ essentially localized to the same $\nu^{1-\delta}$-cube. There are finitely many $k^*(\kappa)$ such that $\text{dist}(\pi_x(\hat{\chi}_t(Q_{k^*(\kappa)}^{\nu^{1-\delta}},\Theta_{\omega_1}^{\nu^{1-\delta}})), Q_{\kappa}^{\nu^{1-\delta}}) \lesssim \nu^{1-\delta}$. We pick one of them. For $|k_i-k^*(\kappa)| \geq C$, we find by kernel estimates that the contribution is smoothing. 
Indeed, estimating $\Psi_{\omega_{1 \kappa}}^c$ in $L^q$, using the high frequency constraint, and using again the pointwise kernel estimates, we find this contribution to be smoothing.

In the following for $\kappa$ fixed, we denote by $k_1 \sim_{\kappa} k_2$ the $k_i$ with $|k_i-k^*(\kappa)| \leq C$. Finally, we turn to estimating
\begin{equation}
\label{eq:FullyLocalized}
    \big( \sum_{\kappa \in \nu \Z^d} \| \Psi^c_{\omega_{1\kappa}}(D) \theta_{\kappa}^2 \psi^c(M_0^{-1}D) \sum_{k_1 \sim_{\kappa} k_2} S_N(t) f_{k_1,\omega_1}^{\nu} S_N(t) g_{k_2,\omega_2}^{\nu} \|_{L^q}^q \big)^{1/q}.
\end{equation}
$S_N(t) f_{k_1^*,\omega_1}^{\nu}$ is localized in phase space. In addition to the pointwise spatial estimate, we have by the localization of the Fourier transform and the definition of $\Psi^c_{\omega_{1 \kappa}}$:
\begin{equation*}
|\tilde{\Psi}^c_{\omega_{1 \kappa}}(D) \psi^c(M_0^{-1} D) S_N(t) f_{k_1^*,\omega_1}^{\nu}(x)| \leq C_m N^{-m} \| \theta_{k_1^*}^{1/2} f^{\nu}_{\omega_1} \|_2.
\end{equation*}
Above $\tilde{\Psi}^c_{\omega_{1 \kappa}}$ denotes a mild enlargement of the original support. The kernel estimates show that $S_N(t) f_{k_1^*,\omega_1}^{\nu}$ is localized to the $\nu$-sector around $\omega_{1 \kappa}$ with aperture $N^{-1/2}$:
\begin{equation*}
    | \mathcal{F}[S_N(t) f_{k_1^*,\omega_1}](\xi)| \leq C_m N^{2d} (1 + N^{\frac{1}{2}} |\hat{\xi} - \omega_{1 \kappa}|)^{-m} \| \theta_{k_1} f_{\omega_1} \|_2.
\end{equation*}
Consequently, the only frequency contribution to \eqref{eq:FullyLocalized} is inessential.
We have by Cauchy-Schwarz and $k_1^*(\kappa)$ essentially disjoint as $\kappa \in \nu \Z^d$
\begin{equation*}
\begin{split}
    &\leq C_m N^{-m} \sum_{\kappa} \sum_{k_1^* \sim k_2^*} \| \theta_{k_1^*} f_{\omega_1}^{\nu} \|_2 \| \theta_{k_2^*} g_{\omega_2}^{\nu} \|_2 \\
    &\lesssim C_m N^{-m} \| f_{\omega_1}^{\nu} \|_2 \| g_{\omega_2}^{\nu} \|_2.
\end{split}
\end{equation*}
This shows that this error term is indeed rapidly decaying.

Next, we prove
\begin{equation*}
    \| \psi(M_0^{-1} D) (S_N(t) f_{\omega_1}^{\nu} S_N(t) g_{\omega_2}^{\nu}) \|_{L^q(\R^d)} \leq C_m N^{-m} \| f \|_{\mathcal{H}^p_{FIO}} \| g \|_{\mathcal{H}^p_{FIO}}.
\end{equation*}
Again, we introduce a spatial localization
\begin{equation*}
\begin{split}
    \psi(M_0^{-1} D) (S_N(t) f_{\omega_1}^{\nu} S_N(t) g_{\omega_2}^{\nu}) &= \psi(M_0^{-1} D) \sum_{|k_1-k_2| \leq 5} (S_N(t) f_{k_1,\omega_1}^{\nu} S_N(t) g_{k_2,\omega_2}^{\nu}) \\
    &\quad + \psi(M_0^{-1} D) \sum_{|k_1-k_2| > 5} (S_N(t) f_{k_1,\omega_1}^{\nu} S_N(t) g_{k_2,\omega_2}^{\nu}).
\end{split}
\end{equation*}
The second term can be estimated like above. For the first, we let for a smooth frequency localization
\begin{equation*}
    S_N(t) f_{k_1,\omega_1}^{\nu} = \psi_{\omega_1^t}^N(D) S_N(t) f_{k_1,\omega_1}^{\nu} + (1-\psi^N_{\omega_1^t})(D) S_N(t) f_{k_1,\omega_1}^{\nu}
\end{equation*}
with $\omega_1^t = \pi_{\xi}(\hat{\chi}^t(x_{k_1},\omega_1))$ and
\begin{equation*}
    \text{supp}(\psi^N_{\omega_1^t}(\cdot)) \subseteq \{ \xi \in \R^d : |\xi| \in [N^{1-\delta},N^{1+\delta}], \, \angle(\xi,\omega_1^t) \leq \nu^{1-\delta} \}
\end{equation*}
and the kernel estimates show for any $q \in [1,\infty]$
\begin{equation*}
    \| (1-\psi_{\omega_1^t}^N)(D) S_N(t) f_{k_1,\omega_1}^{\nu} \|_{L^q} \leq C_m N^{-m} \| \theta_{k_1} f_{\omega_1}^{\nu} \|_{L^q}.
\end{equation*}
But the ``main" contribution is vanishing by impossible frequency interaction:
\begin{equation*}
    \psi(M_0^{-1} D) [\psi_{\omega_1^t}^N(D) S_N(t) f_{k_1,\omega_1}^{\nu} \psi^N_{\omega_2^t}(D) S_N(t) f_{k_2,\omega_2}^{\nu}] = 0.
\end{equation*}
The proof is complete.
\end{proof}

\subsection{Conclusion of the proof of Theorem \ref{thm:mainrough}}\label{subsec:conclude}

\begin{figure}[ht!]
\label{figure:InterpolationRestriction}
\begin{tikzpicture}[scale=4]
	
	\def\d{4}
	
	\coordinate (A) at (0,0);
	\coordinate (B) at (0.5,0.5);
	\coordinate (C) at (1,0);
	
	\coordinate (P) at ({(\d-1)/(2*\d)},{(\d-1)/(2*\d)});
	\coordinate (R) at ({(\d-1)*(\d+3)/(2*(\d*\d+2*\d-1))},
	                  {(\d-1)*(\d+1)/(2*(\d*\d+2*\d-1))});
	\coordinate (E) at ({1/2},{(\d+1)/(2*(\d+3))});

	\coordinate (L0) at ({1/2 - 1/(\d-1)},0);
	\coordinate (L1) at ({1/2 - 1/(\d-1)},{1/2 - 1/(\d-1)});
	\coordinate (M0) at ({1/2 + 1/(\d-1)},0);
	\coordinate (M1) at ({1/2 + 1/(\d-1)},{1/2 - 1/(\d-1)});
	
	\coordinate (calP) at ({1/2 - 1/(\d-1)},
	                       {((\d+1)*(\d-3))/(2*(\d-1)*(\d+3))});
	\coordinate (calR) at (M1);

	\coordinate (D0) at (0.5,0);

	\coordinate (N0) at ({(5-\d)/2},0);
	\coordinate (N1) at ({(\d+7)/(2*(\d+2))},{(\d-3)/(2*(\d+2))});

	\coordinate (S0) at ({4/(\d+1)},0);
	\coordinate (S1) at ({(\d+5)/(2*(\d+1))},{(\d-3)/(2*(\d+1))});

	\path[name path=diagRM] (D0) -- (M1);
	\path[name path=RCseg] (R) -- (C);
	\path[name intersections={of=diagRM and RCseg, by=W}];

	\fill[blue!15] (R) -- (B) -- (calR) -- (W) -- cycle;
	\fill[orange!20] (L0) -- (calP) -- (R) -- (W) -- (D0) -- cycle;
	
	\draw[->] (-0.1,0) -- (1.15,0) node[right] {$1/p$};
	\draw[->] (0,-0.1) -- (0,0.65) node[above] {$1/q$};
	
	\draw[thick] (A) -- (B) -- (C) -- cycle;
	\draw[gray,dashed,thick] (1/2,1/2) -- (1/2,0); 
	\draw[red!45, dashed, thick] (A) -- (E);
	
	\draw[green!60!black, dashed, thick] (L0) -- (L1);
	\draw[green!60!black, dashed, thick] (M0) -- (M1);

	\draw[green!60!black, dashed, thick] (N0) -- (N1);
	\draw[green!60!black, dashed, thick] (S0) -- (S1);

	\draw[green!60!black, dashed, thick] (D0) -- (calR);

	\draw[gray, dashed, thick] (P) -- (C);

	\node at (barycentric cs:R=1,B=1,calR=1,W=1) {$\mathcal{T}_u$};
	\node at (barycentric cs:L0=1,calP=1,R=1,W=1,D0=1) {$\mathcal{T}_l$};

	\fill (A) circle (0.4pt) node[below left] {$(0,0)$};
	\fill (B) circle (0.4pt) node[above] {$\left(\frac12,\frac12\right)$};
	
	\fill[gray] (P) circle (0.4pt)
		node[above left] {$\left(\frac{1}{p^*},\frac{1}{p^*}\right)$};
	
	\fill[red!45] (R) circle (0.4pt)
		node[above left] {$\mathcal{Q}$};

    \fill[gray] (E) circle (0.4pt)
        node[above right] {$\mathcal{E}$};

\node[green!60!black, below]
	at (0.6,0) {$\alpha_u$};

\node[green!60!black, below right, xshift=2pt, yshift=-1pt]
	at (N1) {$\alpha_l$};
    
	\fill[green!60!black] (L0) circle (0.4pt)
		node[below] {$\frac12-\frac{r-1}{d-1}$};

    \fill[green!60!black] (M0) circle (0.4pt)
		node[below] {$\frac12+\frac{r+1}{d-1}$};
    
	\fill[green!60!black] (calP) circle (0.4pt)
		node[above left] {$\mathcal{P}$};
	
	\fill[green!60!black] (calR) circle (0.4pt)
		node[above right] {$\mathcal{R}$};

	\draw (0.5,0.015) -- (0.5,-0.015) node[below] {$\frac12$};
	\draw (1,0.015) -- (1,-0.015) node[below] {$1$};
	\draw (0.015,0.5) -- (-0.015,0.5) node[left] {$\frac12$};
	
\end{tikzpicture}
\caption{We illustrate the conclusion of Theorem \ref{thm:mainrough} for $C^{1,1}$-coefficients when $d=4$. First, for $p \geq 2$ in the polygon $(0,0)-\mathfrak{Q}-(1/2,1/2)-(1/2,0)$ we obtain $\mathcal{H}^p_{\text{FIO}}$-$L^q$ estimates for the truncated evolution. Depending on the regularity of the coefficients, these yield $\HT^p_{FIO}$-$L^q$ smoothing estimates for rough wave equations by Proposition \ref{prop:reduction}. The estimates are sharp up to endpoints.}
\end{figure}
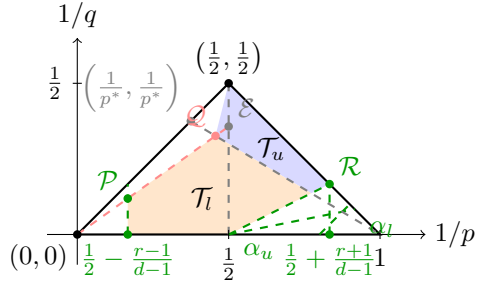

\subsubsection*{Reduction using interpolation}

We will show that, for all $(\frac{1}{p},\frac{1}{q})\in \mathcal{T}_{l}$ close to $\mathfrak{Q}$, and for all $\alpha>\alpha(p,q)-s(p)$ and some $c_{0}>0$, there exists a $C>0$ such that 
\begin{equation}\label{eq:reductionbound2}
\|S_{N}f\|_{L^{q}([0,1]\times\Rd)}\leq CN^{\alpha}\|f\|_{\Hp}
\end{equation}
for $N\in 2^{\N}$ large and for all $f\in\Sw(\Rd)$ satisfying $\supp(\wh{f}\,)\subseteq \{\xi\in\Rn\mid N/8\leq |\xi|\leq 8N, |\hat{\xi}-e_{d}|\leq c_{0}\}$.   Moreover, one may let $\alpha=\alpha(p,q)-s(p)$ if $(\frac{1}{p},\frac{1}{q})$ lies in the interior of $\mathcal{T}_{l}$.

To see that it suffices to prove this claim, first note that \eqref{eq:reductionbound2} immediately extends to all $f\in\Hp$ satisfying $\supp(\wh{f}\,)\subseteq \{\xi\in\Rn\mid N/8\leq |\xi|\leq 8N\}$, by finite decomposition, rotation and density. In particular, one obtains \eqref{eq:reductionbound2} with $S_{N}$ replaced by $S_{N}\sum_{M=N/4}^{4N}\psi_{M}(D)$, for all $f\in\Hp$. Note also that the same bounds hold for $p=q=2$, with $\alpha=0$, and for $q=\infty$, by Proposition \ref{prop:roughfixed} and a standard Sobolev embedding. In fact, in the latter case one may let $\alpha=\alpha(p,\infty)-s(p)=s(p)+\frac{d}{p}$ upon replacing $L^{q}([0,1]\times\Rd)$ by $L^{q}([0,1];\HT^{\infty}(\Rd))$. Finally, for $p=2$ and $q=\frac{2(d+1)}{d-1}$, due to the fact that $\HT^{2}_{FIO}(\Rd)=L^{2}(\Rd)$, \eqref{eq:reductionbound2} is a standard Strichartz estimate, obtained in \cite{Tataru2001} (see also \cite[Appendix A]{SchippaTataru2025}) with $\alpha=\alpha(p,q)-s(p)=\frac{1}{2}$.

 By interpolating these bounds and using that $f=\sum_{M=N/4}^{4N}\psi_{M}(D)f$ for all $f\in\Hp$ with $\supp(\wh{f}\,)\subseteq \{\xi\in\Rd\mid N/4\leq |\xi|\leq 4N\}$, one extends \eqref{eq:reductionbound2} to all $(\frac{1}{p},\frac{1}{q})\in \mathcal{T}_{l}\cup \mathcal{T}_{u}$ with $p\geq 2$, under this slightly stronger assumption on $f$, and one may let $\alpha=\alpha(p,q)-s(p)$ if $(\frac{1}{p},\frac{1}{q})$ lies in the interior of $\mathcal{T}_{l}$.

Next, if $r>\max(2s(p)+1,\alpha(p,q))$, then we can apply Proposition \ref{prop:reduction} to obtain the required bounds but with the solution $u$ from \eqref{eq:defsol} replaced by $\psi_{N}(D)u$, for any $N\in 2^{\N_{0}}$. By summing over $N$, cf.~\eqref{eq:LittlePaley}, this already suffices to prove Theorem \ref{thm:mainrough} with $\alpha>\alpha(p,q)-s(p)$.

To reach the limiting value $\alpha=\alpha(p,q)-s(p)$ for suitable $p$ and $q$, note that the relevant estimate concerns mapping properties of the solution operators $U_{j}=(U_{j}(t))_{t\in\R}$ from Proposition \ref{prop:roughfixed}, for $j\in\{0,1\}$. As such, we can again rely on interpolation. In particular, to reach the endpoint exponent for $p=2$ one can interpolate as before, between the case $q=\frac{2(d+1)}{d-1}$, obtained in \cite{Tataru2001}, and the case $q=\infty$, which follows from Proposition \ref{prop:roughfixed} and a standard Sobolev embedding. 

For $(\frac{1}{p},\frac{1}{q})$ in the interior of $\mathcal{T}_{l}$ and such that $p>2$, from \eqref{eq:reductionbound2} we have already derived the estimate
\begin{equation}\label{eq:dyadicsolution}
\|\psi_{N}(D)U_{j}f\|_{L^{q}([0,1]\times\Rd)}\lesssim \|f\|_{\HT^{\alpha(p,q)-s(p)-j,p}_{FIO}(\Rd)},
\end{equation}
for all $N\in 2^{\N_{0}}$, $f\in\Hp$ and $j\in\{0,1\}$. Now, $(\frac{1}{p},\frac{1}{q})$ lies on a line segment in $\mathcal{T}_{l}$ such that the Sobolev exponent in \eqref{eq:dyadicsolution} is not constant along this line. 
One can then check that the conditions of Lemma \ref{lem:interabstract} are satisfied, and as such
\begin{equation}\label{eq:restrictedweak}
U_{j}\lb D\rb^{-(\alpha(p,q)-s(p)-j)}:(\HT^{p_{0}}_{FIO}(\Rd),\HT^{p_{1}}_{FIO}(\Rd))_{\theta,1}\to L^{q,\infty}([0,1]\times\Rd)
\end{equation}
is bounded for suitable $p_{0},p_{1}\in[1,\infty]$ and $\theta\in(0,1)$ with $\frac{1}{p}=\frac{1-\theta}{p_{0}}+\frac{\theta}{p_{1}}$.

Finally, 
$(\frac{1}{p},\frac{1}{q})$ 
also lies on a line segment along which the Sobolev exponent in \eqref{eq:restrictedweak} is constant. As a result, one can interpolate again, using \eqref{eq:restrictedweak}, the reiteration theorem (see \cite[Theorem 3.5.3]{Bergh-Lofstrom76}) and Lemma \ref{lem:interHpFIO}, to obtain that
\[
U_{j}:\HT^{\alpha(p,q)-s(p)-j,p}_{FIO}(\Rd)\to L^{q,p}([0,1]\times\Rd)
\]
is bounded. Given that $L^{q,p}([0,1]\times\Rd)\subseteq L^{q}([0,1]\times\Rd)$, this suffices.

\subsubsection*{Proof of \eqref{eq:reductionbound2}}

It follows from \eqref{eq:bilinear} that 
\[
\|S_{N}f\|_{L^{q}([0,1]\times\Rd)}^{2}\leq 
\sum_{N^{-1/2}\leq \nu\leq c_{1}}\|B_{\nu}(f,f)\|_{L^{q/2}([0,1]\times\Rd)}
\]
for all $f\in\Sw(\Rd)$ satisfying $\supp(\wh{f}\,)\subseteq \{\xi\in\Rn\mid N/8\leq |\xi|\leq 8N, |\hat{\xi}-e_{d}|\leq c_{0}\}$. Keeping in mind that $q\geq \frac{d+3}{d+1}p$ and $\frac{d+1}{q}\leq (d-1)(1-\frac{1}{p})$ for all $(\frac{1}{p},\frac{1}{q})\in\mathcal{T}_{l}$, with strict inequalities holding in the interior of $\mathcal{T}_{l}$, and that there are approximately $\log(N)$ values $N^{-1/2}\leq \nu\leq c_{1}$, it thus suffices to show the following. 

For all $(\frac{1}{p},\frac{1}{q})$ close to $\mathfrak{Q}$ satisfying $q\geq \frac{d+3}{d+1}p$, and for all $\delta,\veps>0$, one has
\begin{equation}\label{eq:bilinearbound}
\begin{aligned} 
&\|B_{\nu}(f,g)\|_{L^{q/2}([0,1]\times\Rd)}\\
&\lesssim N^{d-1-\frac{2(d+1)}{q}+\frac{2}{p}-2s(p)+\veps}\nu^{2(d-1-\frac{d+1}{q}-\frac{d-1}{p}-\delta)}\|f\|_{\Hp}\|g\|_{\Hp}
\end{aligned}
\end{equation}
for all $f,g\in\Sw(\Rd)$ and $N^{-1/2}\leq \nu\leq c_{1}$. Moreover, one may let $\veps=0$ if $q> \frac{d+3}{d+1}p$. 

To prove \eqref{eq:bilinearbound}, first use Proposition \ref{prop:AlmostOrthogonalityPhaseSpace} to bound $\|B_{\nu}(f,g)\|_{L^{q/2}([0,1]\times\Rd)}$ by $\nu^{-\delta}(\sum_{\w_{1}\sim_{\nu}\w_{2}}\|(S_{N}f^{\nu}_{\w_{1}})(S_{N}g^{\nu}_{\w_{2}})\|_{L^{q/2}([0,1]\times\Rd)}^{(q/2)^{*}})^{1/(q/2)^{*}}$, modulo an error term which decays rapidly in $N$. Given that $\nu\geq 2N^{-1/2}$, this error term is negligible. 

Next, 
Theorem \ref{thm:BilinearEstimates} bounds $\nu^{\delta}(\sum_{\w_{1}\sim_{\nu}\w_{2}}\|(S_{N}f^{\nu}_{\w_{1}})(S_{N}g^{\nu}_{\w_{2}})\|_{L^{q/2}([0,1]\times\Rd)}^{(q/2)^{*}})^{1/(q/2)^{*}}$ by a multiple of
\begin{align*}
N^{d-1-\frac{2(d+1)}{q}+\frac{2}{p}-2s(p)+\veps}\nu^{2(d-1-\frac{d+1}{q}-\frac{d-1}{p})}\!\Big(\!\sum_{\w_{1}\sim_{\nu}\w_{2}}\!\|f^{\nu}_{\w_{1}}\|_{\Hp}^{(q/2)^{*}}\|g^{\nu}_{\w_{2}}\|_{\Hp}^{(q/2)^{*}}\Big)^{\frac{1}{(q/2)^{*}}},
\end{align*}
and one may let $\veps=0$ if $q>\frac{d+3}{d+1}p$. Now Cauchy--Schwarz yields
\begin{align*}
&\Big(\sum_{\w_{1}\sim_{\nu}\w_{2}}\|f^{\nu}_{\w_{1}}\|_{\Hp}^{(q/2)^{*}}\|g^{\nu}_{\w_{2}}\|_{\Hp}^{(q/2)^{*}}\Big)^{1/(q/2)^{*}}\\
&\leq \Big(\sum_{\w_{1}}\|f^{\nu}_{\w_{1}}\|_{\Hp}^{2(q/2)^{*}}\Big)^{1/(2(q/2)^{*})}\Big(\sum_{\w_{2}}\|g^{\nu}_{\w_{2}}\|_{\Hp}^{2(q/2)^{*}}\Big)^{1/(2(q/2)^{*})}\\
&\leq \Big(\sum_{\w_{1}}\|f^{\nu}_{\w_{1}}\|_{\Hp}^{p}\Big)^{1/p}\Big(\sum_{\w_{2}}\|g^{\nu}_{\w_{2}}\|_{\Hp}^{p}\Big)^{1/p}.
\end{align*}
For the last line we used that $2(q/2)^{*}=2\min(q/2,(q/2)')>p$ if $(\frac{1}{p},\frac{1}{q})$ is close to $\mathfrak{Q}$, as follows by noting that this inequality is valid for $(\frac{1}{p},\frac{1}{q})=\mathfrak{Q}$.

Finally, as in the proof of Proposition \ref{prop:bilinearinfty}, one can decompose $f$ into a sum of approximately $N^{(d-1)/2}$ terms, each of which has Fourier support in a cone of angular width $N^{-1/2}$, with fixed finite overlap. This also decomposes each $f^{\nu}_{\w_{1}}$, and one can then apply \cite[Proposition 7.1]{HaPoRoYu26} to obtain $(\sum_{\w_{1}}\|f^{\nu}_{\w_{1}}\|_{\Hp}^{p})^{1/p}\lesssim \|f\|_{\Hp}$. The same statement holds for $g$, which concludes the proof.

\appendix

\section{Kernel bounds}\label{sec:kernel}

In this appendix we collect kernel bounds from \cite{GebaTataru2007} and adapt them to our setting. 

We consider the half-wave equation 
\begin{equation}
    \label{eq:EvolutionEquationGT}
        (D_t + b(x,D)) u= 0, \quad u(0) = u_0,
\end{equation}
for a suitable symbol $b:\R^{2d}\to \R$.  

\subsection{Kernel bounds for operators on phase space}

Geba--Tataru in \cite{GebaTataru2007} showed rapid decay for the kernels of the solution operators to \eqref{eq:EvolutionEquationGT}, after conjugation with wave packet transforms that map to and from phase space, $T^{*}\Rd=\R^{2d}$. Moreover, the decay is measured with respect to a suitable metric on phase space, and it occurs after following the associated Hamiltonian flow. 

It should be noted that more or less the same kernel bounds previously appeared in \cite{Smith98a,Smith1998}, and that they were connected to the Hardy spaces for Fourier integral operators in \cite{Smith98a,HassellPortalRozendaal2020}. We will formulate the bounds as in \cite{HassellPortalRozendaal2020}, given that the Hardy spaces for Fourier integral operators play a key role for us.

As in \cite[Section 2.1]{HassellPortalRozendaal2020} (see also \cite[Section 7]{GebaTataru2007}), let $d_{S}$ be the pseudo-Riemannian metric on the cosphere bundle $\Sp:$ associated with the canonical $1$-form $\hat{\xi}\cdot dx$ on $\Sp$. In fact, the latter notion from contact geometry will play no role for us; 
we will only need the 
following convenient equivalent expression for $d_{S}$:
\begin{equation}\label{eq:metric}
d_{S}((x,\w),(y,\nu))\eqsim \big(|x-y|^{2}+|\w\cdot(x-y)|+|\w-\nu|^{2}\big)^{1/2}
\end{equation}
for all $(x,\w),(y,\nu)\in\Sp$. 

Next, let $h\in C^{\infty}_{c}(\R)$ be real-valued, supported in $B_{1}(0,1/16)$ and such that $\|h\|_{L^{2}(\Rd)}=1$. Set
\[
\phi_{\xi}(\eta) := \langle \eta \rangle^{-\frac{d+1}{4}} h\big(\lb\eta\rb^{-2}|\hat{\eta} \cdot (\xi-\eta)|^2+\lb\eta\rb^{-1}|\Pi_{\eta}^{\perp}(\xi-\eta)|^{2}\big)
\]
for $\xi,\eta\in\Rd$ with $\eta\neq 0$, and $g_{\xi}(0):=h(|\xi|^{2})$. Here $\Pi_{\eta}^{\perp}$ is the orthogonal projection onto the complement of the line $\R\eta$. Then, for every $M\geq0$, one has 
\begin{equation}\label{eq:kernelboundspacket}
|\F^{-1}(\phi_{\xi})(x-z)|\lesssim \lb \xi\rb^{\frac{d+1}{4}}\big(1+\lb\xi\rb d_{S}((x,\hat{\xi}),(z,\hat{\xi}))^{2}\big)^{-M}
\end{equation}
for all $x,z,\xi\in\Rd$, cf.~\cite[p.~ 1074]{GebaTataru2007} and \eqref{eq:metric}. 

For $f\in \Sw'(\Rd)$ and $(x,\xi)$, write
\begin{equation}\label{eq:defW}
Wf(x,\xi):=\phi_{\xi}(D)f(x).
\end{equation}
This is well defined because $\phi_{\xi}\in C^{\infty}_{c}(\Rd)$. Then $W:L^{2}(\Rd)\to L^{2}(\R^{2d})$ is an isometry, and an adjoint of $W$ is given by 
\[
W^{*}F(x):=\int_{\R^{2}}\phi_{\xi}(D)F(\cdot,\xi)(x)\ud \xi,
\]
for $F\in L^{2}(\R^{2d})$ and $x\in\Rd$.

Throughout, fix $b\in \A^{2}S^{1}_{1,1/2}$ real-valued and asymptotically homogeneous, with real-valued limit $a\in C^{1,1}S^{1}_{1,0}$ which is homogeneous of degree $1$ for $|\xi|\geq 1$. Set 
\[
a_{h}(x,\xi):=\lim_{\la\to\infty}\la^{-1}a(x,\la\xi)
\]
for $(x,\xi)\in\Rd$ with $\xi\neq 0$. Then $a_{h}$ is homogeneous of degree $1$ on $T^{*}\Rd$ minus the zero section. Note that $a_{h}$ will typically not lie in $C^{1,1}S^{1}_{1,0}$, but that it does have this regularity away from the zero section.  Under these conditions, the Hamiltonian equations \eqref{eq:flow} are well-posed if one replaces $a$ by $a_{h}$, and they generate a bicharacteristic flow $(\chi_{t})_{t\in\R}$ on $T^{*}\Rd$ minus the zero section, homogeneous of degree $1$. In particular, each $\chi_{t}$ is a Lipschitz symplectomorphism, and it projects down naturally to a map $\hat{\chi}_{t}:\Sp\to \Sp$. Then $\hat{\chi}_{t}$ is bi-Lipschitz both with respect to the standard metric on $\Sp$ and with respect to $d_{S}$, by \cite[Proposition 2.4]{HassellPortalRozendaal2020}.

The following result concerns the group $(S_{t})_{t\in\R}$ of solution operators to \eqref{eq:EvolutionEquationGT}, which are well-defined under the conditions on $b$, as in Lemma \ref{lem:fixedtimesmooth}. For $\sigma>0$, set
\begin{equation*}
    \Upsilon(\sigma) := \min(\sigma,\sigma^{-1}).
\end{equation*}

\begin{proposition}\label{prop:kernelbounds}
Let $t_{0},m>0$. Then there exists a $C_{m}\geq0$ such that $WS_{t}W^{*}$ has an integral kernel $K_{t}:\R^{4d}\to\C$ satisfying 
\[
|K_{t}(x,\xi,y,\eta)|\leq C_{m} \Upsilon \big( \tfrac{\lb\xi\rb}{\lb\eta\rb} \big)^{m}\big(1+\max(| \xi |,|\eta|)d_{S}((x,\hat{\xi}),\hat{\chi}_{t}(y,\hat{\eta}))^{2}\big)^{-m},
\]
for all $t\in[-t_{0},t_{0}]$ and $(x,\xi,y,\eta)\in \R^{4d}$ with $\xi\neq 0\neq \eta$.
\end{proposition}
\begin{proof}
Using that $b\in \A^{2}S^{1}_{1,1/2}$, one can check that $b$ satisfies the conditions of \cite[Theorem 6.2]{GebaTataru2007}. The statement now follows by comparing the distance which appears in that result with $d_{S}$, using \cite[Theorem 2.5]{GebaTataru2007}. Note also that the latter result involves the flow generated by $b$, not the flow generated by $a_{h}$. This makes no difference for the conclusion, as follows by comparing the flows using the arguments in the proofs of \cite[Proposition 7.5]{GebaTataru2007} and \cite[Theorem 8.1]{HassellRozendaal2023}. 
\end{proof}

\begin{remark}\label{rem:kernelboundsconstant}
For fixed $t_{0},M>0$ the constant $C$ in Proposition \ref{prop:kernelbounds} only depends on finitely many of the $\A^{2}S^{1}_{1,1/2}$ seminorms of $b$, and on finitely many of the seminorms implicit in the definition of asymptotic homogeneity, cf.~Definition \ref{def:asymphom}. 
\end{remark}

\begin{remark}\label{rem:otherproof}
Proposition \ref{prop:kernelbounds} also follows from \cite{HassellRozendaal2023}. Indeed, it is shown there (see \cite[Remark 7.2 and Sections 8 and 9]{HassellRozendaal2023}) that one can write $S_{t}f=E_{t}f+\int_{0}^{t}E_{t-\tau}Vf(\tau)\ud \tau$ for all $t\in\R$ and $f\in\Sw(\Rd)$, where $E_{t}$ is a parametrix which satisfies kernel bounds as in Proposition \ref{prop:kernelbounds}, and the error $Vf$ is a convergent series of terms, the kernels of which also satisfy such bounds. 
\end{remark}

\subsection{Kernel estimates for frequency-localized functions}

Next, we apply the bounds from Proposition \ref{prop:kernelbounds} to obtain kernel estimates for frequency-localized functions on $\Rd$.

For $t\in\R$, $N^{-1/2}\leq \nu\leq 1$ and $\w\in S^{d-1}$, we consider the kernel $\tilde{K}^{\nu}_{t,\w}:\R^{2d}\to\C$ of the operator $S_N(t) \psi^{\nu}_{\omega}(D)$, where $\psi^{\nu}_{\w}$ is as in \eqref{eq:AngularLocalization}. That is, 
\begin{equation*}
S_N(t) \psi^{\nu}_{\omega_1}(D) f(x) = \int_{\Rd} \tilde{K}_{t,\w}^{\nu}(x,y) f(y) \ud y
\end{equation*}
for all $f\in\Sw(\Rd)$ and $x\in\Rd$. 
For $\xi\in\Rd$, we also consider the kernel $\tilde{K}^{\nu}_{t,\w,\xi}:\R^{2d}\to\C$ of the operator $\phi_{\xi}(D)WS_{N}(t)\psi^{\nu}_{\w}(D)$. 

For $y\in\Rd$, set $\hat{\mathcal{Y}}_{y,\omega} := \{ (y,\mu)\in\Sp \mid  |\mu- \w| \leq \nu + 2 N^{-1/2} \} $ and $(y(t),\omega(t)) := \hat{\chi}_{t} (y,\omega)$.

\begin{proposition}\label{prop:KernelEstimatesFrequencyLocalization}
Let $t_{0},m>0$. Then there exists a $C\geq0$ such that
\begin{equation*}
\begin{split}
    |\tilde{K}^{\nu}_{t,\w}(x,y)| &\leq CN^{d} \nu^{d-1} \big(1+N \text{dist}_S^2((x,\omega(t)),\hat{\chi}_t(\hat{\mathcal{Y}}_{y,\omega})\big)^{-m},\\
    |\tilde{K}^{\nu}_{t,\w,\xi}(x,y)| &\leq C\Upsilon \big( \tfrac{\lb\xi\rb}{N} \big)^{m}\big(1+ N \text{dist}_S^2((x,\hat{\xi}),\hat{\chi}_t(\hat{\mathcal{Y}}_{y,\omega})\big)^{-m},
\end{split}
\end{equation*}
for all $t\in[-t_{0},t_{0}]$, $N^{-1/2}\leq \nu\leq 1$, $\w\in S^{d-1}$, $\xi\in\Rd\setminus\{0\}$ and $x,y\in\Rd$.
\end{proposition}
\begin{proof}
We define $\hat{\mathcal{Y}}_{y,\omega_1} = \{ (y,\hat{\xi}) : |\hat{\xi} - \omega_1| \leq \nu + 2 N^{-1} \} $. It suffices to estimate
\begin{equation*}
\begin{split}
    |\tilde{K}(x,y)| &\leq \bar{K}(x,y) \leq \int_{\R^{2d}} d\xi dz |\check{\phi}_{\xi}(x,z)| \\
    &\quad \times \int_{\substack{\R^{2d}: |\eta| \sim N, \\ |\hat{\eta} - \omega_1| \leq \nu + 2 N^{-1}}} du d\eta |K(z,\xi,u,\eta)| |\check{\phi}_{\eta}(u-y)|.
\end{split}
\end{equation*}
We first estimate the inner integral by
\begin{equation}
\label{eq:FirstAuxEstimate}
\begin{split}
    &\quad \int_{\substack{\R^{2d}: |\eta| \sim N, \\ |\hat{\eta} - \omega_1| \leq \nu + 2 N^{-1}}} du d\eta |K(z,\xi,u,\eta)| |\check{\phi}_{\eta}(u-y)| \\
    &\lesssim N^{\frac{d+1}{4}} (1 + N \text{dist}_S^2(\hat{\chi}_t^{-1}(z,\hat{\xi}),\hat{\mathcal{Y}}_{y,\omega_1}))^{-m}.
\end{split}
\end{equation}
This is proved by inhomogeneous dyadic decomposition in $N \text{dist}_S^2(\hat{\chi}_t^{-1}(z,\hat{\xi}),\hat{\mathcal{Y}}_{y,\omega_1}))$. In \eqref{eq:FirstAuxEstimate} we use the estimates for $\check{\psi}$ and for the kernel of $W S_N(t) W^*$ to find
\begin{equation*}
\begin{split}
    &\quad \int_{\substack{\R^{2d}: |\eta| \sim N, \\ |\hat{\eta} - \omega_1| \leq \nu + 2 N^{-1}}} du d\eta |K(z,\xi,u,\eta)| |\check{\phi}_{\eta}(u-y)| \\
    &\lesssim \int_{\substack{\R^{2d}: |\eta| \sim N, \\ |\hat{\eta} - \omega_1| \leq \nu + 2 N^{-1}}} du d\eta \Upsilon(|\eta|/|\xi|)^m (1+ \max(N,|\xi|) d_S^2(\hat{\chi}_t^{-1}(z,\hat{\xi}),u,\hat{\eta}))^{-m} \\
    &\quad \times N^{\frac{d+1}{4}} (1+ N \text{dist}_S^2(u,\hat{\eta},\hat{\mathcal{Y}}_{y,\omega_1}))^{-m}.
\end{split}
\end{equation*}
We shall see that the above is dominated by
\begin{equation*}
    N^{\frac{d+1}{4}} \Upsilon(|\xi|/|\eta|)^m (1+\max(|\xi|,N) \text{dist}_S^2(\hat{\chi}_t^{-1}(z,\hat{\xi}), \hat{\mathcal{Y}}_{y,\omega_1}))^{-m},
\end{equation*}
which is seen from dyadic decompositions in $|\xi|$, the bulk of the contribution coming from $|\xi| \sim N$, and the triangle inequality. Indeed, for $|\xi| \sim N$, we consider the cases
\begin{equation*}
    N \text{dist}_S^2(\hat{\chi}_t^{-1}(z,\hat{\xi}), \hat{\mathcal{Y}}_{y,\omega_1})) \leq C, \text{ and } N \text{dist}_S^2(\hat{\chi}_t^{-1}(z,\hat{\xi}), \hat{\mathcal{Y}}_{y,\omega_1})) \sim K, K \geq C.
\end{equation*}
For $N \text{dist}_S^2(\hat{\chi}_t^{-1}(z,\hat{\xi}), \hat{\mathcal{Y}}_{y,\omega_1})) \leq C$ we cannot expect decay from the factors in the integral. Consequently, integrating in $u$ over a rectangle $N^{-\frac{1}{2}} \times \ldots N^{-\frac{1}{2}} \times N^{-1}$ and $\eta$ over a rectangle of size $N^{\frac{1}{2}} \times \ldots \times N^{\frac{1}{2}} \times N$-rectangle yields the claim.
For $N \text{dist}_S^2 (\hat{\chi}_t^{-1}(z,\hat{\xi}), \hat{\mathcal{Y}}^{\nu}_{y,\omega_1}) \sim K \geq C$, one of the factors is rapidly decaying by the triangle inequality.

Next, we estimate
\begin{equation*}
\begin{split}
    &\quad \int_{\R^{2d}} d\xi dz |\check{\phi}_{\xi}(x-z)| N^{\frac{d+1}{4}} \Upsilon(|\xi|/N)^m (1+ (|\xi| \vee N) \text{dist}_S^2(\hat{\chi}_t^{-1}(z,\hat{\xi}),\hat{\mathcal{Y}}_{y,\omega_1}))^{-m} \\
    &\lesssim N^{\frac{d+1}{2}} \int_{\R^{2d}} (1+|\xi| d_S^2(x,\hat{\xi},z,\hat{\xi}))^{-m} \Upsilon(|\xi|/N)^m \\
    &\quad \times (1+ \max(|\xi|,N) \text{dist}_S^2((z,\hat{\xi}),\hat{\chi}_t(\hat{\mathcal{Y}}_{y,\omega_1}))^{-m} d\xi dz.
\end{split}
\end{equation*}
Note that $\hat{\mathcal{Y}}_{y,\omega_1} \subseteq B_{\nu}^{d_S}(y,\omega_1)$, for which reason by the bi-Lipschitz continuity of $\hat{\chi}_t$ we have $\hat{\chi}_t(\hat{\mathcal{Y}}_{y,\omega_1}) \subseteq B_{C \nu}^{d_S}(\hat{\chi}_t(y,\omega_1))$. 

We carry out a dyadic decomposition in $|\xi|$ - the bulk of the contribution coming from $|\xi| \sim N$. And second a dyadic decomposition in $\hat{\xi}$: we consider the cases $|\hat{\xi} - \omega_1^t| \leq C \nu$ and $|\hat{\xi} - \omega_1^t| \sim K$ for $K \geq C \nu$. Finally, we consider a dyadic decomposition for $N \text{dist}_S^2(x,\omega_1^t,\hat{\chi}_t(\hat{\mathcal{Y}}_{y,\omega_1}))$. For $N \text{dist}_S^2(x,\omega_1^t,\hat{\chi}_t(\hat{\mathcal{Y}}_{y,\omega_1})) \leq K$ we cannot expect decay of the factors in the integral. Consequently, integrating in $z$ essentially on a region of size $N^{-\frac{1}{2}} \times \ldots \times N^{-\frac{1}{2}} \times N^{-1}$ and a region in $\xi$ of size $N \times N \nu \times \ldots \times N \nu$ we find the claimed estimate for $|\tilde{K}(x,y)|$ to hold.

Next, we turn to the estimate for $|\tilde{K}_{\xi}(x,y)|$, which is a corollary of the above argument. Following along the above lines, we find
\begin{equation*}
    |\tilde{K}_{\xi}(x,y)| \leq \int_{\R^d} dz |\check{\phi}_{\xi}(x-z)| N^{\frac{d+1}{4}} \Upsilon(|\xi|/N)^m (1+ (|\xi| \vee N) \text{dist}_S^2(\hat{\chi}_t^{-1}(z,\hat{\xi}),\hat{\mathcal{Y}}_{y,\omega_1}))^{-m}.
\end{equation*}
By the pointwise estimate for $|\check{\phi}_{\xi}|$, after a dyadic decomposition in $|\xi|$ it suffices to integrate in $z$. For $|\xi| \sim N$, this gives the main contribution and the estimate
\begin{equation*}
    |\tilde{K}_{\xi}(x,y)| \lesssim (1+ N \text{dist}_S^2((z,\hat{\xi}),\hat{\chi}_t(\hat{\mathcal{Y}}_{y,\omega_1}))^{-m}.\qedhere
\end{equation*}
\end{proof}

\subsection{Kernel estimates for functions localized in phase space}
\label{subsection:KernelBoundsPhaseSpace}
We take the next step and obtain kernel estimates for functions localized in phase space. These will be a consequence of kernel estimates for frequency-localized functions.

We introduce spatially localizing functions, which are positive and have a compactly supported Fourier transform. Let $\theta\in \Sw(\Rd)$ be such that $\supp(\widehat{\theta}\,) \subseteq B_{d}(0,1/2)$ and $\sum_{k \in \Z^d} |\theta(x-k)|^2=1$ for all $x\in\Rd$. The existence of such a function is ensured by the Poisson summation formula.
Indeed, let $\eta \in C^\infty_c(B_{d}(0,1/2))$ be such that $\| \eta \|_{L^2(\Rd)} = 1$, and let $\theta:=\widecheck\eta$. Consider the $1$-periodic function
\begin{equation*}
x\mapsto P(x):=\sum_{k \in \Z^d} |\theta(x-k)|^2
\end{equation*}
on $\R^{d}$. One has $\widehat{P}(k)=0$ for $k\neq0$, and $\widehat{P}(0)=1$. By Poisson summation, $P \equiv 1$. Choosing $\eta$ even, $\theta$ becomes a real-valued function.

For each $m\geq0$, there exists a $C_{m}\geq0$ such that $|\theta(x)| \leq C_m (1+|x|)^{-m}$ for all $x\in\Rd$.  
For $K \in 2^{\N_{0}}$, $\lambda=(\lambda_{1},\ldots,\lambda_{d})\in K^{-1} \Z^{d}$ and $x\in\Rd$, set $\theta_{K,\lambda}(x) := \theta(K(x-\lambda))$. Next, we shall see how an anisotropic decomposition
\begin{equation*}
f = \sum_{k \in \Z^d} \tilde{\theta}^{\nu}_{k,\omega} f
\end{equation*}
transported in phase space. The $\tilde{\theta}_{k,\omega}^{\nu}$ will be positive functions, which are compactly Fourier supported on rectangles of size $N \nu \times \ldots \times N \nu \times N $ with long side pointing into $\omega$-direction. Correspondingly, they are essentially supported on rectangles of size $\nu \times \ldots \times \nu \times \nu^2$ with short side pointing into $\omega$-direction.

We carry out the construction for $\omega = e_d$, upon which the general case is recovered by rotation. Let $\theta \in \mathcal{S}(\R)$ be like above with compact Fourier support in $(-1/8,1/8)$. For $k_i \in (N \nu)^{-1} \Z$, $i=1,\ldots,d-1$, $k_d \in N^{-1} \Z$ we consider
\begin{equation*}
\begin{split}
&\quad \theta^2_{N \nu}(x_1-k_1) \ldots \theta^2_{N \nu}(x_{d-1} - k_{d-1}) \theta^2_N(x_d-k_d) \\
&= \theta^2(N \nu(x_1-k_1)) \ldots \theta^2(N \nu(x_{d-1} - k_{d-1})) \theta^2(N(x_d - k_d)).
\end{split}
\end{equation*}
Consequently, by Fubini we find
\begin{equation*}
\sum_{\substack{ \lambda_1,\ldots, \lambda_{d-1} \in (N \nu)^{-1} \Z, \\ \lambda_d \in N^{-1} \Z}} \theta^2_{N \nu}(x_1-\lambda_1) \ldots \theta^2_{N \nu}(x_{d-1} - \lambda_{d-1}) \theta^2_N(x_d - \lambda_d) \equiv 1.
\end{equation*}

We construct the phase space localization with claimed properties for $k \in \Z^d$ by
\begin{equation*}
\sum_{\substack{ \lambda_1,\ldots,\lambda_{d-1} \in (N \nu)^{-1} \Z, \\ \lambda_i \in (\nu k_i - \nu/2, \nu k_i + \nu/2], \\ \lambda_d \in N^{-1} \Z, \lambda_d \in (\nu^2 k_d - \nu^2/2,\nu^2 k_d + \nu^2/2]}} \theta^2_{N \nu}(x_1 - \lambda_1) \ldots \theta^2_{N \nu}(x_{d-1} - \lambda_{d-1}) \theta_N^2(x_d- \lambda_d) = \tilde{\theta}_{k,e_d}^{\nu}.
\end{equation*}
$\tilde{\theta}_{k,e_d}^{\nu}$ has the claimed properties, i.e., localization to a $\nu \times \ldots \times \nu \times \nu^2$-rectangle centered at $k_{e_d} = (k'\nu, k_d \nu^2) \in \R^d$ with rapid decay off the rectangles on scales $N \nu$ in directions $e_1,\ldots,e_{d-1}$ and $N$ into the direction $e_d$. For arbitrary $\omega \in \mathbb{S}^{d-1}$ we consider a rotation $R \in \R^{d\times d}$, $R^t R = I_d$ such that $R \omega= e_d$ and let 
\begin{equation*}
\tilde{\theta}^{\nu}_{k,\omega}(x) =  \tilde{\theta}^{\nu}_{k,e_d}(R x).
\end{equation*}

In the following proposition we compute the essential phase space localization of $S_N(t) \tilde{\theta}_{k,\omega}^{\nu} f$. First, note that $\tilde{\theta}_k f = \tilde{\psi}^{\nu}_{\omega}(D) \tilde{\theta}_k f$.
We let $\tilde{\psi}^{\nu}_{\omega}(D) \tilde{\theta}_k f = \tilde{\psi}^{\nu}_{\omega}(D) \tilde{\theta}_k^{1/2} g$ for $g \in L^2$. By $\mathcal{Y}_{k,\omega}$ we denote the phase space region $R^{-1} (c(k_\omega + ((-\nu,\nu)^{d-1} \times (-\nu^2,\nu^2)))) \times \Theta_\omega^{c \nu}$ and $\hat{\mathcal{Y}}_{k,\omega}$ the projection to $S^* \R^d$. Denote by $(\bar{x},\bar{\omega})$ the center of $\hat{\mathcal{Y}}_{k,\omega}$ and $\hat{\chi}^t(\bar{x},\bar{\omega}) = (\bar{x}^t,\bar{\omega}^t)$ its evolution under the Hamiltonian flow projected to $S^* \R^d$. 
We consider the kernel 
\begin{equation*}
S_N(t) \tilde{\psi}^{\nu}_{\omega}(D) \tilde{\theta}_k^{1/2} g = \int \tilde{K}(x,y) g(y) dy
\end{equation*}
and show the following:
\begin{proposition}
\label{prop:PhaseSpaceLocalization}
With notations like above, we have
\begin{equation*}
| \tilde{K}(x,y) | \leq C_m N^{d} \nu^{d-1} (1+ N \text{dist}_S^2(x,\omega^t,\hat{\chi}_t(\mathcal{Y}_{k,\omega}) )^{-m}.
\end{equation*}
\end{proposition}
\begin{proof}
We start noting that
\begin{equation*}
\tilde{\theta}_k^{1/4} (y) \leq C_m (1+ N \text{dist}_S^2(y,\omega,\hat{\mathcal{Y}}_{k,\omega}))^{-m}.
\end{equation*}
It appears that there is slack in the directions $\omega_1^{\perp},\ldots,\omega_{d-1}^{\perp}$, but the above expression is indeed invariant under $\nu$-rotations $R$ around $e_d$. Indeed, $\tilde{\theta}_k(R y)$ is again essentially supported in an $\nu \times \ldots \times \nu \times \nu^2$-rectangle with the same decay properties as the Fourier support is essentially preserved a $\nu$-sector around $\omega$ with size of the frequencies comparable to $N$. 

For this reason, we have for any $\hat{\eta} \in \Theta^{\nu}_{\omega}$
\begin{equation*}
\tilde{\theta}_k^{1/4}(y) \leq C_m (1+ N \text{dist}_S^2(y,\hat{\eta},\hat{\mathcal{Y}}_{k,\omega}))^{-m}.
\end{equation*}
We dominate the kernel by
\begin{equation*}
\begin{split}
    |\tilde{K}(x,y)| &\leq \Big( \int_{\R^{2d}} d\xi dz |\check{\phi}_{\xi}(x-z)| \\
    &\quad \times \int_{\substack{\R^{2d}: |\eta| \sim N, \\ |\hat{\eta} - \omega_1| \leq \nu + 2 N^{-1}}} |K(z,\xi,u,\eta)| |\check{\phi}_{\eta}(u-y)| \tilde{\theta}_k^{1/4}(y) \Big) \tilde{\theta}_k^{1/4}(y).
\end{split}
\end{equation*}
We first estimate the inner integral by the pointwise kernel estimate, a pointwise estimate for $\check{\phi}_{\eta}$, and the estimate for $\tilde{\theta}_k$ as:
\begin{equation*}
\begin{split}
    &\quad \int_{\substack{\R^{2d}: |\eta| \sim N, \\ |\hat{\eta} - \omega_1| \leq \nu + 2 N^{-1}}} \Upsilon(|\xi|/N)^m (1+ \max(|\xi|,N) d_S^2(z,\hat{\xi},\hat{\chi}_t(u,\hat{\eta})))^{-m} \\
    &\quad \times N^{\frac{d+1}{4}} (1+ N d_S^2(u,\hat{\eta},y,\hat{\eta}))^{-m} (1+ N \text{dist}_S^2(y,\hat{\eta},\hat{\mathcal{Y}}_{k,\omega}))^{-m} du d\eta \\
    &\lesssim \Upsilon(|\xi|/N)^m N^{\frac{d+1}{4}} (1+ N \text{dist}_S^2(\hat{\chi}_t^{-1}(z,\hat{\xi}),\hat{\mathcal{Y}}_{k,\omega}))^{-m}.
    \end{split}
\end{equation*}
Following along the proof of Proposition \ref{prop:KernelEstimatesFrequencyLocalization} we estimate the second integral together with the estimate for $\tilde{\theta}_k$ as
\begin{equation*}
    |\tilde{K}(x,y)| \leq C N^{d} \nu^{d-1} (1+ N \text{dist}_S^2(x,\omega_1^t,\hat{\chi}_t(\hat{\mathcal{Y}}_{k,\omega_1}))^{-m} (1+ N \text{dist}_S^2(y,\omega_1, \hat{\mathcal{Y}}_{k,\omega_1}))^{-m}.
\end{equation*}
\end{proof}

\section{Interpolation results}\label{sec:inter}

This appendix contains two interpolation lemmas. We use notation and terminology from the theory of real interpolation spaces, as can be found in \cite{Bergh-Lofstrom76}. 

The first lemma is a multilinear extension of an abstract interpolation result from \cite[Section 6.2]{CaSeWaWr99}. In a more concrete setting, which will not suffice for our purposes, the multilinear statement can already be found in \cite[Lemma 2.6]{Lee03}, and the linear case goes back to \cite{Bourgain85}. 

\begin{lemma}\label{lem:interabstract}
Let $n\in\N$ and $\gamma_{0}<0<\gamma_{1}$. Let 
$(Y_{0},Y_{1})$ and 
$(X^{j}_{0},X^{j}_{1})$, for $j\in\{1,\ldots, n\}$, be compatible Banach couples. 
Then there exists a $C\geq0$ such that the following holds.  Let $M_{0},M_{1}>0$ and, for each $k\in \Z$, let $T_{k}:\prod_{j=1}^{n}(X_{0}^{j}\cap X_{1}^{j})\to Y_{0}+Y_{1}$ be $n$-linear and such that 
\begin{equation}\label{eq:interassump}
\|T_{k}x\|_{Y_{l}}\leq M_{l}2^{k\gamma_{l}}\prod_{j=1}^{n}\|x_{j}\|_{X_{l}^{j}},
\end{equation}
for all $x=(x_{1},\ldots,x_{n})\in \prod_{j=1}^{n}(X_{0}^{j}\cap X_{1}^{j})$ and $l\in\{0,1\}$. Set $\theta:=\frac{\gamma_{0}}{\gamma_{0}-\gamma_{1}}\in(0,1)$. Then $T:=\sum_{k\in\Z}T_{k}$ extends uniquely to an $n$-linear map $T:\prod_{j=1}^{n}(X_{0}^{j},X_{1}^{j})_{\theta,1}\to (Y_{0},Y_{1})_{\theta,\infty}$ satisfying
\begin{equation}\label{eq:interconc}
\|Tx\|_{(Y_{0},Y_{1})_{\theta,\infty}}\leq CM_{0}^{1-\theta}M_{1}^{\theta}\prod_{j=1}^{n}\|x_{j}\|_{(X_{0}^{j},X_{1}^{j})_{\theta,1}},
\end{equation}
for all $x=(x_{1},\ldots,x_{n})\in \prod_{j=1}^{n}(X_{0}^{j},X_{1}^{j})_{\theta,1}$.
\end{lemma}
\begin{proof}

Since $X_{0}^{j}\cap X_{1}^{j}$ is dense in $(X_{0}^{j},X_{1}^{j})_{\theta,1}$ for all $j\in\{1,\ldots,n\}$, it suffices to obtain \eqref{eq:interconc} for $x\in \prod_{j=1}^{n}(X_{0}^{j}\cap X_{1}^{j})$. For $t>0$ and for a given $N\in\Z$, we can write $Tx=\sum_{k=N+1}^{\infty}T_{k}x+\sum_{k=-\infty}^{N}T_{k}x$ and use \eqref{eq:interassump} to bound the $K$-functional 
\[
K(t,Tx):=\inf\{\|y_{0}\|_{Y_{0}}+t\|y_{1}\|_{Y_{1}}\mid y_{0}\in Y_{0},y_{1}\in Y_{1}, Tx=y_{0}+y_{1}\}
\]
from above by a multiple of $M_{0}2^{N\gamma_{0}}\prod_{j=1}^{n}\|x_{j}\|_{X^{j}_{0}}+tM_{1}2^{N\gamma_{1}}\prod_{j=1}^{n}\|x_{j}\|_{X^{j}_{1}}$. Choosing $N$ which minimizes this quantity yields
\begin{equation}\label{eq:interineq1}
K(t,Tx)\lesssim M_{0}^{1-\theta}M_{1}^{\theta}\prod_{j=1}^{n}\|x_{j}\|_{X^{j}_{0}}^{1-\theta}\|x_{j}\|_{X^{j}_{1}}^{\theta},
\end{equation}
for an implicit constant which depends only on $\gamma_{0}$ and $\gamma_{1}$.

Now, by \cite[Lemma 3.2.3 and Theorem 3.3.1]{Bergh-Lofstrom76}, for each $j\in\{1,\ldots,n\}$ there exists a sequence $(x_{jm})_{m\in\Z}\subseteq X_{0}^{j}\cap X^{j}_{1}$ such that $x_{j}=\sum_{m\in\Z}x_{jm}$ and $\sum_{m\in\Z}\la_{jm}\lesssim \|x_{j}\|_{(X_{0}^{j},X_{1}^{j})_{\theta,1}}$, where $\lambda_{jm}:=2^{-\theta m}\max(\|x_{jm}\|_{X^{j}_{0}},2^{m}\|x_{jm}\|_{X_{1}^{j}})$ for $m\in\Z$. Then 
\begin{equation}\label{eq:interineq2}
\sum_{m\in\Z}\|x_{jm}\|_{X^{j}_{0}}^{1-\theta}\|x_{jm}\|_{X^{j}_{1}}^{\theta}\leq \sum_{m\in\Z}\la_{jm}\lesssim \|x_{j}\|_{(X_{0}^{j},X_{1}^{j})_{\theta,1}}.
\end{equation}
Finally, use the $n$-linearity of $T$ to write
\begin{align*}
\|Tx\|_{(Y_{0},Y_{1})_{\theta,\infty}}&\leq \sum_{m_{1},\ldots,m_{n}\in\Z}\|T(x_{1m_{1}},\ldots,x_{nm_{n}})\|_{\overline{Y}_{\theta,\infty}}\\
&=\sum_{m_{1},\ldots,m_{n}\in\Z}\sup_{t>0}t^{-\theta}K(t,T(x_{1m_{1}},\ldots,x_{nm_{n}}))\\
&\lesssim M_{0}^{1-\theta}M_{1}^{\theta}\sum_{m_{1},\ldots,m_{n}\in\Z}\prod_{j=1}^{n}\|x_{jm_{j}}\|_{X^{j}_{0}}^{1-\theta}\|x_{jm_{j}}\|_{X^{j}_{1}}^{\theta}\\
&= M_{0}^{1-\theta}M_{1}^{\theta}\prod_{j=1}^{n}\sum_{m_{j}\in\Z}\|x_{jm_{j}}\|_{X^{j}_{0}}^{1-\theta}\|x_{jm_{j}}\|_{X^{j}_{1}}^{\theta}\\
&\lesssim M_{0}^{1-\theta}M_{1}^{\theta}\prod_{j=1}^{n}\|x_{j}\|_{(X_{0}^{j},X_{1}^{j})_{\theta,1}},
\end{align*}
where we also applied \eqref{eq:interineq1} with $x$ replaced by $(x_{1m_{1}},\ldots,x_{nm_{n}})$, and \eqref{eq:interineq2}. 
\end{proof}

\begin{remark}\label{rem:sublinear}
If $Y_{0}=L^{q_{0}}(\Omega)$ and $Y_{1}=L^{q_{1}}(\Omega)$ for some measure space $\Omega$ and $q_{0},q_{1}\in[1,\infty]$, then the conclusion of Lemma \ref{lem:interabstract} holds under the weaker assumption that $T$ is sublinear in every coordinate, as follows from the proof. 
\end{remark}

Our next lemma shows that the Hardy spaces for Fourier integral operators also behave well under real interpolation.

\begin{lemma}\label{lem:interHpFIO}
Let $p,p_{0},p_{1}\in(1,\infty)$, $p_0 \neq p_1$ and $\theta\in(0,1)$ be such that $\frac{1}{p}=\frac{1-\theta}{p_{0}}+\frac{\theta}{p_{1}}$, and let $s\in\R$. Then $(\HT^{s,p_{0}}_{FIO}(\Rd),\HT^{s,p_{1}}_{FIO}(\Rd))_{\theta,p}=\Hps$.
\end{lemma}
\begin{proof}
It suffices to consider the case where $s=0$. By \cite[Theorem 3.4 and Proposition 4.3]{HaPoRoYu26}, $\Hp$ is isomorphic to a complemented subspace of $L^{p}(\R^{d}\times S^{d-1};L^{2}(0,\infty))$, and similarly for $\HT^{p_{0}}_{FIO}(\Rd)$ and $\HT^{p_{1}}_{FIO}(\Rd)$. Moreover, the relevant projection is the same in each case. Hence the required statement follows from the standard result on interpolation of complemented subspaces (see \cite[Theorem 1.17.1.1]{Triebel78}), and from the identity
\[
(L^{p_{0}}(\R^{d}\times S^{d-1};L^{2}(0,\infty)),L^{p_{1}}(\R^{d}\times S^{d-1};L^{2}(0,\infty)))_{\theta,p}=L^{p}(\R^{d}\times S^{d-1};L^{2}(0,\infty)),
\]
contained in e.g.~\cite[Theorem 5.2.1]{Bergh-Lofstrom76}.
\end{proof}
\section{An $\varepsilon$-removal lemma}\label{sec:noeps}

In this appendix we prove an $L^{2}$-based bilinear estimate with sharp loss of derivatives away from the endpoint, in effect already stated as part of Corollary \ref{cor:bilinearL2main}. We use localization arguments as in \cite{SchippaTataru2025}, and then we remove the $\varepsilon$ loss by an interpolation argument, originating in the work \cite{Tao1999} of Tao.

Throughout, $N\in 2^{\N}$ is large and fixed, 
and notation is as in Section \ref{subsec:bilinear}.

\begin{proposition}\label{prop:SharpBilinearEstimate}
Let $\frac{2(d+3)}{d+1}<q\leq \infty$ and $\delta>0$. 
Then there exists a $C\geq0$ such that the following holds. Let $N^{-1/2}\leq \nu\leq 1$ and let $\Theta_{1},\Theta_{2}\subseteq \Theta_{0}$ be spherical caps of angular width $c_{2}\nu$, distance at least $c_{1}\nu$ apart. Let $f_{1},f_{2}\in\Sw(\Rd)$ be such that  
\[
\supp(\wh{f_{j}})\subseteq \{ \xi \in \R^d \mid c_{3}^{-1}\leq |\xi|\leq c_{3}, \hat{\xi}\in\Theta_{j}\}   
\]
for $j\in\{1,2\}$. Then 
\[
\| \tilde{S}_N  f_1 \tilde{S}_N f_2 \|_{L^{q/2}([0,N] \times \R^d)} \leq C\nu^{d-1-\frac{2(d+1)}{q}-\delta} \prod_{j=1}^2 \| f_j \|_{L^{2}(\Rd)}.
\]
\end{proposition}

The rest of this appendix is dedicated to the proof of Proposition \ref{prop:SharpBilinearEstimate}.

By mild dilation and rotation, we suppose in the following
\begin{equation*}
\text{supp}(\widehat{f}_{i}) \subseteq \{ \xi \in \R^d \mid \xi_d \in (1/2,3/2), \, \big| \frac{\xi'}{\xi_d} - \xi_{\alpha_i} \big| \leq \nu/32 \}
\end{equation*}
and $|\xi_{\alpha_1} - \xi_{\alpha_2}| \in [ \nu/2,2\nu] $.

\textbf{(0) Tiny transversality $\nu \in [N^{-\frac{1}{2}}, N^{-\frac{1}{2}+\delta/C}]$.}
In this case we decompose $f_i$ in frequency into sectors $\theta_i$ of aperture $N^{-\frac{1}{2}}$, which only incurs a loss of $\nu^{-\frac{\delta}{2}}$ choosing $C$ large enough. Then it suffices to prove
\begin{equation*}
\| \tilde{S}_N(t) f_{1 \theta_1} \tilde{S}_N(t) f_{2 \theta_2} \|_{L^{q/2}_{t,x}([0,N] \times \R^d)} \lesssim N^{\frac{d+1}{2q}-\frac{d-1}{2}+\delta^2} \prod_{i=1}^2 \| f_i \|_2.
\end{equation*}
This is immediate from the fixed-time estimate proved in Proposition \ref{prop:roughfixed}, after rescaling:
\begin{equation*}
\sup_{t \in [0,1]} \| S_N(t) f'_{i \theta_i} \|_{L^{q}_{x}(\R^d)} \lesssim N^{\frac{d+1}{2} \big( \frac{1}{2} -\frac{1}{q}\big)} \| f_i \|_2.
\end{equation*}
In the following we suppose that $\nu \geq N^{-\frac{1}{2}+\delta}$.

\smallskip

\textbf{(1) Almost orthogonal decomposition in space:}
First, by Proposition \ref{prop:BilinearSpatialLocalization}, it suffices to show the spatially localized estimate
\begin{equation*}
    \| \tilde{S}_N(t) f_1 \tilde{S}_N(t) f_2 \|_{L^{q/2}_{t,x}([0,N] \times B_d(0,CN))} \leq C \nu^{d-1-\frac{2(d+1)}{q}} \prod_{j=1}^2 \| f_j \|_{L^2},
\end{equation*}
where we suppress the spatial localization of $f_i$ to lighten the notation.

Next, we use the kernel estimates from Section \ref{subsection:KernelBoundsPhaseSpace} to localize  $f_i$ to the same (or neighborhing) rectangle of size $N \nu \times N \nu \times \ldots \times N \nu^2$, with short side pointing into $e_d$-direction. We write
\begin{equation*}
    f_i = \sum_{k \in \Z^d} \tilde{\theta}_k^{\nu} f_i
\end{equation*}
with $f_{i k} = \tilde{\theta}_k^{\nu} f_i$ smoothly localized to an $N \nu \times \ldots N \nu \times N \nu^2$ neighborhood of $N \nu (k_1 e_1 + \ldots + k_{d-1} e_{d-1}) + N \nu^2 k_d e_d$.

We have the decomposition
\begin{align*}
&\big\| \tilde{S}_N(t) f_1 \tilde{S}_N(t) f_2 \big\|_{L^{q/2}_{t,x}([0,N] \times B_d(0,CN))} \\
&\lesssim \sum_{|k_1 - k_2| \leq 5} \big\| \tilde{S}_N(t) f_{1 k_1} \tilde{S}_N(t) f_{2 k_2} \big\|_{L^{q/2}_{t,x}([0,N] \times B_d(0,CN))} + C_m N^{-m} \| f_1 \|_2 \| f_2 \|_2.
\end{align*}
\begin{remark}
In principle, this decomposition is still possible for $\nu \to N^{-\frac{1}{2}}$, but the kernel bounds have to be refined to prove its veracity without $\nu^{-\delta}$-factor.
\end{remark}

\textbf{(2) Recentering and rescaling:}

Next, we carry out the recentering and rescaling argument from \cite[Section~5.3]{SchippaTataru2025}. The idea is to reduce to the unit transversality case via an anisotropic rescaling. In the variable-coefficient case, we first have to center around a central bicharacteristic ray. Here the Hamiltonian is not $1$-homogeneous, but asymptotically homogeneous. This does not cause a serious issue though.

For the $1$-homogeneous flow we changed to angular variables, which is recalled presently. Let $\omega(t)$ be the ratio $\omega(t) = \xi'(t) / \xi_d(t)$. Here $\bar{\omega}(t) = \frac{\bar{\xi}'(t)}{\bar{\xi}_d(t)}$ denotes the ratio of the $(\xi',\xi_d)$-component of the central bicharacteristic emanating from $(\bar{x}(0),0,1)$. $\bar{x}(0)$ is determined by the spatial localization to an anisotropic rectangle.
Recall the decomposition
\begin{equation*}
a_N(x,\xi) = a_N^h(x,\xi) + r_N(x,\xi)
\end{equation*}
with $a_N^h$ $1$-homogeneous for $|\xi| \sim 1$ and favorable estimates for $r_N$.

We have the following connection between the Hamiltonian evolution of $b_N^h$ in Cartesian coordinates $(x(t),\xi(t))$ and angular frequencies $(x(t),\omega(t),\xi_d(t))$:
\begin{lemma}
\label{lem:AnisotropicBiLipschitz}
The evolution equations for $(x^t, \omega^t, \xi_d^t)$ read:
\begin{equation}
\label{eq:ModifiedHamiltonian}
\left\{ \begin{array}{cl}
\dot{x}^t &= \partial_{\xi} a_N^h(x^t, \omega^t, 1), \\
\dot{\omega}^t &= - \partial_{x'} a_N^h(x^t,\omega^t,1) + \partial_{x_d} a_N^h(x^t,\omega^t,1) \omega^t, \\
\partial_t \log(\xi_d^t) &= - \partial_{x_d} a_N^h(x^t,\omega^t,1).
\end{array} \right.
\end{equation}
In particular $(x^t,\omega^t)$ are governed by an ODE:
\begin{equation*}
\frac{d}{dt} (x^t, \omega^t) = F(x^t,t,\omega^t)
\end{equation*}
with $F \in C_t C^{0,1}_{x,\omega}$, and we have the bi-Lipschitz dependence: For $(x^0,\omega^0)$, $(\bar{x}^0,\bar{\omega}^0)$ with 
\begin{equation*}
\{ |x^0 - \bar{x}^0| \leq \nu N, \quad |\omega^0 - \bar{\omega}^0| \leq \nu \},
\end{equation*}
it holds
\begin{equation*}
\{ |x^t - \bar{x}^t| \lesssim \nu N, \quad |\omega^t - \bar{\omega}^t| \lesssim \nu \}.
\end{equation*}
\end{lemma}
The evolution of the central bicharacteristic of the essential part of the Hamiltonian $a_N^h$ determines the ray centered at $\{ (\bar{x}(t),\bar{\omega}(t) \xi_d, \xi_d ) : \xi_d \in (1/2,3/2) \}$ around which we are centering in phase space, like in \cite[Section~5.3.2]{SchippaTataru2025}. By a linear transformation in phase space $\mathcal{S}$ the solutions $v_i(t) = \tilde{S}_N(t) f_{i k_i}$ are transformed to solutions $v_i(t)$ with governing Hamiltonian $\mathcal{S} a_N^h + \mathcal{S} r_N$ (cf. \cite[Lemma~5.12,~Lemma~5.13]{SchippaTataru2025}). Now we can carry out the anisotropic rescaling $\mathcal{T}: x' \to \nu^{-1} x', \, x_d \to x_d, \; \xi' \to \nu \xi', \, \xi_d \to \xi_d$ and let $u_i(t) = w_i(\nu^{-1} x', x_d, \nu^{-2} t)$ such that $u_i(t)$ become solutions to the evolution equation governed by the Hamiltonian
\begin{equation*}
\tilde{a}_{N \nu^2} = \nu^{-2}( \mathcal{T} \mathcal{S} a_N^h + \mathcal{T} \mathcal{S} r_N),
\end{equation*}
for which the precise form of $\mathcal{T} \mathcal{S} a_N^h$ was recorded in \cite[Proposition~5.15]{SchippaTataru2025}.
The precise form is not important for the following analysis. It matters that the functions are now at unit transversality and the size and regularity of the transformed Hamiltonian $\nu^{-2} \mathcal{T} \mathcal{S} a_N$. In \cite{SchippaTataru2025} was proved that the estimate
\begin{equation*}
\| w_1 w_2 \|_{L^{\bar{q}/2}_{t,x}([0,N \nu^2] \times B_d(0,N \nu^2))} \lesssim_\varepsilon (N \nu^2)^{\varepsilon} \prod_{i=1}^2 \| w_i \|_2
\end{equation*}
remains valid for $\bar{q} = \frac{2(d+3)}{d+1}$. More precisely, this was proved for solutions governed by $\nu^{-2} \mathcal{T} \mathcal{S} a^h_N$ at unit transversality. The bounds for the remainder term $r_N$ suffice to argue that $\nu^{-2} \mathcal{T} \mathcal{S} r_N$ can still be regarded as a perturbation to $\nu^{-2} \mathcal{T} \mathcal{S} b^h_N$, as pointed out in Section 3.5. We abuse notation and denote the evolution governed by $\tilde{a}_{N \nu^2}$ with $\tilde{S}_{N \nu^2}$.

\smallskip

\textbf{(3) An $\varepsilon$-removed estimate:} It will suffice to show
\begin{equation}
\label{eq:epsRemoved}
\| w_1 w_2 \|_{L^{q/2}_{t,x}([0,N \nu^2] \times B_d(0,N \nu^2))} \lesssim \prod_{i=1}^2 \| w_i \|_2
\end{equation}
for $q > \bar{q}$, as then the claim will follow from reversing the scaling and the almost orthogonality proved in \textbf{(1)}. The following notion of sparse sets introduced in \cite{Tao1998,Tao1999} will be crucial for the argument.

\begin{definition}[\cite{Tao1998,Tao1999}]
Let $R \geq 1$. A family of $R$-balls $\{B(x_j,R)\}_{j=1}^M$ in $\R^n$ is sparse if the centres are $(RM)^{\bar{C}}$-separated. Here $\bar{C} \geq 1$ is a fixed constant, which can be chosen as $\bar{C} = 12$.
\end{definition}

To make the interpolation argument more transparent, we pause to elaborate on the constants in use. We will choose
\begin{equation*}
\varepsilon = \varepsilon(q), \quad M = \bar{C}^{-1} \log(1/\varepsilon), \quad \bar{C} = 12
\end{equation*}
such that
\begin{equation*}
\sum_{k \geq 1} 2^{-k (1-\bar{C} q' / \log(1/\varepsilon) - q'/\bar{q}')} \lesssim 1.
\end{equation*}

In the proof we use an improved estimate for bilinear expressions on sparse sets. Crucially, we improve the constant from $(N \nu^2) ^\varepsilon$ to $R^{\varepsilon}$. In the following we let $f_i$ be two functions at unit frequencies with angularly separated Fourier support. 
\begin{proposition}
\label{prop:SparseImprovement}
Let $E \subseteq B_{d+1}(0,N) \subseteq \R^{d+1}$ be a sparse set on scale $R \geq 1$. It holds for $M \leq R^M$:
\begin{equation}
\label{eq:SparseEstimateI}
\| \tilde{S}_{N \nu^2} (t) f_1 \tilde{S}_{N \nu^2}(t) f_2 \|_{L^{\bar{q}}_{t,x}(E)} \lesssim_{\varepsilon} R^\varepsilon  \prod_{i=1}^2 \| f_i \|_2.
\end{equation}
For $1 \leq R \leq C$ we have
\begin{equation}
\label{eq:SparseEstimateII}
    \| \tilde{S}_{N \nu^2}(t) f_1 \tilde{S}_{N \nu^2}(t) f_2 \|_{L^{\bar{q}}_{t,x}(E)} \lesssim_C M \prod_{i=1}^2 \| f_i \|_2.
\end{equation}
\end{proposition}
The proof hinges on almost orthogonality of phase space projections associated with the sparse family of balls. This in turn is verified using the dispersive estimate and finite speed of propagation. Before we turn to the proof, we argue how to conclude the argument by invoking a covering lemma to reduce general sets to a controlled union of sparse sets. The following goes back to Tao \cite{Tao1998,Tao1999}:
\begin{lemma}
Let $E \subseteq \R^n$ be a finite union of $1$-cubes and $M \geq 1$. Define the radii $R_j$ inductively by
\begin{equation*}
R_0 := 1, \quad R_j := R_{j-1}^{\bar{C}} |E|^{\bar{C}} \text{ for } 1 \leq j \leq M-1.
\end{equation*}
Then for each $0 \leq j \leq M-1$ there exists a family of sparse collections $(\mathcal{B}_{j,\alpha})_{\alpha \in A_j}$ of balls of radius $R_j$ such that the index sets $A_k$ have cardinality $\mathcal{O}(|E|^{\frac{1}{M}})$ and
\begin{equation*}
E \subseteq \bigcup_{j=0}^{M-1} \bigcup_{\alpha \in A_j} S_{j,\alpha},
\end{equation*}
where $S_{j,\alpha}$ is the union of all the balls belonging to the family $\mathcal{B}_{j,\alpha}$.
\end{lemma}

\begin{proof}[Proof~of~(Proposition \ref{prop:SharpBilinearEstimate}~$\Leftarrow$~Proposition \ref{prop:SparseImprovement})]
    
Let $\chi_{k}(x,t)$ denote a smooth localization to a unit ball in space-time. First, we observe that the space frequencies of $\chi_k(x,t) \tilde{S}_{N \nu^2}(t) f_i$ are localized to the unit ball, which follows from the isotropic wave packet decomposition. For this reason the function is essentially constant on the unit scale spatially. And moreover, the first time derivative satisfies
\begin{equation*}
\frac{d}{dt} (\tilde{S}_{N \nu^2}(t) f_i) = i \tilde{a}_{N \nu^2}(x,D) f_i(x)
\end{equation*}
and by the spatial frequency localization, we find that $\chi_k(x,t) S'_{N \nu^2}(t) f_i$ is locally constant. We use duality to write
\begin{equation*}
\| \tilde{S}_{N \nu^2}(t) f_1 \tilde{S}_{N \nu^2}(t) f_2 \|_{L^q_{t,x}(B_{d+1}(0,N))} = \sup_{\| g \|_{L^{q'}_{t,x}} = 1} \iint \tilde{S}_{N \nu^2}(t) f_1 \tilde{S}_{N \nu^2}(t) f_2 \, g dx dt.
\end{equation*}
We note that $g$ is locally constant on the unit scale too. We carry out a level set decomposition $g = \sum_{k \in \Z} g_k$ with $g_k = g \cdot \chi_{E_k}$ for
\begin{equation*}
E_k = \{ x \in \R^{d+1} : 2^{-k} \leq |g(x)| < 2^{-k+1} \}.
\end{equation*}
$\chi_{E_k}$ denotes the indicator function. By the locally constant property, we can suppose $E_k$ to be a union of $1$-cubes. (This can be made more precise by comparing to the scattered modulation sum; see e.g. \cite[Section~2.4]{Lee2018}.) And clearly, $E_k = \emptyset$ for $k < 0$. We obtain
\begin{equation*}
\iint \tilde{S}_{N \nu^2}(t) f_1 \tilde{S}_{N \nu^2}(t) f_2 g dx dt = \sum_{k \in \Z} \iint \tilde{S}_{N \nu^2}(t) f_1 \tilde{S}_{N \nu^2}(t) f_2 g_k dx dt.
\end{equation*}
We estimate
\begin{equation*}
\iint \tilde{S}_{N \nu^2}(t) f_1 \tilde{S}_{N \nu^2}(t) f_2 g_k dx dt \leq \| \tilde{S}_{N \nu^2}(t) f_1 \tilde{S}_{N \nu^2}(t) f_2 \|_{L^{\bar{q}}_{t,x}(E_k)} \| g_k \|_{L^{\bar{q}'}(E_k)}.
\end{equation*}
Suppose that $|E_k| \geq 2$, in which case we invoke the covering lemma for $E_k$ to find
\begin{equation*}
\| \tilde{S}_{N \nu^2}(t) f_1 \tilde{S}_{N \nu^2}(t) f_2 \|_{L^{\bar{q}}_{t,x}(E_k)} \lesssim_\varepsilon M |E_k|^{\frac{1}{M}+\varepsilon \bar{C}^M} \| f_1 \|_2 \| f_2 \|_2.
\end{equation*}
The contributions of the sparse sets with $R \lesssim 1$ and $R \geq |E_k|^{\bar{C}}$ (i.e. $R^M \geq M$) are estimated separately and can be summed up to the above.

Moreover, for $|E_k| < 2$, which implies $|E_k| \in \{ 1, 0 \}$, the above estimate follows by easier means. Choosing $M = \bar{C}^{-1} \log(1/\varepsilon)$ we find
\begin{equation*}
\| \tilde{S}_{N \nu^2}(t) f_1 \tilde{S}_{N \nu^2}(t) f_2 \|_{L^{\bar{q}}_{t,x}(E_k)} \lesssim_\varepsilon |E_k|^{\bar{C}/\log(1/\varepsilon)} \| f_1 \|_2 \| f_2 \|_2.
\end{equation*}
By Tschebychev's inequality we find $\| g_k \|_{L^{\bar{q}'}} \sim 2^{-k} |E_k|^{\frac{1}{\bar{q}'}}$. Furthermore,
\begin{equation*}
|E_k|^{\bar{C}/\log(1/\varepsilon) + 1/ \bar{q}'} 2^{-k} \leq 2^{-k ( 1- \bar{C} q'/\log(1/\varepsilon) - q'/\bar{q}')},
\end{equation*}
which is summable, since $q > \bar{q}$ and choosing $\varepsilon$ small enough.
\end{proof} 

It remains to show Proposition \ref{prop:SparseImprovement}. Then the proof of Proposition \ref{prop:SharpBilinearEstimate} will be complete.

\begin{proof}[Proof~of~Proposition~\ref{prop:SparseImprovement}]
The estimate for $1 \leq R \leq 2$ follows from decomposing $E$ and applying H\"older's and Bernstein's inequality
\begin{equation*}
\| \tilde{S}_{N \nu^2}(t) f_1 \tilde{S}_{N \nu^2}(t) f_2 \|_{L^{\bar{q}}_{t,x}(E)} \leq \sum_{k=1}^M \| \tilde{S}_{N \nu^2}(t) f_1 \tilde{S}_{N \nu^2}(t) f_2 \|_{L^{\bar{q}}_{t,x}(B_R^k)} \lesssim M \prod_{i=1}^2 \| f_i \|_2.
\end{equation*}

Next, we consider the case $R^K \geq M$.
We will estimate the sparse balls separately. The spatial projection of the space-time ball $B_R^k$ is denoted by $B_{x,R}^k$ and the center of the time projection by $t_k$.
In the first step, we argue that for a ball $B_R^k \subseteq \R^{d+1}$ it holds
\begin{equation}
\label{eq:SparseEstimateLocalized}
\begin{split}
    \| \tilde{S}_{N \nu^2}(t) f_1 \tilde{S}_{N \nu^2}(t) f_2 \|_{L^{\bar{q}}_{t,x}(B^k_R)} &\lesssim_{\varepsilon} R^{\varepsilon}  \prod_{i=1}^2 \| \chi_{B_{x,R}^k} \tilde{S}_{N \nu^2}(t_k) f_i \|_{L^2} \\
    &\quad + C_m R^{-m} \prod_{i=1}^2 \| f_i \|_2.
\end{split}
\end{equation}
Above $\chi_{B_{x,R}^k}$ denotes a smooth and compactly supported function, which is identically one on $C B_{x,R}^k$ and vanishing off $2 C B_{x,R}^k$.

The key input is to apply the bilinear estimate \cite{SchippaTataru2025} 
on the essential contribution of $f_i$ in phase space, which, by kernel estimates is like above.

In the second step we show the almost orthogonality of the phase space contributions for sparse balls with notation introduced below:
\begin{equation}
\label{eq:AlmostOrthogonalityPhaseSpace}
    \sum_{k=1}^M \| \chi_{B_{x,R}^k} \tilde{S}_{N \nu^2}(t_k) f \|_{L^2}^2 \lesssim \| f \|_2^2.
\end{equation}
The estimate in the previous display follows from a $TT^*$-argument. We consider
\begin{equation*}
    T: L^2(\R^d) \to \ell^2_M L^2(\R^d), \quad f \mapsto ( \chi_{R_i} \tilde{S}_{N \nu^2}(t_i) f )_{i=1,\ldots,M}
\end{equation*}
with adjoint
\begin{equation*}
    T^*: \ell^2_M L^2(\R^d) \to L^2(\R^d), \quad (g_i)_{i=1,\ldots,M} \mapsto \sum_{i=1}^M \tilde{S}_{N \nu^2}(-t_i) \big( \chi_{R_i} g_i \big).
\end{equation*}
Squaring the adjoint gives
\begin{equation*}
\begin{split}
    \| T^* g \|_{L^2(\R^d)}^2 &= \sum_{i,j=1}^M \langle \tilde{S}_{N \nu^2}(t_i) \big( \chi_{B_{x,R}^i} g_i \big), \tilde{S}_{N \nu^2}(t_j) \big( \chi_{B_{x,R}^j} g_j \big) \rangle \\
    &= \sum_{i=1}^M \| \tilde{S}_{N \nu^2}(t_i) \chi_{B_{x,R}^i} g_i \|_{L^2}^2 + \sum_{i \neq j} \langle \tilde{S}_{N \nu^2}(t_i) \big( \chi_{B_{x,R}^i} g_i \big), \tilde{S}_{N \nu^2}(t_j) \big( \chi_{B_{x,R}^j} g_j \big) \rangle. 
\end{split}
\end{equation*}
We can estimate the diagonal contribution trivially. We need favorable bounds for the off-diagonal part: For $|t_i - t_j| \geq (RM)^{\bar{C}}$ by the dispersive estimate
\begin{equation*}
    \begin{split}
    \langle \tilde{S}_{N \nu^2}(t_i) (\chi_{B_{x,R}^i} g_i), \tilde{S}_{N \nu^2}(t_j) ( \chi_{B_{x,R}^j} g_j) &\leq \| \chi_{B_{x,R}^i} g_i \|_{L^1} \| \tilde{S}_{N \nu^2}(t_j-t_i) \chi_{B_{x,R}^j} g_j \|_{L^{\infty}} \\
    &\lesssim (MR)^{\frac{d+1}{2}} (RM)^{\bar{C}( - \frac{d-1}{2}) } (MR)^{\frac{d+1}{2}} \| g_i \|_2 \| g_j \|_2.
    \end{split}
\end{equation*}
On the other hand, for $|t_i-t_j| \leq c (RM)^{\bar{C}}$ for a small $c \in (0,1)$ and $|x_i-x_j| \gtrsim (RM)^{\bar{C}}$, we have by finite speed of propagation for the off-diagonal terms even more favorable bounds. Notably, the number of terms is restricted to $M^2$. The proof is complete.
\end{proof}

\section{The smooth setting}\label{sec:smooth}

In this appendix we show that the results from Section \ref{sec:BilinearSmoothing} can be improved slightly for the Euclidean wave equation.

To this end, we first obtain a version of Theorem \ref{thm:BilinearEstimates} without any $\veps$ loss. By the Sobolev embeddings in \eqref{eq:Sobolev}, this result strictly improves upon known $L^{p}$-based bilinear restriction estimates (see e.g.~\cite[Theorem 2.2]{ChLeLi26}). 

\begin{proposition}\label{prop:BilinearEstimatesEuclidean}
Let $2\leq p\leq \infty$, $\frac{d+3}{d+1}p\leq q\leq \infty$ and $c_{1},c_{2},c_{3}>0$. 
Then there exists a $C\geq0$ such that the following holds. Let $N\in 2^{\N}$ and $N^{-1/2}\leq \nu\leq 1$, and let $\Theta_{1},\Theta_{2}\subseteq S^{d-1}$ be spherical caps of angular width $c_{2}\nu$, distance at least $c_{1}\nu$ apart if $\nu\geq 2N^{-1/2}$. 
Let
$f_{1},f_{2}\in\Sw(\Rd)$ be such that 
\[
\supp(\wh{f_{j}})\subseteq \{ \xi \in \R^d \mid c_{3}^{-1}N\leq |\xi|\leq c_{3}N, \hat{\xi}\in\Theta_{j}\} 
\]
for $j\in\{1,2\}$. Then 
\[
\begin{aligned}
&\big\|e^{it\sqrt{-\Delta}}f_{1}(x)e^{it\sqrt{-\Delta}}f_{2}(x)\big\|_{L_{tx}^{q/2}([0,1] \times \R^d)}\\
&\leq CN^{d-1-\frac{2(d+1)}{q}+\frac{2}{p}-2s(p)} \nu^{2(d-1-\frac{d+1}{q}-\frac{d-1}{p})} \prod_{j=1}^2 \| f_{j} \|_{\Hp}.
\end{aligned}
\]
\end{proposition}
\begin{proof}
By reasoning as in the proof of Theorem \ref{thm:BilinearEstimates}, it suffices to consider $(\frac{1}{p},\frac{1}{q})\in\{(\frac{1}{2},\frac{d+1}{2(d+3)}), (\frac{1}{2},0),(0,0)\}$.  

For $(\frac{1}{p},\frac{1}{q})\in\{(\frac{1}{2},\frac{d+1}{2(d+3)}), (\frac{1}{2},0)\}$, due to the fact that $\HT^{2}_{FIO}(\Rn)=L^{2}(\Rn)$, the result is well known (see e.g.~\cite[Theorem 2.2]{ChLeLi26}). Most relevantly, for $q=\frac{2(d+3)}{d+1}$ the conclusion follows from the endpoint statement without $\veps$ loss from \cite{Tao01} after rescaling, combined with Bernstein's inequality for $N^{-1/2}\leq\nu<2N^{-1/2}$.  

On the other hand, the argument for $p=q=\infty$ is just as in the proof of Theorem \ref{thm:BilinearEstimates}. One decomposes $f_{j}$ into a sum $f_{j}=\sum_{\w} f_{j\w}$ of approximately $(\nu N^{1/2})^{d-1}$ terms $f_{j\w}$, each of which satisfies the support condition in \eqref{eq:bilinearinftyassume} and the estimate $\|f_{j\w}\|_{\HT^{\infty}_{FIO}(\Rd)}\lesssim \|f_{j}\|_{\HT^{\infty}_{FIO}(\Rd)}$. One can then use standard kernel bounds (cf.~also Proposition \ref{prop:bilinearinfty}):
\begin{align*}
\big\|e^{it\sqrt{-\Delta}}f_{1}(x)e^{it\sqrt{-\Delta}}f_{2}(x)\big\|_{L_{tx}^{\infty}([0,1] \times \R^d)}&\leq 
\prod_{j=1}^{2}\sum_{\w}\|e^{it\sqrt{-\Delta}}f_{j\w}(x)\|_{L_{tx}^{\infty}([0,1] \times \R^d)}.\\
&\eqsim \prod_{j=1}^{2}\sum_{\w}\|f_{j\w}\|_{L^{\infty}(\R^d)}.
\end{align*}
Finally, Lemma \ref{lem:Knapp} implies that \begin{equation}\label{eq:Knappbilinear}
\begin{aligned}
\sum_{\w}\|f_{j\w}\|_{L^{\infty}(\R^d)}&\eqsim N^{-s(\infty)}\sum_{\w}\|f_{j\w}\|_{\HT^{\infty}_{FIO}(\Rd)}\\
&\lesssim N^{\frac{d-1}{2}-s(\infty)}\nu^{d-1}\|f_{j}\|_{\HT^{\infty}_{FIO}(\Rd)}
\end{aligned}
\end{equation}
for $j\in\{1,2\}$, as required.
\end{proof}

We can now improve upon Theorem \ref{thm:mainrough} in the case of the Euclidean wave equation, by enlarging the region on which the endpoint exponent is attained.

\begin{theorem}\label{thm:Euclidean}
Let $2\leq p<q\leq \infty$ be such that $(\frac{1}{p},\frac{1}{q})\in \mathcal{T}_{l}\cap \mathcal{T}_{u}$, and let $\alpha\geq \alpha(p,q)-s(p)$. Suppose that $r>\max(2s(p)+1,\alpha+s(p))$, and that one of the following conditions holds:
\begin{enumerate}
\item\label{it:Euclidean1} $\alpha>\alpha(p,q)-s(p)$; 
\item\label{it:Euclidean2} $(\frac{1}{p},\frac{1}{q})$ lies in the interior of $\mathcal{T}_{l}$ or in the interior of $\mathcal{T}_{u}$. 
\end{enumerate} 
Then there exists a $C\geq0$ such that
\begin{equation}\label{eq:Euclideanbound}
\Big(\int_{0}^{1}\|e^{it\sqrt{-\Delta}}f\|_{L^{q}(\Rd)}\ud t\Big)^{1/q}\leq C\|f\|_{\HT^{\alpha,p}_{FIO}(\R^d)}
\end{equation}
for all $f\in \HT^{\alpha,p}_{FIO}(\R^d)$.
\end{theorem}
\begin{proof}
The argument mostly follows \cite{TaoVargas2000}. In fact, the proof is a simpler version of the arguments in Section \ref{sec:BilinearSmoothing}, so we mainly indicate where modifications are required. 

\subsubsection*{Sharp bounds}

First consider $(\frac{1}{p},\frac{1}{q})\in \mathcal{T}_{l}\cup \mathcal{T}_{u}$ close to $\mathfrak{Q}$ and such that $q\geq \frac{2(d+3)}{d+1}p$. Here we will obtain estimates using the Whitney decomposition from Section \ref{subsec:orthogonal}. As such, for a suitable $c_{1}>0$ and for $N\in 2^{\N}$ and $N^{-1/2}\leq \nu\leq c_{1}$ with $N\geq 16$, let $B_{\nu}$ be as in \eqref{eq:Bnu} but with $S_{N}(t)f^{\nu}_{\w_{1}}(x)$ replaced by $e^{it\sqrt{-\Delta}}f^{\nu}_{\w_{1}}(x)$, for $(t,x)\in\R^{1+d}$.  Then, by \cite[Lemma 7.1]{TaoVargas2000} and Proposition \ref{prop:BilinearEstimatesEuclidean}, 
\begin{align*}
&\|B_{\nu}(f,g)\|_{L^{q/2}([0,1]\times\Rd)}\lesssim \Big(\sum_{\w_{1}\sim_{\nu}\w_{2}}\|e^{it\sqrt{-\Delta}}f^{\nu}_{\w_{1}}(x)e^{it\sqrt{-\Delta}}g^{\nu}_{\w_{2}}(x)\|_{L^{q/2}_{tx}([0,1]\times\Rd)}^{q^{*}}\Big)^{1/q^{*}}\\
&\lesssim N^{d-1-\frac{2(d+1)}{q}+\frac{2}{p}-2s(p)} \nu^{2(d-1-\frac{d+1}{q}-\frac{d-1}{p})}\Big(\sum_{\w_{1}\sim_{\nu}\w_{2}}\|f^{\nu}_{\w_{1}}\|_{\Hp}^{q^{*}}\|g^{\nu}_{\w_{2}}\|_{\Hp}^{q^{*}}\Big)^{1/q^{*}}
\end{align*}
for all $N^{-1/2}\leq \nu\leq c_{1}$ and $f,g\in\Sw(\Rd)$, where $q^{*}=\min(q/2,(q/2)')$. Moreover, since $2q^{*}>p$ for $(\frac{1}{p},\frac{1}{q})$ close to $\mathfrak{Q}$, we can use Cauchy--Schwarz and the embedding $\ell^{p}\subseteq\ell^{2q^{*}}$:
\begin{align*}
&\Big(\sum_{\w_{1}\sim_{\nu}\w_{2}}\|f^{\nu}_{\w_{1}}\|_{\Hp}^{q^{*}}\|g^{\nu}_{\w_{2}}\|_{\Hp}^{q^{*}}\Big)^{1/q^{*}}\\
&\leq \Big(\sum_{\w_{1}}\|f^{\nu}_{\w_{1}}\|_{\Hp}^{p}\Big)^{1/p}\Big(\sum_{\w_{2}}\|g^{\nu}_{\w_{2}}\|_{\Hp}^{p}\Big)^{1/p}.
\end{align*}
Finally, as before, one can decompose $f$ into a sum of approximately $N^{(d-1)/2}$ terms, each of which has Fourier support in a cone of angular width $N^{-1/2}$, with finite overlap. This also decomposes each $f^{\nu}_{\w_{1}}$, and then \cite[Proposition 7.1]{HaPoRoYu26} yields $(\sum_{\w_{1}}\|f^{\nu}_{\w_{1}}\|_{\Hp}^{p})^{1/p}\lesssim \|f\|_{\Hp}$. The same holds for $g$, and we obtain
\begin{equation}\label{eq:BnuboundEuclidean}
\begin{aligned}
&\|B_{\nu}(f,g)\|_{L^{q/2}([0,1]\times\Rd)}\\
&\lesssim N^{d-1-\frac{2(d+1)}{q}+\frac{2}{p}-2s(p)} \nu^{2(d-1-\frac{d+1}{q}-\frac{d-1}{p})}\|f\|_{\Hp}\|g\|_{\Hp}.
\end{aligned}
\end{equation}

\subsubsection*{Summation and interpolation}

Next, consider $(\frac{1}{p},\frac{1}{q})$ near $\mathfrak{Q}$ on the closed line between $\mathfrak{Q}$ and $(0,0)$. 
Choose distinct points $(\frac{1}{p_{0}},\frac{1}{q_{0}})\in \mathcal{T}_{l}$ and $(\frac{1}{p_{1}},\frac{1}{q_{1}})\in \mathcal{T}_{u}$, both close to $\mathfrak{Q}$ and such that $q_{j}=\frac{d+3}{d+1}p_{j}$ for $j\in\{0,1\}$, with the property that $\frac{1}{p}=\frac{1-\theta}{p_{0}}+\frac{\theta}{p_{1}}$ and $\frac{1}{q}=\frac{1-\theta}{q_{0}}+\frac{\theta}{q_{1}}$ for some $\theta\in(0,1)$. Write $X_{\theta,1}:=(\HT^{p_{0}}_{FIO}(\Rd),\HT^{p_{0}}_{FIO}(\Rd))_{\theta,1}$. Then \eqref{eq:BnuboundEuclidean} and Lemma \ref{lem:interabstract} imply that the bilinear operator
\begin{equation}\label{eq:mappingBnusum}
\sum_{2N^{-1/2}\leq \nu\leq c_{1}}B_{\nu}:X_{\theta,1}\times X_{\theta,1}\to L^{q/2,\infty}([0,1]\times\Rd)
\end{equation}
is bounded, with norm controlled by $N^{d-1-\frac{2(d+1)}{q}+\frac{2}{p}-2s(p)}=N^{2(\alpha(p,q)-s(p))}$. 

Recall from \eqref{eq:bilinear} that 
$(e^{it\sqrt{-\Delta}}f(x))^{2}=\sum_{N^{-1/2}\leq \nu\leq c_{1}}B_{\nu}(f,f)(t,x)$ 
for all $f\in \Sw(\Rd)$ with $\supp(\wh{f}\,)\subseteq \{\xi\in\Rn\mid N/8\leq |\xi|\leq 8N, |\hat{\xi}-e_{d}|\leq c_{0}\}$, for some $c_{0}>0$. Hence \eqref{eq:mappingBnusum}, finite decomposition, rotation and density imply that
\begin{equation}\label{eq:restrictedatQ}
\|e^{it\sqrt{-\Delta}}\psi_{N}(D)f(x)\|_{L^{q,\infty}_{tx}([0,1]\times\Rd)}\lesssim N^{\alpha(p,q)-s(p)}\|f\|_{X_{\theta,1}}
\end{equation}
for all $f\in X_{\theta,1}$.

Now, using the reiteration theorem (see \cite[Theorem 3.5.3]{Bergh-Lofstrom76}) and Lemma \ref{lem:interHpFIO}, one can interpolate \eqref{eq:restrictedatQ} with the strong bound
\begin{equation}\label{eq:strongbounds}
\|e^{it\sqrt{-\Delta}}\psi_{N}(D)f(x)\|_{L^{q}_{tx}([0,1]\times\Rd)}\lesssim N^{\alpha(p,q)-s(p)}\|f\|_{\Hp},
\end{equation}
for $(\frac{1}{p},\frac{1}{q})\in\{(\frac{1}{2},\frac{1}{2}),(\frac{1}{2},\frac{d-1}{2(d+1)}),(\frac{1}{2},0)\}$. Since $\HT^{2}_{FIO}(\Rd)=L^{2}(\Rd)$, these correspond to a standard energy estimate, a Strichartz estimate and a standard Sobolev embedding. As a result, one extends \eqref{eq:strongbounds} to all of $\mathcal{T}_{u}\setminus \{\mathfrak{Q}\}$, and to much of $\mathcal{T}_{l}\setminus \{\mathfrak{Q}\}$.

It only remains to obtain \eqref{eq:strongbounds} on the line between $\mathfrak{Q}$ and $(0,0)$, after which one can interpolate again to extend \eqref{eq:strongbounds} to all of $\mathcal{T}_{l}\setminus\{\mathfrak{Q}\}$. Given that Lemma \ref{lem:interHpFIO} does not cover the case where one of the parameters equals $\infty$, here we argue slightly differently. Since \eqref{eq:restrictedatQ} holds for all $(\frac{1}{p},\frac{1}{q})$ close to $\mathfrak{Q}$ on the line between $\mathfrak{Q}$ and $(0,0)$, the reiteration theorem and Lemma \ref{lem:interHpFIO} immediately upgrade the estimate to \eqref{eq:strongbounds} for $(\frac{1}{p},\frac{1}{q})\neq \mathfrak{Q}$. One can then use complex interpolation with the case $p=q=\infty$ of \eqref{eq:strongbounds}, which is a consequence of \cite[Corollary 6.11]{HassellPortalRozendaal2020} and \eqref{eq:Sobolev}.

Having obtained \eqref{eq:strongbounds} on all of $(\mathcal{T}_{l}\cup\mathcal{T}_{u})\setminus\{\mathfrak{Q}\}$, we can sum the bounds using \eqref{eq:LittlePaley}. To this end, first note that \eqref{eq:strongbounds} also holds for $1\leq N\leq 8$, by \eqref{eq:Sobolev}. Now, each $(\frac{1}{p},\frac{1}{q})\in (\mathcal{T}_{l}\cup\mathcal{T}_{u})\setminus\{\mathfrak{Q}\}$ lies on a line segment in $(\mathcal{T}_{l}\cup\mathcal{T}_{u})\setminus \{\mathfrak{Q}\}$ along which the Sobolev exponent is not constant. For $(\frac{1}{p},\frac{1}{q})\in \mathcal{T}_{l}\cap \mathcal{T}_{u}$ one chooses the line segment to be part of $\mathcal{T}_{l}\cap \mathcal{T}_{u}$. Then Lemma \ref{lem:interabstract} yields the bound
\begin{equation}\label{eq:restrictedfull}
\|e^{it\sqrt{-\Delta}}\lb D\rb^{-\alpha(p,q)}f(x)\|_{L^{q,\infty}_{tx}([0,1]\times\Rd)}\lesssim \|f\|_{X_{\theta,1}}.
\end{equation}
Finally, if $(\frac{1}{p},\frac{1}{q})$ lies in the interior of $\mathcal{T}_{l}$ or in the interior of $\mathcal{T}_{u}$, then it also lies on a line segment along which the Sobolev exponent is constant. Then one can use reiteration and Lemma \ref{lem:interHpFIO} to upgrade \eqref{eq:restrictedfull} to the required strong estimate. 
\end{proof}

\begin{remark}\label{rem:dyadicestimatesEuclidean}
We note explicitly 
that, for $(\frac{1}{p},\frac{1}{q})$ on the Schlag--Sogge line, i.e.~on $\mathcal{T}_{l}\cap \mathcal{T}_{u}$, on the line between $\mathfrak{Q}$ and $(0,0)$, and on the line between between $\mathfrak{Q}$ and $(\frac{1}{2},\frac{1}{2})$, 
the conclusion of Theorem \ref{thm:mainrough} still holds with $\alpha=\alpha(p,q)$ if one assumes that the initial data has frequency support in a dyadic annulus. 
As a consequence, one obtains a restricted weak-type inequality at these points. 
\end{remark}

\begin{remark}\label{rem:ell2decoupling}
In \eqref{eq:Euclideanbound}, one cannot replace $\HT^{\alpha,p}_{FIO}(\Rd)$ by the space $\mathcal{L}^{2,p}_{W,\alpha}(\Rd)$ from \cite{HaPoRoYu26} or by its Besov analogue from \cite{RozendaalSchippa2023}, unless $q\geq \frac{2(d+1)}{d-1}$. This can be seen by choosing $f$ to be an anisotropic Knapp-type example, in which case one has $\|f\|_{\mathcal{L}^{2,p}_{W,\alpha}(\Rd)}\eqsim \|f\|_{\HT^{\alpha,p}(\Rd)}$, cf.~\cite[Proposition 7.1]{HaPoRoYu26}. In fact, one cannot replace $\HT^{\alpha,p}_{FIO}(\Rd)$ by $\mathcal{L}^{2,u}_{W,\alpha}(\Rd)$ for any $1\leq u<p$, and similarly in the Besov space setting.

It then follows from the proof of Theorem \ref{thm:Euclidean} that one also cannot in general improve the bilinear estimates in Proposition \ref{prop:BilinearEstimatesEuclidean} using the spaces from \cite{HaPoRoYu26} or \cite{RozendaalSchippa2023}. It is precisely the behavior of the Knapp example which is crucial here, cf.~\eqref{eq:Knappbilinear}.

This observation is relevant because the $\mathcal{L}^{u,p}_{W,\alpha}(\Rd)$ norm of a function $f$ is, when restricted to dyadic frequency annuli, equivalent to the $\ell^{u}L^{p}$ decoupling norm of $t\mapsto e^{it\sqrt{-\Delta}}f$. For $q\geq \frac{2(d+1)}{d-1}$ and $\alpha>\alpha(p,q)-s(p)$, \eqref{eq:Euclideanbound} follows from interpolation between the $\ell^{2}$ decoupling theorem in \cite{BourgainDemeter2015} and other known bounds. However, for $2\leq p<q<\frac{2(d+1)}{d-1}$ this does not appear to be the case.
\end{remark}

\begin{remark}\label{rem:FIOestimates}
Theorem \ref{thm:mainrough} extends to Fourier integral operators satisfying the cinematic curvature condition; this class strictly contains the solution operators to smooth wave equations. More precisely, for $m\in\R$, let $T$ be a Fourier integral operator of order $m-1/4$ associated with a canonical relation from $T^{*}\Rd$ to $T^{*}\R^{d+1}$ satisfying the cinematic curvature condition from \cite{Sogge1991}, and suppose that $T$ has a compact Schwartz kernel. Then, for $(\frac{1}{p},\frac{1}{q})\in \mathcal{T}_{l}\cup\mathcal{T}_{u}$ and $\alpha>m+\alpha(p,q)-s(p)$, 
\begin{equation}\label{eq:FIObound}
\|Tf\|_{L^{q}(\R^{d+1})}\lesssim \|f\|_{\HT^{\alpha,p}_{FIO}(\Rd)}
\end{equation}
for all $f\in \HT^{\alpha,p}_{FIO}(\Rd)$. The proof follows along the lines of the corresponding result with initial data in $W^{\alpha+s(p),p}(\Rn)$ from \cite[Corollary 1.5]{Lee2006}, albeit with modifications  as in the present article.

We also note that \eqref{eq:FIObound} is, apart from possibly the endpoint value, sharp for any Fourier integral operator which is non-characteristic somewhere, and in particular for the solutions operator to any smooth wave equation. This follows from the arguments in \cite[Section 7]{LiRoSoYa24}, which also imply that an estimate such as \eqref{eq:FIObound} cannot hold for any smooth wave equation if $\alpha<0=m$, e.g.~not near the point $(\frac{n-1}{2n},\frac{n-1}{2n})$. Moreover, as already indicated in \cite{Lee2006}, it seems likely that the endpoint exponent $\alpha=m+\alpha(p,q)-s(p)$ can be attained in \eqref{eq:FIObound} for suitable $(\frac{1}{p},\frac{1}{q})$, by obtaining a bilinear restriction theorem without $\veps$ loss using arguments as in Appendix \ref{sec:noeps}. Finally, \eqref{eq:FIObound} can be extended to Fourier integral operators between Riemannian manifolds $M$ and $N$ with bounded geometry, using the space $\HT^{\alpha,p}_{FIO}(M)$ from \cite{LiRoSoYa24}.
\end{remark}

\section*{Acknowledgments}
J. Rozendaal is partially supported by the National Science Center, Poland, grants 2021/43/D/ST1/00667 and 2023/49/B/ST1/01961. R. Schippa gratefully acknowledges support by the Humboldt foundation via a Feodor-Lynen return fellowship.

\end{document}